\documentclass{amsart}
\usepackage{subcaption}
\usepackage{amsmath}
\usepackage{graphicx}
\usepackage{textcomp}
\usepackage{amsbsy}
\usepackage{esint}
\usepackage{latexsym}
\usepackage[mathscr]{eucal}
\usepackage{amsfonts,amsmath,amsthm}
\usepackage{amssymb}
\usepackage{diagbox}
\usepackage{hhline}
\usepackage{xcolor}
\usepackage{tikz}
\usepackage{todonotes}
\usepackage[shortlabels]{enumitem}
\usepackage{ulem}

\usepackage{comment}
\usepackage{fullpage}

\def\fa{\forall}
\newtheorem{lemma}{Lemma}[section]
\newtheorem{algorithm}[lemma]{Algorithm}

\newtheorem{corollary}[lemma]{Corollary}
\newtheorem{theorem}[lemma]{Theorem}
\newtheorem{remark}[lemma]{Remark}

\newtheorem{assumption}[lemma]{Assumption}
\newtheorem{proposition}[lemma]{Proposition}

{\end{tabbing}\end{samepage}}

\def\RR{\rm \hbox{I\kern-.2em\hbox{R}}}
\def\NN{\rm \hbox{I\kern-.2em\hbox{N}}}
\def\ZZ{\rm {{\rm Z}\kern-.28em{\rm Z}}}
\def\CC{\rm \hbox{C\kern -.5em {\raise .32ex \hbox{$\scriptscriptstyle
|$}}\kern
-.22em{\raise .6ex \hbox{$\scriptscriptstyle |$}}\kern .4em}}

\def\<{\langle}
\def\>{\rangle}

\def\e{\varepsilon}

\def\de{\delta}

\def\pa{\partial}
\def\sig{\sigma}
\def\lang{\langle}
\def\rang{\rangle}
\def\grad{\nabla}
\def\lap{\bigtriangleup}

\def\Chi{\raise .3ex
\hbox{\large $\chi$}} 
\def\lsima{\hbox{\kern -.6em\raisebox{-1ex}{$~\stackrel{\textstyle<}{\sim}~$}}\kern -.4em}
\def\lsim{\hbox{\kern -.2em\raisebox{-1ex}{$~\stackrel{\textstyle<}{\sim}~$}}\kern -.2em}
\def\[{\Bigl [}
\def\]{\Bigr ]}
\def\({\Bigl (}
\def\){\Bigr )}
\def\[{\Bigl [}
\def\]{\Bigr ]}
\def\({\Bigl (}
\def\){\Bigr )}

\def\R{\mathbb{R}}

\def\T{{\relax\ifmmode I\!\!\hspace{-1pt}T\else$I\!\!\hspace{-1pt}T$\fi}}
\def\N{\mathbb{N}}

\def\la{\lambda}
\def\N{\mathbb{N}}

\def\lsim{\raisebox{-1ex}{$~\stackrel{\textstyle<}{\sim}~$}}

  \def\NN{N}                  

\def\la{\lambda}
\def\La{\Lambda}

\def\al{\alpha}
\def\om{\omega}

\def\argmin{\mathop{\rm argmin}}

\def\be{\beta}

\def\bz{{\bf z}}

\def\ep{\epsilon}

\def\bz{{\bf z}}

\def\argmin{\mathop{\rm argmin}}

\def\Om{\Omega}

\newcommand{\bea}{$$ \begin{array}{lll}}
\newcommand{\eea}{\end{array} $$}

\newcommand{\beqn}{\begin{equation}}
\newcommand{\eeqn}{\end{equation}}

\def\endproof{\hfill\rule{1.5mm}{1.5mm}\\[2mm]}

\makeatletter
\@addtoreset{equation}{section}
\makeatother

\newcommand\dist{\mathop{\rm dist}}

\newcommand\eref[1]{{\rm (\ref{#1})}}

\def\int{\intop\limits}

\newcommand\bU{{\bf U}}

\newcommand{\bom}{{\mbox{\boldmath$\omega$}}}

\date{\today}

\graphicspath{{./}{./Figures}}

\begin{document}

\title{Alleviating missing boundary conditions in elliptic partial differential equations using interior point measurements}

\author{Andrea Bonito}
\address{Department of Mathematics, Texas A\&M University, College Station,   Texas 77843}
\curraddr{}
\email{bonito@tamu.edu}
\thanks{The first author was supported in part by NSF grant DMS-2409807.}

\author{Alan Demlow}
\address{Department of Mathematics, Texas A\&M University, College Station,   Texas 77843}
\curraddr{}
\email{demlow@tamu.edu}
\thanks{The second author was supported in part by NSF grant DMS-2012326.}

\author{Joshua M. Siktar}
\address{Department of Mathematics, Texas A\&M University, College Station,   Texas 77843}
\curraddr{}
\email{jmsiktar@tamu.edu}
\thanks{}

\begin{abstract}
We consider an optimal recovery problem for the Poisson problem when the boundary data is unknown.  Compensating information is provided in the form of a finite number of measurements of the solution. A finite element algorithm for this problem was given in \cite{binev2024solving}, where measurements were assumed to be either bounded linear functionals of the solution or point measurements at locations lying anywhere in the closure of the computational domain.  In contrast, we focus on the case of point measurements at locations lying in the interior of the domain.  This lowers the regularity requirements placed on the solution.  Also, a key ingredient in the recovery process is the finite element approximation of Riesz representers associated with the measurements. Our main result is a pointwise error estimate for the Riesz representers.  We apply this to obtain improved estimates which measure the performance of the recovery algorithm in various norms.  
\end{abstract}


\maketitle


\section{Introduction}\label{Sec: Intro}


This paper focuses on the incomplete Poisson problem
\begin{equation}\label{Eq: Poisson}
        -\lap u \ = \ f \qquad \text{in} \ \ \Om,
\end{equation}
where $\Om \subset \mathbb R^n$, $n=2,3$ is a bounded domain with Lipschitz boundary, and $f\in H^{-1}(\Omega)$ is given.  
When for a given function $g \in H^{1/2}(\partial \Omega)$ the system \eqref{Eq: Poisson} is supplemented by the boundary conditions 
\begin{equation}\label{Eq: Poisson_bdy}
u=g \qquad \textrm{on}\ \ \Gamma:=\partial \Om, 
\end{equation}
the Lax-Milgram theory applies and guarantees a unique solution $u \in H^1(\Omega)$. We assume that we are given complete knowledge of $f$, but are lacking precise knowledge of the boundary data $g$ and thus are unable to uniquely determine $u$.  This setting is motivated by applications in which the inability to measure boundary data rises, including in fluid dynamics \cite{brunton2020machine, grinberg2008outflow, xu2021explore}; wind engineering \cite{richards2019appropriate, richards2011appropriate}, and inverse heat conduction \cite{beck1985inverse, galybin2010reconstruction}, among other areas.


Our goal is to find a ``best possible'' version of $u$ given certain information that further narrows possible choices. Following the framework from \cite{binev2024solving}, we first assume that $g$ belongs to a compact subset $K^B$ of $H^{1/2}(\Gamma)$. 
Given a constant $C_B>0$ and $s>\frac 1 2$, let
\begin{equation}\label{e:model_g}
K^B \ := \ \{ g \in H^s(\Gamma) \ | \ \| g \|_{H^s(\Gamma)} \leq C_B \}.
    \end{equation}
To alleviate the missing information on the boundary data, we further assume that $m$ pointwise measurements of $u$ inside $\Omega$ are provided. More precisely, we assume to know the locations $x_i\in \Omega$ and the values $\om_i=\lambda_i(u):=u(x_i)$, $i=1,...,m$. Note that this requires the interior continuity $u \in C^0(\Omega)$.

For $g \in K^B$, we let $u(g)$ be the solution to \eqref{Eq: Poisson}-\eqref{Eq: Poisson_bdy} and let 
$$
\mathcal K \ := \ \{ u(g) \ | \ g \in K^B \}
$$
be the (compact) set of $H^1(\Omega)$ consisting of all functions that satisfy the incomplete Poisson problem \eqref{Eq: Poisson} with boundary data in $K^B$.
It is sometimes referred to as the model class or the prior. 
All known information available to recover $u$ is encoded in 
\begin{equation}\label{Eq: Kw}
    \mathcal{K}_{\om} \ := \ \{u \in \mathcal K \ | \  \la_i(u) = \om_i, \ i=1,...,m\}.
\end{equation}
Notice that any function in $\mathcal{K}_{\om}$ could be the function from which the measurements were taken. In particular, we cannot expect a recovery algorithm to distinguish between them. Instead, we say that $u^* \in X$ is an \textbf{optimal recovery} in a Banach space $X$ with semi-norm $|.|$ for functions in $\mathcal K \subset X$, if 
\begin{equation}\label{e:opt_intro}
\sup_{v\in \mathcal K_\om} | v - u^* |_X \leq \sup_{v\in \mathcal K_\om} | v- u |_X, \qquad \forall u \in X.
\end{equation}
In this work, we focus on three cases, $X=H^1(\Omega)$ with $|.|_X=\|.\|_{H^1(\Omega)}$, $X=L_\infty(\Omega)$ with $\|.\|_{L_\infty(\Omega)}$, and $X=\{ v \in L_2(\Omega) \ | \ v|_{\Omega_d} \in L_\infty(\Omega_d) \} $ with $|.|_{X}=\|.\|_{L_\infty(\Omega_d)}$ and where $\Omega_d$ is a strict subset of $\Omega$ containing the location of the measurements.

The theoretical formulation of optimal recovery problems, and more broadly the topic of optimal learning, have been studied extensively.
The center of a ball $B(\mathcal{K}_{\om})_X$ in $X$ of smallest radius that contains the set $\mathcal{K}_{\om}$ is an optimal recovery in $X$ and is called a Chebyshev center. The corresponding optimal recovery error is the radius $R(\mathcal{K}_{\om})_X$ of that ball, referred to as the Chebyshev radius. We refer to the classical texts \cite{bojanov1994optimal, micchelli1985lectures, micchelli1976optimal, traub1973theory} for more details and to \cite{binev2024optimal} for a discussion on their approximations. See also \cite{binev2024optimal, binev2017data, maday2015parameterized} for optimal least squares-based approximation schemes.

We remark that the optimal recovery problems considered here are related to optimal control problems where the objective function enforces the value of the state at $x_i=1,...,m$ and the control variable is the Dirichlet boundary values.  In particular, one may minimize $\sum_{j=1}^m |u(x_i)-\omega_i|^2 + \frac{\alpha}{2} \|g\|_{V(\Gamma)}^2$ subject to $-\Delta u =f$ in $\Omega$ and $u=g$ on $\Gamma$.  Here either $V(\Gamma)=L_2(\Gamma)$ or $V(\Gamma)=H^s(\Gamma)$. We are not aware of direct study of this problem in the literature; cf. \cite{AOS18, BDE16, LMV13,vexler2025numerical} for study of other control problems with pointwise measurements in the objective functional. However, optimality in classical optimal control differs from how the term is used in this work. Here, the aim is to study an algorithm designed to approximate the ``best'' possible function $u^*$, i.e., satisfying the optimality condition \eqref{e:opt_intro}. Notice that a function $u^*$ satisfying \eqref{e:opt_intro} up to an absolute multiplicative constant $C>1$ is said to be near-optimal. For instance, any $u \in \mathcal K_\omega$ is near-optimal with a constant at most $2$. This weaker notion of optimality is analyzed in \cite{binev2024optimal}. It was shown that for the optimal control alternative to produce a near-optimal approximation, several conditions must be met, including penalization of the boundary control in $H^s(\Gamma)$ (that is, $V=H^s(\Gamma)$) and not in $L^2(\Gamma)$ as is most frequently done in optimal control.

In the specific case of recovering the solution to the PDE \eqref{Eq: Poisson}, it is convenient to decompose $u=u_0+u_{\mathcal H}$ where $u_0 \in H^1_0(\Omega) \cap C^0(\Omega)$ is the (completely determined) solution of \eqref{Eq: Poisson} supplemented with the vanishing boundary condition $u_{0}=0$ on $\Gamma$ and $u_{\mathcal H}$ is a harmonic function that satisfies the measurements 
\begin{equation}\label{e:measurement_u0}
u_{\mathcal H}(x_i) \ = \ \omega_i - u_0(x_i) \ =: \ \widetilde \omega_i.
\end{equation}
Furthermore, among all harmonic functions satisfying these measurements, the one with minimal norm in $H^s(\Gamma)$ is an optimal recovery for the harmonic part of the solution \cite{binev2024solving}.

The minimal norm property is exploited in \cite{batlle2025error} to design a kernel method algorithm for the approximation of an optimal recovery of general, possibly nonlinear, partial differential equations (PDEs). 
Instead, in \cite{binev2024solving}, the authors focus on the PDE \eqref{Eq: Poisson} and take advantage of the representation, 
\begin{equation}\label{Eq: uHRep}
u_{\mathcal{H}}^* \ = \ \sum_{i=1}^m U_i \phi_i,
\end{equation}
where the coefficients $U_i \in \mathbb R$, $i=1.,..,m$, are determined by the constraints $u_\mathcal{H}^*(x_i)=\tilde{\omega}_i$; these coefficients are determined independently of $X$ since \eqref{Eq: uHRep} is equivalent to finding the function of minimal fractional norm on the boundary amongst functions satisfying the prescribed point measurements (see Subsection \ref{Subsec: MinNormInterp}), which will be crucial for conducting simultaneous error analyses in different norms. For $i=1,...,m$, the function $\phi_i \in \mathcal H^s(\Omega)$ is the Riesz representer of the measurement $\lambda_i$ in the sense that
\begin{equation}\label{Eq: RieszOverview}
    \lang \phi_i, v\rang_{\mathcal{H}^s(\Om)} \ = \ \la_i(v), \ v \in \mathcal{H}^s(\Om),
\end{equation}
where $\mathcal H^s(\Omega)$ is the Hilbert space of harmonic functions with $H^s(\Gamma)$ traces (equipped with the $H^s(\Gamma)$ scalar product).  Consequently, the algorithm proposed in \cite{binev2024solving} consists of approximating $u_0$, $\{\phi_i\}_{i=1}^m$, and $\{U_i\}_{i=1}^m$ using a finite element method.  

We consider variations of two main assumptions made in the predecessor work \cite{binev2024solving}:  The manner in which the recovery error is measured and the location of the measurement points $\{x_i\}$.  First, in \cite{binev2024solving} the recovery error was measured in $X=H^1(\Omega)$.  We additionally consider recovery in $X=L_\infty(\Omega)$, which as a non-Hilbertian Banach space requires a more general theoretical framework.  More precisely, when $X$ is not a uniformly convex Banach space such as when $X=L_\infty(\Omega)$, the Chebyshev ball $B(\mathcal K_{\om})_X$ may not be unique and there thus may exist multiple optimal recovery functions \cite{CDS}.  Finally, we also consider recovery of $u$ in $L_\infty(\Omega_d)$, where $\Omega_d:=\{x \in \Omega:{\rm dist}(x, \partial \Omega)>d\}$ lies on the interior of $\Omega$ a uniform distance $d$ away from $\Gamma$. Note that in all cases the algorithm we obtain is \textbf{near-optimal}, that is, an optimal recovery can be established for any measurement data with the error bounded by a constant only depending on $R(\mathcal{K}_{\om})$. A precise definition is given in Subsection \ref{Sec: Algorithms}. 

We also place an important restriction on the placement of the measurement points $\{x_i\}$.  Whereas in \cite{binev2024solving} the measurement points are allowed to lie anywhere in $\overline{\Omega}$, we require that they lie in $\Omega_d$; cf. \cite{galybin2010reconstruction} for an example in which sensors need to be placed inside of $\Omega$.  To understand the consequences of this difference, we note that the Riesz representers $\{\phi_i\}$ can be calculated via a two-step process.  Let $E:H^{1/2}(\Gamma) \rightarrow H^1(\Omega)$ be the harmonic extension operator.  First one seeks $\psi_i \in H^s(\Gamma)$ satisfying
\begin{equation}
    \langle \psi_i, v\rangle = \lambda_i(v)=Ev(x_i), ~~v \in H^s(\Gamma).
\end{equation}
Then $\phi_i = E \psi_i$.  

If $x_i \in \Gamma$, then $\psi_i$ is a Green's function for an $H^s$-elliptic problem.  It is necessary to assume that $s>\frac{n-1}{2}$ in order to guarantee that the functional $\lambda_i=Ev(x_i) $ lies in $H^{-s}(\Gamma)$, and even if this assumption holds the smoothness of $\lambda_i$ and thus of $\psi_i$ and $\phi_i$ is highly limited regardless of the smoothness of $\Gamma$.  Let $\{U_{i,h}\}_{i=1}^m$ be the approximations to the coefficients $\{U_i\}$ obtained from the finite element algorithm.  In assessing the discretization error in the recovery process, it is necessary to control $\max_{i=1}^M |U_i-U_{i,h}|$ and thus $\max_{1 \le i,j \le m} |(\phi_{i}-\phi_{i,h})(x_j)|$ independent of the choice of $X$.  In \cite{binev2024solving} it is shown that when sufficient regularity is present and the coefficients $U_i$ can be stably determined, there holds for any $\epsilon>0$
\begin{equation} \label{eq: global_order}
|U_i-U_{i,h}| \le C h^{s-\frac{n-1}{2}-\epsilon},
\end{equation}
with the parameters $\frac{1}{2}<s \le 1$ excluded when $n=3$ as previously noted.  

In contrast, when $x_i$ lies in the interior of $\Omega$ the smoothness of the functional $\lambda_i$ is limited only by the smoothness of $\Gamma$.  We thus may allow $\frac{1}{2}<s<\frac{3}{2}$ independent of space dimension.  On the other hand, the smoothness of $\lambda_i$ degenerates as $x_i \rightarrow \Gamma$ and $\psi_i$ thus approaches a boundary Green's function.  In summary, we often obtain higher smoothness and thus a higher order of convergence in our finite element algorithm when $x_i \in \Omega$, but with the caveat that error estimates show deterioration as the measurement points approach $\Gamma$.  More precisely, when $\Gamma$ is sufficiently smooth and $x_i \in \Omega_d=\{x \in \Omega: {\rm dist}(x,\Gamma)>d\}$, $i=1,...,m$, and $U_{i,h}$ is the set of coefficients obtained from the finite element algorithm, we show that when $d \ge c_0 h$ with $c_0$ sufficiently large
$$|U_i-U_{i,h}| \le C d^{-n} (h^{2-\epsilon} + h^{4-2s}).$$
The rates of convergence with respect to $h$ thus obtained are higher than those from \cite{binev2024solving} quoted above in \eqref{eq: global_order} assuming $x_i \in \overline{\Omega}$ for any $\frac{1}{2}<s< \frac{3}{2}$.  On the other hand, the estimates indicate that it is necessary for $h$ to sufficiently resolve $d$ in order for these higher convergence rates to offset the negative powers of $d$ in the estimates.  For example, in the case $n=2$ and $s=1$ we obtain $d^{-n} h^{2-\epsilon}$ whereas \eqref{eq: global_order} yields $h^{1/2-\epsilon}$.  Our numerical experiments confirm that better convergence rates are indeed obtained when the measurement points lie in $\Omega_d$ instead of $\overline{\Omega}$.

We now outline the contents of the remainder of this paper. Section \ref{Subsec: Prelim} includes needed notation for the paper and preliminary results. Next, Section \ref{s:optimal} details the algorithm used for computing an optimal recovery solution through the Riesz representers. Section \ref{s:FEM} contains details of finite element methods used to implement this algorithm, while Section \ref{Subsec: regularityAssump} outlines regularity assumptions and results needed in our analysis.  The core part of the paper is Section \ref{Sec: Mainresult}, where we apply the results of the previous sections to the finite element setting by proving refined pointwise and $H^1$ estimates on the Riesz representers and using these to obtain convergence estimates for the optimal recovery. To conclude, Section \ref{Sec: NumericalResults} includes numerical experiments.

\section{Notations and Preliminaries}\label{Subsec: Prelim}

Let $\Om \subset \R^n$ be a bounded Lipschitz domain. In order to simplify the discussion and avoid accounting for the approximation of $\Omega$, we assume that $\Omega$ is polygonal when $n = 2$, and polyhedral when $n = 3$. We postpone to Section~\ref{Subsec: HigherOrder} a discussion on the extension of the results presented below to smooth domains.



\subsection{Function spaces}\label{Subsec: FcnSpaces}
We now briefly introduce standard functional spaces and refer to  \cite{adams2003sobolev, demengel2012functional, di2012hitchhikers, leoni2023first} for  more details.
For an open set $\mathcal O \subset \Omega$ or $\mathcal O \subset \Gamma$, the space of continuous functions on $\mathcal O$ is denoted $C^0(\mathcal O)$. Among those, we let $C^0(\overline{\mathcal O})$ be the space of functions that can be continuously extended to the closure of $\mathcal O$. We use the standard notation $L_p(\mathcal O)$, $1 \le p \leq \infty$ to denote the standard Lebesgue spaces equipped with the standard norms $\|.\|_{L_p(\mathcal O)}$.
The $L_2(\mathcal O)$ scalar product is denoted $(.,.)_{L^2(\mathcal O)}$, and we write $H^\sigma(\mathcal O)$, $0<\sigma<\infty$ to denote the Sobolev spaces. For a positive integer $k$, we set
\begin{equation}\label{Eq: Hk}
    \|v\|^2_{H^k(\mathcal O)} \ := \ \sum_{|\al| \leq k}\|\pa^{\al}v\|^2_{L_2(\mathcal O)},
\end{equation}
where $\partial^\alpha$ are tangential derivatives when $\mathcal O \subset \Gamma$.
For a non-integral $\sigma>0$, we define
\begin{equation}\label{Eq: Hr}
    \|v\|^2_{H^{\sigma}(\mathcal O)} \ := \ \|v\|^2_{H^k(\mathcal O)} + \sum_{|\al| = k}\iint_{\mathcal O \times \mathcal O}\frac{|\pa^{\al}v(x) - \pa^{\al}v(y)|^2}{|x - y|^{p + 2(\sigma - k)}}dxdy,
\end{equation}
where $k := \lfloor{\sig\rfloor}$, and $p=n$ when $\mathcal O\subset \Omega$ and $p=n-1$ when $\mathcal O \subset \Gamma$.
When $\mathcal O$ is a subset of $\Gamma$, some restrictions on $\alpha$ occur depending on the smoothness of $\Gamma$. 
We comment on these in the next subsection.
The corresponding scalar products are denoted $\langle.,.\rangle_{H^\sigma(\mathcal O)}$.
We shall also need $H^{-1}(\Om)$ to denote the dual space of $H^1(\Om)$. 
Finally, we let $H^1_0(\Om)$ denote the functions in $H^1(\Om)$ with zero trace on $\Gamma$.


\subsection{Trace and extension operators}\label{Subsec: Trace}
The trace operator $T$ is defined for $w \in C(\overline{\Om})$ as the restriction of $w$ to the boundary $\Gamma$, and then subsequently extended to Sobolev functions by density. Thus if $v \in H^1(\Om)$ we can denote
\begin{equation}\label{Eq: TraceDef}
    v_{\Gamma} \ := \ T(v) \ = \ v|_{\Gamma}.
\end{equation}
When no confusion is possible, we shall simply write $v$ instead of $Tv$ or $v|_\Gamma$. This motivates the trace definition of the Fractional Sobolev spaces on the boundary, $H_{T}^\sigma(\Gamma)$: for any $\sigma \geq \frac{1}{2}$, we define 
\begin{equation}\label{Eq: THs}
    H_T^\sigma(\Gamma) \ := \ T(H^{\sigma + 1/2}(\Om)),
\end{equation}
with the underlying trace norm
\begin{equation}\label{Eq: HsNorm}
    \|g\|_{H_T^\sigma(\Gamma)} \ := \ \min\Big\{\|v\|_{H^{\sigma + 1/2}(\Om)} \ | \ v_{\Gamma} \ = \ g\Big\}.
\end{equation}
  When $\Om$ has a smooth boundary and $\sigma \geq \frac{1}{2}$, it is well-known that the intrinsic definition \eqref{Eq: Hr} and the trace definition \eqref{Eq: THs}, \eqref{Eq: HsNorm} of boundary Sobolev spaces are equivalent, i.e. the spaces coincide and the norms are equivalent \cite{leoni2023first}.  The interplay between the two definitions is more subtle in the case of nonsmooth domains.  In \cite{binev2024solving} it is stated without rigorous proof that the two definitions are equivalent on polyhedral domains when $\frac{1}{2}<\sigma<\frac{3}{2}$. This equivalence is not needed in the analysis presented in the following, which frees us from such a restriction on $s$. However, we need a related but weaker result that generally requires $\sigma<\frac{3}{2}$ on polyhedral domains (the focus of this paper; see Section~\ref{Subsec: HigherOrder} for extensions). But because fractional spaces $H^\sigma(\Gamma)$ are degenerate on polyhedral surfaces when $\sigma \ge \frac{3}{2}$, such an assumption is not restrictive. 

We also define for $\frac{1}{2} < \sigma < \frac{3}{2}$ the harmonic function space
\begin{equation}\label{eq: HsDef}
    \mathcal{H}^\sigma(\Om) \ := \ \{v: \Om \rightarrow \R \ | \ \lap v = 0, \ v_{\Gamma} \in H^\sigma(\Gamma)\},
\end{equation}
with norm
\begin{equation}\label{Eq: HSTrace}
    \|v\|_{\mathcal{H}^\sigma(\Om)} \ := \ \|v\|_{H^\sigma(\Gamma)}.
\end{equation}
It is clear that $(\mathcal{H}^\sigma(\Om), \|\cdot\|_{\mathcal{H}^\sigma(\Om)})$ is a Hilbert space; more information on this space is presented in \cite{auchmuty2009reproducing}. Lax Milgram theory, see for instance \cite{Bre, yosida2012functional}, guarantees the existence of constants $c_L, C_L > 0$ depending on $\Om$ such that
\begin{equation}\label{Eq: LMbound}
    c_L\|v\|_{H^1(\Om)} \ \leq \ \|\Delta v\|_{H^{-1}(\Om)} + \|v_{\Gamma}\|_{H^{1/2}(\Gamma)} \ \leq \ C_L\|v\|_{H^1(\Om)}, \ \qquad  v \in H^1(\Om).
\end{equation}
As a consequence, there exists another constant $C_\sigma > 0$ so that
\begin{equation}\label{Eq: CsStability}
    \|v\|_{H^1(\Om)} \ \leq \ C_\sigma\|v\|_{\mathcal{H}^\sigma(\Om)}, \qquad  {v \in \mathcal{H}^\sigma(\Om)}.
\end{equation}

%

We recall that we can decompose any function $u\in \mathcal K$ as 
\begin{equation}\label{eq:decomp}
u=u_0+u_{\mathcal H},
\end{equation} 
where $u_0 \in H^1_0(\Omega) \cap C^0(\Omega)$ is given by \eqref{Eq: Poisson} for $f\in H^{{-1+\beta_\Omega}}(\Omega)$, $\beta_\Omega > \max(0, \frac{n}{2} - 1)$ 
 while $u_{\mathcal H}$ is harmonic in $\Omega$ and satisfies the boundary condition \eqref{Eq: Poisson_bdy} for some $g \in H^{\frac 1 2}(\Gamma)$.
 
 This motivates the introduction of the harmonic extension operator. For any $g \in H^{\frac{1}{2}}(\Gamma)$, we define our harmonic extension operator $E: H^\frac{1}{2}(\Gamma) \rightarrow H^1(\Om)$ via
\begin{equation}\label{Eq: defE}
Eg \ :=\ {\rm argmin}\{\|\nabla v\|_{L_2(\Omega)} \ | \  v_\Gamma=g\}.
\end{equation}
The function $Eg$ is characterized by $T(Eg)=g$ and
\begin{equation}\label{Eq: EgChar}
\int_\Omega \nabla Eg \cdot \nabla v \ = \ 0, \quad v\in H^1_0(\Omega).
\end{equation}
From the left inequality in \eqref{Eq: LMbound}, one has
\begin{equation}\label{Eq: EgStable}
\|Eg\|_{H^1(\Omega)} \ \leq \ c_L^{-1}\|g\|_{H^{1/2}(\Gamma)}, \quad g\in H^{1/2}(\Gamma).
\end{equation}
With $E$ thus defined, we can alternatively write $\mathcal{H}^{\sigma}(\Om) \ = \ \{Ev \ | \ \ v|_{\Gamma} \in H^{\sigma}(\Gamma)\}$.  We postpone to Section~\ref{Subsec: regularityAssump} discussions on the mapping properties of the extension operator.

Finally, we recall that the unknown boundary data function $g$ is assumed to belong to the compact set $K^B$ in \eqref{e:model_g} for some $s>\frac 1 2$.
In view of the discussion above, we restrict the parameter $s$ to $\frac 1 2 <s < \frac 3 2$ and define 
$$
\mathcal K^{\mathcal H}:= \{ v \in \mathcal H^s (\Omega) \ | \ \| v \|_{\mathcal H^s(\Omega)} \leq C_B \} = \{ Eg \ : \ g \in K^B\}
$$
corresponding to the functional space containing all the possible components $u_{\mathcal H}$ in the decomposition \eqref{eq:decomp}.




\subsection{Measurements}

We assume that the locations $x_i$ of the measurements $\lambda_i$ are strictly inside $\Omega$.  In particular, $x_i \in \Omega_d$ ($1 \le i \le m$), where for some parameter $d$, $\Omega_d= \{ x \in \Omega: {\rm dist}(x,\partial \Omega)>d\}$. Whence, there is a constant $\Lambda_0>0$ depending on $\beta_\Omega$ and $d$ such that the following estimate holds:
\begin{equation}\label{e:Lambda0}
\max_{i=1,\dots,m} |\lambda_i(u_0)| \ = \ \max_{i=1,...,m}|u_0(x_i)| \ \leq \ \| u_0\|_{L_\infty(\Omega_d)} \ \leq \ \Lambda_0 \| f \|_{H^{-1+\beta_\Omega}(\Omega)}.
\end{equation}
Similarly, there is a constant $\Lambda_{s,d}>0$ such that for all $v\in \mathcal K^{\mathcal H}$, we have
\begin{equation}\label{e:Lambdas}
\max_{i=1,\dots,m}|\lambda_i(v)| \ \leq \ \max_{i = 1, \dots, m}|v(x_i)| \ \leq \  \La_{s, d} \| v \|_{\mathcal H^s(\Omega)} \ \leq \ \Lambda_{s,d} C_B.
\end{equation}
To establish this bound, note that $v(x_i)= \frac{1}{B_d(x_i)} \int_{B_d(x_i)} v \lesssim d^{-n/2} \|v\|_{L_2(\Omega)}$, where $B_d(x_i)$ is the open ball of radius $d$ centered at $x_i$.  Here and below we write $a \lesssim b$ in order to indicate that $a \le Cb$ with $C$ depending possibly on $s$, $\Omega$, $\Gamma$, and the shape regularity and quasi-uniformity properties of the finite element mesh defined below, but not on other essential quantities.  This combined with boundedness of the harmonic extension operator $E$ (see \eqref{eq: EL2stable} below) yields the desired result.  Alternatively, if $s>\frac{n-1}{2}$ one can combine the maximum principle with a Sobolev embedding $H^s(\Gamma) \subset C(\Gamma)$ in order to obtain the desired result with constant independent of $d$.  Such an assumption on $s$, however, unnecessarily restricts the allowed values of $s$ to $1<s<\frac{3}{2}$ when $n=3$.


It will be convenient to use the following weighted $\ell^2$ norm on $\R^m$ for the vector of the measurements values, i.e.,
 \begin{equation}\label{Eq: weightl2Norm}
     \|\bz \| \ := \ \left(\frac{1}{m}\sum^{m}_{i = 1}|z_i|^2\right)^{\frac{1}{2}} \ = \ m^{-1/2}\|\bz\|_{\ell_2}, \ \bz = (z_1, \dots, z_m) \in \R^m.
\end{equation}
With this notation, for any measurements $\bom:= (\omega_i)_{i=1}^m$ and $u \in \mathcal K_\om$, we have using the decomposition $u = u_0 + u_{\mathcal{H}}$ that
\begin{equation}\label{e:estim_meas}
\| \bom \| \ = \ \left(\frac{1}{m}\sum^{m}_{i = 1}|u(x_i)|^2\right)^{\frac{1}{2}} \ \leq \ \Lambda_0 \|f \|_{H^{-1+\beta_\Omega}(\Omega)} + \Lambda_{s,d} C_B.
\end{equation}
We shall also need an estimate for the modified measurement vector $\widetilde \bom:= (\widetilde \om_i)_{i=1}^m$ defined in \eqref{e:measurement_u0}, which reads
\begin{equation}\label{e:estim_mod_meas}
\| \widetilde \bom \|  \ = \ \|\la(u_{\mathcal{H}})\| \ \leq \ \La_{s, d}C_B, 
\end{equation}
owing to \eqref{e:Lambdas}. Notice the upper bound is independent of $\bom$.

\section{Optimal recovery}\label{s:optimal}

Our task is to recover a function $u \in \mathcal K_\omega$ (see \eqref{Eq: Kw}) using the information provided by the measurements $\lambda_i(u):= u(x_i)=\omega_i$, $i=1,..,m$, to alleviate the lack of information on the boundary data. 
Since $g$ is the only unknown data in \eqref{Eq: Poisson}-\eqref{Eq: Poisson_bdy}, the recovery process consists mainly in determining the best possible harmonic component $u_{\mathcal H} \in \mathcal K^{\mathcal H}$ in \eqref{eq:decomp} that satisfies the measurements \eqref{e:measurement_u0}.
We set $\widetilde \bom:= \bom - (\lambda_i(u_0))_{i=1}^m$ and define the functional set gathering all the information a-priori known on the harmonic component
$$
\mathcal K^{\mathcal H}_{\widetilde \omega}:= \{ v \in \mathcal K^{\mathcal H} \ | \ \lambda_i(v) = \widetilde \omega_i, \ i=1,...,m\};
$$
compare with \eqref{Eq: Kw}.
Note that $\mathcal K^{\mathcal H} \subset C^0(\overline{\Omega_d})$, thereby justifying the use of pointwise measurements located in $\Omega_d$.


In this section, we present near-optimal abstract algorithms based on the approximation of the minimal norm interpolant. A key ingredient is the approximation of the Riesz representers. In Section~\ref{s:FEM}, we present a practical finite element algorithm to achieve the desired approximation properties of the Riesz representers.


\subsection{Optimal Recovery in Banach spaces}\label{Subsec: BanachSpaceSetting}

We first define the notions of Chebychev radius and Chebychev center.
For a Banach space $X$ and $Z \subset X$, the \emph{Chebychev radius} of $X$ with respect to $Z$ is defined as
    \begin{equation}\label{Eq: ChebyRadDef}
        R(Z)_X \ := \ \inf_{v \in X}\sup_{z \in Z}|v - z|_X,
    \end{equation}
    where $|.|_X$ denotes a semi-norm in $X$.
    Furthermore, a function $v^* \in X$ is said to be a \emph{Chebychev center} of $X$ with respect to $Z$ if
    \begin{equation}\label{Eq: ChebyCenterDef}
        v^* \in \argmin_{v \in X}\sup_{z \in Z}|v - z|_X.
    \end{equation}
It is known that if $X$ is uniformly convex and $|.|_X$ is a norm, then the Chebychev center is unique \cite[Lemma 2.3]{DPW}.

The Chebychev radius is the best performance an algorithm can achieve with the partial information encoded in $K_\omega$. Moreover, we say that $v^* \in X$ is the optimal recovery in $(X,|.|_X)$ of functions in $Z$ if
$$
\sup_{z \in Z} | v^*-z |_X = R(Z)_X.
$$

In our setting $Z=\mathcal K_{\omega}$; in view of the decomposition \eqref{eq:decomp},  $u^* = u_0 + u^*_{\mathcal H}$ is an optimal recovery in $(X,|.|_X)$ of functions in $\mathcal K_{\omega}$ if and only if $u^*_{\mathcal H}$ is an optimal recovery in $X$ of functions in $\mathcal K^{\mathcal H}_{\widetilde \omega}$, provided $u_0 \in X$. Therefore, we require the following assumptions.

\begin{assumption}[on $(X,|.|_X)$]\label{ass:X}
Assume that $\mathcal H^s(\Omega) \subset X$ for the topology generated by the semi-norm $|.|_X$, and there is a constant $\Lambda_X$ such that
\begin{equation}\label{e:embbed_X}
| v |_{X} \ \leq \ \Lambda_X \| v \|_{\mathcal H^s(\Omega)}, \qquad \forall v \in \mathcal H^s(\Omega).
\end{equation}
\end{assumption}

In this work, we focus on three settings:
\begin{itemize}
    \item Case 1: $X=H^1(\Omega)$ and $|.|_X=\|.\|_{H^1(\Omega)}$;
    \item Case 2: $X=L_\infty(\Omega)$ with $|.|_X=\|.\|_{L_\infty(\Omega)}$;
    \item Case 3: $X=\{ v \in L_2(\Omega) \ | \  v|_{\Omega_d} \in L_\infty(\Omega_d)\}$ and $|.|_X=|.|_{L_\infty(\Omega_d)}$. 
\end{itemize}
In Case 2, Assumption~\ref{ass:X} is satisfied provided $s>\frac{n-1}{2}$.  In contrast, in Cases 1 and 3, Assumption~\ref{ass:X}  is satisfied without additional restriction besides  $s>1/2$.

Regarding the function $u_0 \in H^1_0(\Omega)$ defined by $-\Delta u_0 = f$ in $\Omega$, we shall require that
$$u_0 \in X \cap C^0(\Omega).$$
We assume throughout that $f \in H^{-1+\beta_\Omega}(\Omega)$ for $\beta_\Omega> \max(0,n/2-1)$, which automatically gives $u_0 \in C^0(\Omega) \cap H^1(\Omega) \subset X$ when $X=H^1(\Omega)$ or $X=\{ v \in L_2(\Omega) \ | \ v|_{L_\infty(\Omega)}\}$. Instead, we further need that $f$ belongs to $L_b(\Omega)$ for $b>n/2$ and that $\Omega$ satisfies the exterior cone property  when $X=L_\infty(\Omega)$ \cite{gilbarg1977elliptic}.

\subsection{Minimal Norm Interpolant}\label{Subsec: MinNormInterp}

We start by recalling that from the decomposition \eqref{eq:decomp}, we conclude that $u_\mathcal H \in \mathcal K^{\mathcal H}_{\widetilde \omega}$ is the component of $u$ to be recovered. To this aim, we define the \emph{minimal norm interpolant} $u^*_{\mathcal H}:=u^*_{\mathcal H}(\widetilde \bom)$ as
\begin{equation}\label{Eq: MinNorm}
    u_{\mathcal H}^* \ := \ \argmin\{\|v\|_{\mathcal{H}^s(\Om)} \ | \ v \in \mathcal{H}^s(\Om), \ \la_i(v) = \widetilde \omega_i, \ i =1,...,m\}.
    \end{equation}
    As in \eqref{Eq: uHRep}, the quantity \eqref{Eq: MinNorm} is independent of $X$.
    Notice that the harmonic part of the solution to recover and from which the measurements are obtained belongs to the set $\{ v \in \mathcal{H}^s(\Om) \ | \ \la_i(v) = \widetilde \omega_i, \ i=1,..,m\}$ (which is therefore not empty). This, together with the strong convexity of the $\mathcal{H}^s(\Om)$-norm guarantees the existence and uniqueness of the minimal norm interpolant. 

For any Banach space $X$ and semi-norm $|.|_X$ satisfying Assumption~\ref{ass:X}, the minimal norm interpolant is an optimal recovery in $(X,|.|_X)$ for functions in $\mathcal K^{\mathcal H}_{\widetilde \omega}$. Indeed, the arguments in Section 1.4 of \cite{binev2024solving} solely based on the central symmetry of $\mathcal K^{\mathcal H}_{\tilde \omega}$ around $u^*_{\mathcal H}$ and the triangle inequality property satisfied by the (semi-)norm in $X$ carry through general Banach spaces $X$, even when the Chebysvhev center is not unique. In particular, $u^*_{\mathcal H} \in \mathcal K^{\mathcal H}_{\omega} \subset X$ and we have
\begin{equation}\label{Eq: GenericBEstimateOverview1}
    \sup_{w\in \mathcal K^{\mathcal H}_{\widetilde \omega}}|w - u^*_{\mathcal H}|_{X}= R(\mathcal K^{\mathcal H}_{\widetilde w})_X.
\end{equation}
This implies that under Assumption~\ref{ass:X} and for $u_0\in X$, $u^*=u_0 + u^*_{\mathcal H}$ is an optimal recovery in $(X,|.|_X)$  for functions in $\mathcal K^{\mathcal H}_{\widetilde \omega}$.

The minimal norm interpolant has an algorithm-friendly characterization using the Riesz representers of the measurement functionals $\phi_i$, $i=1,...,m$, defined by \eqref{Eq: RieszOverview}. Let $\mathbb W:= \textrm{span}(\phi_1,...,\phi_m)$ and note that the measurements $\lambda_i(v)=\tilde \omega_i:= \omega_i - \lambda_i(u_0)$, $i=1,...,m$ determine the component of $v \in \mathcal H^s(\Omega)$ in $\mathbb W$. Furthermore, the minimal norm interpolant $u_{\mathcal H}^*$ belongs to $\mathbb W$ because any nonzero component in the orthogonal complement would increase the $\mathcal H^s(\Omega)$-norm.
Whence, we find that 
\begin{equation}\label{e:u_star_span}
u_{\mathcal H}^* \ = \ \sum_{i=1}^m U_i \phi_i,
\end{equation}
for some $U_i \in \mathbb R$ that satisfies the relation
\begin{equation}\label{e:obs_system}
G \bU \ = \ \widetilde{\bom}.
\end{equation}
Here $\bU = (U_i)^m_{i = 1} \in \mathbb R^m$, $\widetilde \bom = (\omega_i)^m_{i = 1} \in \mathbb R^m$, and $G\in \mathbb R^{m\times m}$ is the observation matrix of components $G_{ij}:= \lambda_i(\phi_j)$. We point out that the matrix $G$ may not be invertible but $\widetilde \bom$ belongs to its range by construction. Yet, if $G^{-1}$ denotes the Moore-Penrose pseudo inverse of $G$, we select one of the solutions $\bU$ by the relation
\begin{equation}\label{e:u_star_inv_span}
\bU \ = \ G^{-1} \widetilde \bom.
\end{equation}
In order to simplify the discussion, from now on we assume that the Riesz representers are linearly independent, and thus $G$ is invertible. 

\subsection{Near-optimal recovery algorithm}\label{Sec: Algorithms}

In this section, we slightly modify the algorithm introduced in \cite{binev2024solving} for the approximation of the minimal norm interpolant $u^*_{\mathcal H}$. Compared to \cite{binev2024solving}, we provide a more general version of recovery in a Banach space $X$ with semi-norm $|.|_X$ that satisfies Assumption~\ref{ass:X}. However, the essence of the algorithm remains unchanged. It consists of replacing the Riesz representers in \eqref{e:u_star_span} and \eqref{e:obs_system} by computable approximations. Our goal is to achieve a near-optimal recovery $u^*$ in the sense that for some fixed $C \ge 1$, 
\begin{equation} \label{eq:nearopt}
\sup_{v \in \mathcal K_{\om}} |u^*-v|_X \le C R(\mathcal K_\omega)_X.
\end{equation}
In practice we instead assume that we are given $\epsilon>0$ with $\epsilon \lesssim R(K_\omega)_X$ and seek $u^*$ such that
\begin{equation} \label{eq:nearopt2}
\sup_{v \in \mathcal K_\om} |u^*-v|_X \le R(\mathcal K_\omega)_X+\epsilon.
\end{equation}

\begin{algorithm}[Recovery Algorithm in $X$]\label{algo:rec}~
\begin{itemize}
    \item[Input:] $\Omega$, $u_0$, $f$, $s$, measurement couples $(\lambda_j,\omega_j)_{j=1}^m$ and tolerances $\ep_1>0$, $\ep_2>0$;
    \item[Output:] Approximation $\widehat{u}$ of recovery {$u^*$}.
    \end{itemize}
\begin{enumerate}[Step 1.]
  \item Compute an approximation $\widehat{u_0}$ of $u_0$ that satisfies
  \begin{equation}\label{Eq: NROAlgEq0B}
  |\widehat{u}_0 - u_0|_{X} +\max_{i=1,...,m} |\widehat u_0(x_i)-u_0(x_i)| \ \leq \ \ep_1
  \end{equation}
  and compute the vector of modified measurements 
  $$
  \widehat \bom \ := \ (\omega_i - {\widehat{u_0}}(x_i))_{i=1}^m \subset \mathbb R^m.
  $$
  \item Compute approximations $\widehat{\phi}_j$ of each $\phi_j$, $1 \leq j \leq m$ that satisfies 
  \begin{equation}\label{Eq: NROAlgEqTB}
  |\widehat{\phi}_j - \phi_j|_{X}  +\max_{i=1,...,m} |\widehat \phi_j(x_i)-\phi_j(x_i)| \ \leq \ \ep_2.
  \end{equation}
  \item Assemble $\widehat{G} := (\widehat \phi_j(x_i))_{1 \leq i, j \leq m}$. 
  \item Set $\widehat u^*_{\mathcal H}:= \sum_{i=1}^m \widehat U_i \widehat \phi_i$, where the coefficients $\widehat \bU:=(\widehat U_i)_{i=1}^m$ are given by
  \begin{equation}\label{e:hat_U}
  \widehat{\bU} \ := \ \widehat G^{-1}\widehat \bom, 
  \end{equation}
  with $\widehat G^{-1}$ being the Moore-Penrose pseudo-inverse of $\widehat G$.
  \item Output $u^*:=\widehat u_0 + \widehat u^*_{\mathcal H}$.
  \end{enumerate}
\end{algorithm}

Theorem~\ref{Thm: NORB} below assesses the efficiency of the algorithm.
Its proof relies on the following lemma gathering several estimates proved in \cite{binev2024solving}.

\begin{lemma}\label{l:estim_G}
    Let $G := (\lambda_i(\phi_j))_{i,j=1}^m$ be the observation matrix and $\widehat G=(\lambda_i(\widehat \phi_j))_{i,j=1}^m$ its approximation constructed in Algorithm~\ref{algo:rec} with input tolerances $\ep_1>0, \ep_2>0$.
    There holds
    \begin{equation}\label{e:estim_G}
    \| G - \widehat G\|_{\ell_1\to \ell_1} \ \leq \ m\ep_2.
    \end{equation}
    Furthermore, assume that $G$ is invertible. Then there is $\overline{\ep}_2>0$ depending on $\| G^{-1}\|_{\ell_1 \to \ell_1}$ and $m$ such that for $\ep_2 \leq \overline{\ep}_2$, $\widehat G$ is invertible and we have
\begin{equation}\label{e:delta}
\|G^{-1} - \widehat{G}^{-1}\|_{\ell_1\to \ell_1} \ \leq \ \frac{m \| G^{-1}\|_{\ell_1\to \ell_1}^2\ep_2}{1 - m \| G^{-1}\|_{\ell_1\to \ell_1}\ep_2} \ =: \ \de.
\end{equation}
\end{lemma}
%

We are now in position to provide the near-optimal recovery theorem.

\begin{theorem}[Near-optimal recovery in $X$]\label{Thm: NORB}
    Assume $(X, |\cdot|_X)$ is a Banach space satisfying Assumption~\ref{ass:X}, that $u_0 \in X \cap C(\Omega)$, and that the observation matrix $G$ is invertible. Let $\ep > 0$. Select $\ep_1, \ep_2>0$ such that $\ep_2 \leq \overline{\ep}_2$, where $\overline{\ep}_2$ is given by Lemma~\ref{l:estim_G} and such that 
    \begin{equation}\label{Eq: NORB}
        \ep_1 +m \|G^{-1}\|_{\ell_1\to \ell_1} ({
        \Lambda_{s,d}C_B})\ep_2 +(C_X+\ep_2)(m \|G^{-1}\|_{\ell_1\to \ell_1}\ep_1 +m{\Lambda_{s,d}C_B}+\ep_1 )\delta) \ \leq \ \ep.
    \end{equation}
    Here $C_X := \max_{1 \leq  j \leq m}|\phi_j|_{X}<\infty$, $\delta$ is given by \eqref{e:delta}, and $\La_0$ and $\La_s$ are defined in \eqref{e:Lambda0} and \eqref{e:Lambdas} resp. Then the function $\widehat{u}^*$ generated by Algorithm \ref{algo:rec} with parameters $\ep_1,\ep_2$ satisfies 
    \begin{equation}\label{Eq: NORThmB}
        |v - \widehat u^*|_{X} \ \leq \ R(\mathcal{K}_{\om})_{X} + \ep, \qquad \forall v \in \mathcal K_\omega.
    \end{equation}
  That is, $\widehat{u}^*$ is a near-optimal recovery in $(X,|.|_X)$ for functions in $\mathcal{K}_{\om}$.
  \end{theorem}
\begin{proof} 
Let $v\in \mathcal K_w$. From the definition of $R(\mathcal K_w)_X$ in {\eqref{Eq: ChebyRadDef}}, we directly find that
$$
 |v - \widehat u^*|_{X} \ \leq \ R(\mathcal K_w)_X + | u^* - \widehat u^*|_X,
$$
where $u^*=u_0+u^*_{\mathcal H}$ and $u^*_{\mathcal H}$ is the minimal norm interpolant; see \eqref{Eq: MinNorm}. 
Furthermore, in view of \eqref{Eq: NROAlgEq0B}, we get
$$
|u^* - \widehat u^*|_{X} \ \leq \ |u_0 - \widehat{u}_0|_X + | u^*_{\mathcal H}-\widehat u^*_{\mathcal H}|_{X} \leq \ep_1 +  | u^*_{\mathcal H}-\widehat u^*_{\mathcal H}|_X
$$
so that 
\begin{equation}\label{e:recov_tmp}
|v - \widehat u^*|_X  \ \leq \ R(\mathcal K_w)_X + \ep_1 +  | u^*_{\mathcal H}-\widehat u^*_{\mathcal H}|_X.
\end{equation}

The rest of the proof focuses on the derivation of an estimate of the discrepancy in the harmonic part $ | u^*_{\mathcal H}-\widehat u^*_{\mathcal H}|_{X}$. 
Recall that $u^*_{\mathcal H}=\sum_{i=1}^m U_i \phi_i$, where $U_i \in \mathbb R$, $i=1,...,m$, satisfy the algebraic system \eqref{e:u_star_span}. We compute
 \begin{equation}\label{e:to_estim_meas}
  |u^*_{\mathcal H}-\widehat u_{\mathcal H}|_{X} \ \leq \ \Big|\sum_{j=1}^m U_j(\phi_j- \widehat \phi_j)\Big|_{X}+\Big|\sum_{j=1}^m (U_j-\widehat U_j) \widehat \phi_j\Big|
  _{X}
\  \leq \ \|\bU\|_{\ell_1}\ep_2+\|\bU-\widehat \bU\|_{\ell_1}( C_X+\ep_2),
\end{equation}
where we have used the approximation estimate \eqref{Eq: NROAlgEqTB} for the Riesz representers and the definition of $C_X= \max_{i=1,...,m} | \phi_i |_{X}$.
Note that $C_X \leq \Lambda_X \max_{i=1,...,m} \| \phi_i \|_{\mathcal H^s(\Omega)} < \infty$ thanks to \eqref{e:embbed_X}.

The algebraic relation \eqref{e:u_star_span} for $\bU$ yields
\begin{equation*}
  \| \bU \|_{\ell_1}=\|G^{-1} \widetilde \bom \|_{\ell_1}
  \le    \|G^{-1}\|_{\ell_1\to \ell_1} \| \widetilde \bom\|_{\ell_1} \leq  m  \|G^{-1}\|_{\ell_1\to \ell_1}
   \|\widetilde \bom \|\leq m  \|G^{-1}\|_{\ell_1\to \ell_1}{ \Lambda_{s,d}C_B} 
\end{equation*}
  where we used the estimate \eqref{e:estim_mod_meas} on the modified measurement values.

  Combining the previous two estimates, we obtain
  \begin{equation}
  \label{vest}
  |{u^*_{\mathcal{H}}-\widehat u_{\mathcal{H}}}|_{X}\leq m \|G^{-1}\|_{\ell_1\to \ell_1} {\Lambda_{s,d} C_B}\ep_2+\|\bU-\widehat \bU\|_{\ell_1}( C_X+\ep_2).
\end{equation}
  For the estimation of $\|\bU-\widehat \bU\|_{\ell_1}$, we introduce the intermediate vector $\widetilde \bU:= G^{-1} \widehat \bom \in \R^m$ and compute
$$  
  \|\widetilde \bU - \bU\|_{\ell_1}
\  = \ \|G^{-1}(\widehat \bom-\widetilde \bom)\|_{\ell_1}
  \ \leq \ \|G^{-1}\|_{\ell_1\to \ell_1} \|\widehat \bom-\widetilde \bom\|_{\ell_1}
  \ \leq \ m \|G^{-1}\|_{\ell_1\to \ell_1}\|\widehat \bom-\widetilde \bom\|.
$$ 
Notice that $\widehat \omega_i-\widetilde \omega_i = \lambda_i(u_0-\widehat u_0) \leq \ep_1$ thanks to \eqref{Eq: NROAlgEq0B} and so
$$
\|\widetilde \bU - \bU\|_{\ell_1}  \ \leq \ m\|G^{-1}\|_{\ell_1\to \ell_1} \e_1.
$$
Furthermore, in view of \eref{e:estim_meas}, \eref{Eq: NROAlgEq0B}, and estimate {\eqref{e:delta}} for $G^{-1}-\widehat G^{-1}$, we have
$$
  \|\widetilde \bU -  \widehat \bU\|_{\ell_1}
  =\|(G^{-1}-\widehat G^{-1})\widehat \bom\|_{\ell_1} \ \leq \ \|(G^{-1}-\widehat G^{-1})\|_{\ell_1\to\ell_1} \|\widehat \bom\|_{\ell_1}
\ \leq \ m({\Lambda_{s,d}C_B} +\ep_1) \delta.
$$ 
Combining the estimates for $\| \bU - \widetilde \bU\|_{\ell_1}$ and $\| \widetilde \bU - \widehat \bU\|_{\ell_1}$, 
   we find that
$$ 
  \|\bU-\widehat \bU\|_{\ell_1}
\ \leq \ m\|G^{-1}\|_{\ell_1\to \ell_1}\e_1 +m({\Lambda_{s,d}C_B} +\ep_1 )\delta.
$$  
 We now insert this bound into \eref{vest} and invoke \eref{Eq: NORB} to obtain 
  \begin{equation*}
            |u_{\mathcal H}^*-\widehat u_{\mathcal H}|_{X} \ \leq \ {\ep - \ep_1}.
  \end{equation*}
The statement of the theorem follows using the above inequality in \eqref{e:recov_tmp}.
  \end{proof}

{ We end this section noting that a similar analysis holds when the measurements are noisy, i.e., one observes $\bom = \bom^{exact}+{\boldsymbol \eta}$ where $\bom^{exact}$ is the exact measurement vector and the noise ${\boldsymbol \eta}$ satisfies $\| {\boldsymbol \eta}\| \leq \delta$ for some $\delta>0$. We refer to \cite{binev2024optimal} for similar considerations. In contrast, the case of Gaussian noise is more intricate and subject to further studies.
}

\section{Finite element realizations of the recovery algorithms}\label{s:FEM}

In this section, we discuss how to construct $\widehat u_0$ and $\widehat \phi_j$, $j=1,...,m$ satisfying \eqref{Eq: NROAlgEq0B} and \eqref{Eq: NROAlgEqTB}.

\subsection{Finite element setting}\label{Subsec: ApproxFESetting}


Recall that $\Om$ is assumed to be Lipschitz and polygonal ($n = 2$) or polyhedral ($n = 3$).
With $h$ denoting the maximal diameter within a given mesh, let $(\mathcal{T}_h)_{h > 0}$ denote a family of exact sub-divisions made of triangles ($n=2)$ or tetrahedra ($n=3$) of $\Om$. We also denote by $\mathcal{T}_h^\Gamma$ the restriction of $\mathcal{T}_h$ to $\Gamma$. Below $\mathcal{N}_h^\Gamma$ is the set of vertices in the mesh $\mathbb{T}_h$, and given $z \in \mathcal{N}_h^\Gamma$, $\omega_z$ is the patch of elements in $\mathbb{T}_h^\Gamma$ sharing the vertex $z$.

We assume that $(\mathcal T_h)_{h>0}$ is shape-regular, that is, all elements $T$ have a uniformly bounded ratio between their diameters $h(T)$ and the diameter $\rho(T)$ of their inner circle.
We additionally assume that the mesh is quasi-uniform, i.e., $h(T')$ is uniformly equivalent to $h$ for all $T' \in \mathcal{T}_h$.  We let $\mathbb{V}_h$
denote the continuous piecewise linear finite element space associated with $\mathcal{T}_h$. In addition, we denote
\begin{equation}\label{Eq: Wh}
    \mathbb{V}_h^0 \ := \ \{v_h \in \mathbb{V}_h \ | \ T(v_h) = 0\}
\end{equation}
as the finite element subspace with a homogeneous (Dirichlet) boundary condition.  The restriction to $\Gamma$ of functions in $\mathbb V_h$ is
$$
\mathbb T_h \ := \ \{ T(v_h) \ : \ v_h \in \mathbb V_h\},
$$
where $T:H^{\sigma+1/2}(\Omega) \rightarrow H^{\sigma}(\Gamma)$ for $0<\sigma<3/2$ is the trace operator introduced in Section~\ref{Subsec: Prelim}.
Since $\mathbb V_h \subset H^1$ and we have assumed that $s < \frac{3}{2}$ (where we recall that $s$ is as in \eqref{e:model_g}), we also have $\mathbb{T}_h \subset H^s(\Gamma)$.

We shall also need finite element approximation results and inverse estimates in $H^{\sig}(\Gamma)$ for $0 \le \sigma < \frac{3}{2}$.  We first state a foundational localization result.  
\begin{proposition}[Localization] \label{prop: localization}
Let $u \in H^\sigma(\Gamma)$. When $0 < \sigma <1$, there holds
\begin{equation} \label{faermann_localization}
    \|u\|_{H^\sigma(\Gamma)}^2 \ \lesssim \  h^{-2\sigma} \|u\|_{L_2(\Gamma)}^2 + \sum_{z \in \mathcal{N}} |u|_{H^\sigma (\omega_z)}^2.
\end{equation}
 If $0<\sigma <\frac{1}{2}$, there additionally holds
\begin{equation} \label{pasciak_localization}
    \|u\|_{H^\sigma(\Gamma)}^2 \ \lesssim \ h^{-2\sigma} \|u\|_{L_2(\Gamma)}^2 +\sum_{T \in \mathcal{T}_h^\Gamma} |u|_{H^\sigma(T)}^2.
\end{equation}
When instead $1<\sigma<\frac{3}{2}$, we have for $\alpha=\sigma-1$ that
\begin{equation} \label{fine_localization}
    \|u\|_{H^\sigma(\Gamma)} ^2 \ \lesssim \ \|u\|_{L_2(\Gamma)}^2 + h^{-2\alpha} \|\nabla_\Gamma u\|_{L_2(\Gamma)}^2 + \sum_{T \in \mathcal{T}_h^\Gamma}|\nabla_\Gamma u|_{H^{\alpha}(T)}^2.
\end{equation}
\end{proposition}

The well-known result \eqref{faermann_localization} is found in \cite{faermann2000localization, faermann2002localization}.  The finer result \eqref{pasciak_localization} is found in \cite[Lemma A.1]{BLP19}, while \eqref{fine_localization} is obtained by a straightforward combination of \eqref{pasciak_localization} and the definition of $\|\cdot\|_{H^\sigma(\Gamma)}$. We emphasize that \eqref{pasciak_localization} and \eqref{fine_localization} have only element-wise terms on the right hand side.  The proof of \eqref{pasciak_localization} in \cite{BLP19} involves mainly elementary computations involving double integrals along with the use of bounded extension operators in $H^\sigma$ norms.  A different proof can be obtained by noting from \cite{Str67} that indicator functions of subdomains are multipliers for $H^\sigma(\mathbb{R}^{n-1})$ when $0<\sigma<\frac{1}{2}$.  Combining this result with bounded extension operators for $H^\sigma$ and \eqref{faermann_localization}, while noting that the Jacobian of the natural patchwise reference transformation from $\omega_z$ to a corresponding Euclidean reference patch is piecewise constant also yields \eqref{pasciak_localization} and \eqref{fine_localization}.  

We denote by $P_h: H^1(\Omega) \rightarrow \mathbb V_h$ the Scott-Zhang projection operator on $\Omega$ \cite{SZ} and recall that it preserves boundary values and is $H^1$ stable in the following sense:
\begin{equation}\label{Eq: szproj}
    P_h v_h|_\Gamma \ = \ v_h|_\Gamma ~~ \forall v_h \in \mathbb V_h, \qquad \textrm{and} \qquad \|P_hv\|_{H^1(\Omega)}\leq B\|v\|_{H^1(\Omega)},\quad v\in H^1(\Omega),
\end{equation}
for some $B$ independent of $h$. We shall use standard approximation properties of $P_h$ in Sobolev spaces on $\Omega$ without citation.

We also need a Scott-Zhang projection operator on $\Gamma$ that is consistent with $P_h$. 
For each $z \in \mathcal{N}_h^\Gamma$, choose an arbitrary $T_z \in {\mathcal{T}_h^\Gamma}$ with $z \in \overline{T_z}$, let $\phi_z$ be the standard Lagrange basis function corresponding to $z$, and let $\psi_z \in \mathbb{P}^1$ satisfy $\int_{T_{i}} \phi_{z_i} \psi_{z_j}= \delta_{ij}$.  Then define
\begin{equation} \label{eq: pihdef}
(\pi_h u )(x)= \sum_{z \in \mathcal{N}_h^\Gamma} \phi_z(x) \int_{T_z} u \psi_z.    
\end{equation}
The faces $T_z$ in \eqref{eq: pihdef} may be chosen so that the definitions of $\pi_h$ and $P_h$ are consistent, i.e.,
\begin{equation}\label{e:pi_hP_h}
(P_h v)|_\Gamma = \pi_h (v_\Gamma).
\end{equation}
We shall use this property in our analysis. 
Moreover, standard arguments yield
\begin{equation}\label{e:stab_pi_h_sigma0}
\| \pi_h \kappa \|_{L_2(\Gamma)}\lesssim \| \kappa \|_{L_2(\Gamma)}, \qquad \forall \kappa \in L_2(\Gamma)
\end{equation}
and for $0<\xi \leq 1$, there holds
\begin{equation} \label{approx_smaller_sigma0}
    \|\kappa-\pi_h \kappa\|_{L_2(\Gamma)} \ \lesssim \ h^{\xi} |\kappa|_{H^\xi(\Gamma)}, \qquad \forall \kappa \in H^\xi(\Gamma).
\end{equation}

We shall need stability and approximability properties of $\pi_h$ on $H^\sigma(\Gamma)$ for $0\leq \sigma < 3/2$ and these are the subject of Lemmas~\ref{frac_approximation} and~\ref{frac_approximation2} below. In preparation for these results, we introduce yet another Scott-Zhang type interpolant. 
When $n=3$ we let $I_h: H^1(\Gamma) \rightarrow \mathbb{T}_h$ be the Scott-Zhang interpolant defined using integrals over triangle edges (instead of triangles as in \eqref{eq: pihdef}) as degrees of freedom while when $n=2$ (and so $\Gamma$ is a 1D curve) we take $I_h$ to be the Lagrange interpolant.  
Note that $I_h$ is locally defined and locally stable in $H^\sigma{(\Gamma)}$ for any $\sigma \ge 1$.
Lemma~3.1 and Theorem~3.2 in \cite{CD15} guarantee that if $1 \le \sigma < \frac{3}{2}$,  $\xi \ge \sigma$, and $T \in \mathcal{T}_h^\Gamma$, there holds
\begin{equation} \label{Eq: camacho-demlow}
\|\kappa-I_h \kappa\|_{H^\sigma(T)}^2 \lesssim h^{2(\xi-\sigma)} \sum_{T' \subset \omega_{T'}} |\kappa|_{H^\xi(T')}^2.
\end{equation}
Here $\omega_{T'}$ consists of all elements $T \in \mathcal T_h^\Gamma$ such that $T\cap T' \not = \emptyset$.
 The result as stated in \cite{CD15} assumes that $\sigma$ and $\xi$ are integers, but the proof holds nearly verbatim for fractional $\sigma, \xi$ after application of the element-wise Bramble-Hilbert Lemma for fractional norms (cf. \cite[Theorem 6.1]{DS80}). Instead, when $n=2$, $I_h$ is the Lagrange interpolant and is element-wise stable on the 1D surface $\Gamma$, i.e., \eqref{Eq: camacho-demlow} holds with right-hand side $h^{2(\xi-\sigma)} |\kappa|^2_{H^\xi(T)}$.  
The projection $I_h$ is not stable in $L_2(\Gamma)$ and does not satisfy the consistency relation \eqref{e:pi_hP_h}. It is merely used as a tool to obtain estimates similar as \eqref{Eq: camacho-demlow} for $\pi_h$, see Lemmas~\ref{frac_approximation} and~\ref{frac_approximation2} below.

We also note that to measure the regularity of a function $\kappa:\Gamma \rightarrow \mathbb R$ when $\Gamma$ is polynomial/polyhedral, it is necessary to modify the definition of $H^{\xi}(\Gamma)$ for $\xi \geq 3/2$ to avoid degeneracy across edges and vertices of $\Gamma$.
We define 
\begin{equation}\label{e:broken_spaces}
H^\xi_h(\Gamma) :=
\begin{cases}
    H^\xi(\Gamma) & \qquad \textrm{when }0<\xi <\frac{3}{2} ;\\
    \Pi_{T\in \mathcal T_h^\Gamma} H^\xi(T) & \qquad \textrm{when } \frac{3}{2}\leq \xi \leq 2,
\end{cases}
\end{equation}
and set $|\kappa|_{\tilde{H}^\xi(\Gamma)} := |\kappa|_{H^\xi(\Gamma)}$ when $0<\xi<\frac{3}{2}$ and $|\kappa|_{\tilde{H}^\xi(\Gamma)}^2 := \sum_{T \in \mathbb{T}_h} |\kappa|_{H^\xi(T)}^2$  when $\frac{3}{2} \le \xi \le 2$.
As alluded to above, the distinction between $\xi<\frac{3}{2}$ and $\xi\ge \frac{3}{2}$ in the definition of $H^\xi_h(\Gamma)$ is critical when considering polygonal/polyhedral surfaces $\Gamma$. Measuring regularity in these spaces globally or on element patches that may overlap multiple faces of $\Gamma$ is not feasible. 
 On the other hand, as we discuss below, solutions to elliptic problems on polyhedral surfaces may possess $H^\xi_h(\Gamma)$ regularity with $\xi \ge \frac{3}{2}$ even when they do not possess such {\it global} regularity.

\begin{lemma} \label{frac_approximation}
Let $\kappa \in H^\sigma(\Gamma)$ for some $0 \le \sigma<\frac{3}{2}$.  There holds
\begin{equation} \label{Eq: interp_stability}
\|\pi_h \kappa\|_{H^\sigma(\Gamma)} \lesssim \|\kappa\|_{H^\sigma(\Gamma)} \hbox{ if } 0 \le \sigma <\frac{3}{2}, \qquad \|I_h \kappa\|_{H^\sigma(\Gamma)} \lesssim \|\kappa\|_{H^\sigma(\Gamma)} \hbox{ if } 1 \le \sigma < \frac{3}{2}.
\end{equation}
\end{lemma}
\begin{proof}
    The stability estimate for $I_h$ follows by taking $1 \le \xi=\sigma< \frac{3}{2}$ in \eqref{Eq: camacho-demlow}, {applying the localization result \eqref{fine_localization}}, and summing the resulting estimates over $T \in \mathcal T_h^\Gamma$ {while recalling the finite overlap property of the patches $\omega_{T'}$}.

    We now turn our attention to the stability of $\pi_h$. The case $\sigma=0$ is \eqref{e:stab_pi_h_sigma0} and so we consider $0 <  \sigma  \le 1$. Let $z \in \mathcal N^{\Gamma}_h$. Since $|C|_{H^\sigma(\omega_z)} = 0$ {and $\pi_h C=C$} for any $C \in \mathbb R$, $\Gamma$ is a Lipschitz polyhedral domain, and $ 0 < \sigma \le 1$, we may use standard finite element reference maps and scaling arguments.  In particular, $\#\{T\in \mathcal T^\Gamma_h \ | \ T \subset \omega_z \}$ is uniformly bounded depending on the shape regularity of $\{ \mathcal T_h \}_{h>0}$ and the maximum valence of vertices of $\Gamma$, and thus $\mathbb{T}_h|_{\omega_z}$ is a finite dimensional space of uniformly bounded dimension. Therefore, \eqref{eq: pihdef}, standard scaling arguments, and the Bramble-Hilbert Lemma yield for properly chosen $C \in \mathbb{R}$ that
\begin{equation} \label{eq: stable1}
\begin{aligned} 
|\pi_h \kappa|_{H^\sigma(\omega_z)}  & = |\pi_h (\kappa-C)|_{H^\sigma(\omega_z)}  
= \left| \sum_{z' \in \mathcal{N}_{\omega_z}} \phi_{z'} \int_{T_{z'}} (\kappa-C) \psi_{z'}\right|_{H^\sigma(\omega_z)}
\\ &  \lesssim \sum_{z' \in \mathcal{N}_{\omega_z}} |\phi_{z'}|_{H^\sigma(\omega_z)} \|\kappa-C\|_{L_2(T_{z'})} \|\psi_{z'}\|_{L_2(T_{z'})} 
\\ &  \lesssim  \sum_{z' \in \mathcal{N}_{\omega_z}} h^{(n-1)/2-\sigma} \|\kappa-C\|_{L_2(T_{z'})} h^{-(n-1)/2}
 \lesssim h^{-\sigma} \|\kappa-C\|_{L_2(\omega_{z}')} \lesssim |\kappa|_{H^\sigma(\omega_{z}')}.
\end{aligned}    
\end{equation}
Here $\mathcal{N}_{\omega_z}$ is the set of nodes lying in $\overline{\omega_z}$ and $\omega_{z}':= \cup_{z' \in \mathcal N^\Gamma_h, \ z' \in \omega_z} \omega_{z'}$ is the patch of elements touching $\omega_z$.   Note that we have applied a form of the Bramble-Hilbert Lemma that applies to unions of star-shaped domains \cite[Theorem 7.1]{DS80}, since $\omega_z$ is star-shaped but $\omega_z'$ may not be. {$C$ is taken to be an appropriate degree-0 polynomial having the stated approximation properties, which is guaranteed to exist by the cited version of the Bramble-Hilbert Lemma.}  Thus employing \eqref{faermann_localization}, \eqref{approx_smaller_sigma0}, finite overlap of the patches $\omega_z'$, and \eqref{eq: stable1} we obtain
\begin{equation*}
\begin{aligned}
\|\pi_h \kappa\|_{H^\sigma(\Gamma)}^2 & \lesssim \|\kappa-\pi_h \kappa\|_{H^\sigma(\Gamma)}^2  +\|\kappa\|_{H^\sigma(\Gamma)}^2 
\\ & \lesssim h^{-2\sigma} \|\kappa-\pi_h \kappa\|_{L_2(\Gamma)}^2 + \sum_{z \in \mathcal{N}} \left [ |\kappa|_{H^\sigma(\omega_z)}^2 + |\pi_h \kappa|_{H^\sigma(\omega_z)}^2 \right ] + \|\kappa\|_{H^\sigma(\Gamma)}^2
\\ & \lesssim \|\kappa\|_{H^\sigma(\Gamma)}^2.  
\end{aligned}
\end{equation*}
This is the asserted global stability \eqref{Eq: interp_stability} of $\pi_h$ for $0 <\sigma \le 1$. 

Next we prove \eqref{Eq: interp_stability} for $\pi_h$ and $1<\sigma<\frac{3}{2}$.  Applying $\pi_h v_h =v_h$ for $v_h \in \mathbb{T}_h$, \eqref{fine_localization}, \eqref{Eq: interp_stability} for $I_h$, an element-wise inverse estimate, \eqref{Eq: interp_stability} for $\pi_h$ for $\sigma=0,1$, and \eqref{Eq: camacho-demlow} with $\sigma \leftarrow 1$ and $\xi \leftarrow \sigma$ we obtain
\begin{equation}
\begin{aligned}
\| \pi_h v \|_{H^\sigma(\Gamma)}^2 & \lesssim \|\pi_h (v-I_h v)\|_{H^\sigma(\Gamma)}^2 + \|I_h v\|_{H^\sigma(\Gamma)}^2
 \\ & \lesssim \|\pi_h (v-I_h v)\|_{L_2(\Gamma)}^2 + h^{-2(\sigma-1)}\|\pi_h (v-I_h v)\|_{H^1(\Gamma)}^2 
 \\ & ~~~~~+ \sum_{T \in \mathcal{T}_h^\Gamma} |\pi_h(v-I_h v)|_{H^\sigma(T)}^2+ \|v\|_{H^\sigma(\Gamma)}^2
 \\ & \lesssim \|v-I_h v\|_{L_2(\Gamma)}^2 + h^{-2(\sigma-1)} \|v-I_h v\|_{H^1(\Gamma)}^2 + \|v\|_{H^\sigma(\Gamma)}^2  \lesssim \|v\|_{H^\sigma(\Gamma)}^2. 
\end{aligned}    
\end{equation}
This ends the proof.
\end{proof}

\begin{lemma}\label{frac_approximation2}
Let  $0<\xi \leq 2$, $0\leq \sigma <3/2$ and assume that $\kappa \in H^\xi_h(\Gamma)$.
For $0\leq \sigma\leq \xi$ there holds
\begin{equation} \label{approx_smaller}
    \|\kappa-\pi_h \kappa\|_{H^\sigma(\Gamma)} \lesssim h^{\xi-\sigma} |\kappa|_{H_h^\xi(\Gamma)}.
\end{equation}
Instead, when $1\leq \sigma \leq \xi$, we have
\begin{equation}\label{approx_top}
    \|\kappa-I_h \kappa\|_{H^\sigma(\Gamma)} \lesssim h^{\xi-\sigma} | \kappa|_{H_h^\xi(\Gamma)}.
\end{equation}
\end{lemma}
\begin{proof}
We start by proving \eqref{approx_top}. 
 When $\sigma=1$ \eqref{approx_top} follows immediately from \eqref{Eq: camacho-demlow}, while when $1<\sigma<\frac{3}{2}$ the estimate \eqref{Eq: camacho-demlow} is combined with \eqref{fine_localization} applied to $u=\kappa-I_h \kappa$ along with the result for the case $\sigma=1$. The first two terms are handled by the case $\sigma=1$ and the third term using \eqref{Eq: camacho-demlow}. This completes the proof of \eqref{approx_top}.

We now prove \eqref{approx_smaller}.
The case $\sigma=0$ is \eqref{approx_smaller_sigma0}. For $0<\sigma \le \xi \le 1$, let $C_z$ be the constant Sobolev polynomial over the star-shaped domain $\omega_z$, yielding $|\kappa-C_z|_{H^\sigma(\omega_z)} \lesssim h^{\xi-\sigma} |\kappa|_{H^\xi(\omega_z)}$ (cf. \cite[Theorem 6.1]{DS80}).  With this same choice of $C_z$, \cite[Theorem 7.1]{DS80} yields $\|\kappa-C_z\|_{L_2(\omega_z')} \lesssim h^{\xi}|\kappa|_{H^\xi(\omega_z')}$.  Applying these observations, \eqref{faermann_localization}, \eqref{approx_smaller} for $\sigma=0$,  scaling and inverse inequalities, local $L_2$ stability of $\pi_h$, and the finite overlap of the patches $\omega_z'$ then yields
\begin{equation}
\begin{aligned}
\|\kappa-\pi_h \kappa\|_{H^\sigma(\Gamma)}^2 & \lesssim h^{-2\sigma} \|\kappa-\pi_h \kappa\|_{L_2(\Gamma)}^2 + \sum_{z \in \mathcal{N}} |(\kappa-C_z)-\pi_h(\kappa-C_z)|_{H^\sigma(\omega_z)}^2
\\ & \lesssim h^{2(\xi-\sigma)} |\kappa|_{H^\xi(\Gamma)} + \sum_{z \in \mathcal{N}} \left [|\kappa-C_z|_{H^\sigma(\omega_z)}^2 + h^{-2\sigma} \|\kappa-C_z\|_{L_2(\omega_z')}^2 \right] 
\\ & \lesssim h^{2(\xi-\sigma)} |\kappa|_{H^\xi(\Gamma)} + h^{2(\xi-\sigma)} \sum_{z \in \mathcal{N}} \left [ |\kappa|_{H^\xi(\omega_z)}^2 + |\kappa|_{H^\xi(\omega_z')}^2 \right ]
\\ & \lesssim h^{2(\xi-\sigma)} |\kappa|_{H^\xi(\Gamma)}.
\end{aligned}
\end{equation}
This completes the proof of \eqref{approx_smaller} for $\xi\le 1$.  





It remains to show \eqref{approx_smaller} for $1<\xi\leq 2$. We write $\kappa-\pi_h \kappa = (\kappa-I_h \kappa) -\pi_h (\kappa-I_h \kappa)$. If $\sigma<1$ we apply \eqref{approx_smaller} with $\xi=1$ followed by \eqref{approx_top} with $\sigma=1$.  If $\sigma \ge 1$ we apply the triangle inequality, \eqref{Eq: interp_stability} for $\pi_h$ to the resulting second term, and then \eqref{approx_top}. 
\end{proof}

We emphasize that in standard estimates for the Scott-Zhang operator, the sums over triangles on the right hand sides of \eqref{approx_smaller} with $\xi \leq 1$  and \eqref{approx_top} are replaced by corresponding sums over element patches $\omega_T$.  Reduction to an element-wise sum is critical in the current context because $H^\xi(\omega_T)$ may be degenerate when $\xi \ge \frac{3}{2}$.  

We finally state inverse estimates.  The localization results in Proposition \ref{prop: localization} are not sufficient to obtain inverse estimates for global fractional norms by using element-wise scaling arguments as is typically done for integer-order norms, so we instead employ interpolation techniques to obtain the needed results.  Note that we do not obtain (and do not need) local inverse estimates.  

\begin{proposition}[Inverse estimates]\label{p:inverse}
Let $v_h \in \mathbb{T}_h$.  For $0\le t<\sigma<\frac{3}{2}$, there holds
\begin{equation}
    \label{eq: inverse_estimate}
    \|v_h\|_{H^\sigma(\Gamma)} \ \lesssim \ h^{t-\sigma} \|v_h\|_{H^t(\Gamma)}.
\end{equation}
\end{proposition}
\begin{proof}
Denote by $\{\Gamma_j\}_{j=1}^J$ the boundary edges (when $n=2$) or boundary faces (when $n=3$) of $\Gamma$, and let $1 \le t<\sigma<\frac{3}{2}$. By {application of \eqref{fine_localization} to a coarse triangulation associated with $\{\Gamma_j\}^{J}_{j = 1}$ (so that $h \simeq 1$ in \eqref{fine_localization}) },
\begin{equation} \label{eq: global_localization}
\|\kappa\|_{H^\sigma(\Gamma)}^2 \ \lesssim \ \sum_{j=1}^J \|\kappa\|_{H^\sigma(\Gamma_j)}^2, ~~\kappa \in H^\sigma(\Gamma).
\end{equation}
Each face $\Gamma_j$ is a Lipschitz domain in $\mathbb{R}^{n-1}$.  Theorem B.8 of \cite{mclean00} (cf. Theorems 3.18, 3.30, and A.4 for verification of definitions) implies that for $\sigma_0, \sigma_1 \ge 0$ that $(H^{\sigma_0}(\Gamma_j), H^{\sigma_1}(\Gamma_j))_{\theta, 2} = H^\sigma(\Gamma_j)$ with $\sigma =(1-\theta){\sigma_0}+ \theta {\sigma_1}$.  Here $(\cdot, \cdot)_{\theta,2}$ denotes real interpolation via the $K$-method.  Note that $\|v_h\|_{H^\sigma(\Gamma_j)} \lesssim h^{1-\sigma} \|v_h\|_{H^1(\Gamma_j)}$ due to \eqref{fine_localization} and a standard element-wise scaling argument for the term $|v_h|_{H^\sigma(T)}$. We carry out operator interpolation (cf. \cite[Theorem 7.1 in Chapter 6]{DL93}) for the Scott-Zhang interpolant $I_h$ on $\Gamma_j$ with endpoint estimates $\|I_h \kappa \|_{H^\sigma(\Gamma_j)} \lesssim \|\kappa\|_{H^\sigma(\Gamma_j)}$ and $\|I_h \kappa\|_{H^\sigma(\Gamma_j)} \lesssim h^{1-\sigma} \|\kappa\|_{H^1(\Gamma_j)}$ (the latter obtained by combining an inverse estimate with the $H^1$ stability of $I_h$).  This yields $I_h : H^t(\Gamma_j) \rightarrow H^\sigma(\Gamma_j)$ with operator norm $C h^{t-\sigma}$, and recalling that $I_h$ is the identity on $\mathbb{T}_h$ yields the face-wise inverse estimate $\|v_h\|_{H^\sigma(\Gamma_j)} \lesssim h^{t-\sigma} \|v_h\|_{H^t(\Gamma_j)}$, $1\le t<\sigma<\frac{3}{2}$.  Summing over $j=1,...,J$ then completes the proof of \eqref{eq: inverse_estimate} for $1 \le t<\sigma<\frac{3}{2}$.  

Let now $0 \le t<\sigma \le 1$.  Let $\mathcal{N}$ denotes the set of vertices of the polygon or polyhedron $\Gamma$.  We assume that the faces of $\Gamma$ are triangular; if not then we break the non-triangular faces into triangles having vertices lying in $\mathcal{N}_z$ and proceed as below using the expanded set of triangular ``faces''.  For each $z \in \mathcal{N}$ let $\Gamma_z = \cup_{\Gamma_j  : z \in \overline{\Gamma}_j} \overline{\Gamma}_j$, where as above we denote by $\{\Gamma_j\}$ the set of faces of $\Gamma$.   The localization estimate \eqref{faermann_localization} yields
\begin{equation}\label{macro_localization}
\|u\|_{H^\sigma(\Gamma)}^2 \ \le \ \|u\|_{L_2(\Gamma)}^2 + \sum_{z \in \mathcal{N}} |u|_{H^\sigma(\Gamma_z)}^2, ~~u \in H^\sigma(\Gamma),~~0<\sigma\le 1.
 \end{equation}
Since $\Omega$ is a Lipschitz polyhedral domain, $\Gamma_z$ is the graph of a bi-Lipschitz map $\Phi_z$ over a Euclidean domain $\widetilde{\Gamma}_z \subset \mathbb{R}^{n-1}$.  Notice $\widetilde{\Gamma}_z$ is a line segment when $n=2$ and is a polygon when $n=3$, and $\Phi_z$ has piecewise constant Jacobian.  With these properties, $\Phi_z$ also induces a transformation of $\mathbb{T}_h|_{\Gamma_z}$ to a finite element space $\mathbb{X}_{h,z}$ on $\widetilde{\Gamma}_z$ possessing standard approximation and inverse properties.  As above we employ operator interpolation to obtain inverse estimates on $\widetilde{\Gamma}_z$; cf. \cite{KMRPP23} for similar arguments and finer results also valid on non-quasi-uniform meshes. Let $\pi_{h,z}$ be an $L_2$-stable Scott-Zhang interpolant onto $\mathbb{X}_{h,z}$ analogous to \eqref{eq: pihdef}; $\pi_{h,z}$ is then stable on $H^\sigma(\Gamma)$ for $0 \le \sigma \le 1 $ as with $\pi_h$, see Lemma~\ref{frac_approximation}.  Carrying out operator interpolation for $\pi_{h,z}$ on $\widetilde{\Gamma}_z$ with endpoint estimates $\|\pi_{h,z} u \|_{H^1(\widetilde{\Gamma}_z)} \lesssim \|u\|_{H^1(\widetilde{\Gamma}_z)}$, $\|\pi_{h,z} u\|_{H^1(\widetilde{\Gamma}_z)} \lesssim h^{-1} \|u\|_{L_2(\widetilde{\Gamma}_z)}$ (the latter obtained by combining an inverse estimate with the $L_2$ stability of $\pi_{h,z}$) and recalling that $\pi_{h,z}$ is the identity on $\mathbb{X}_{h,z}$ yields $\|\widetilde{v}_h\|_{H^\sigma(\widetilde{\Gamma}_z)} \lesssim h^{-\sigma} \|\widetilde{v}_h\|_{L_2(\widetilde{\Gamma}_z)}$, $0\le \sigma\le 1$. Repeating the argument with $H^\sigma(\widetilde{\Gamma}_z)$ replacing $H^1(\widetilde{\Gamma}_z)$, we obtain $\|\widetilde{v}_h\|_{H^\sigma(\widetilde{\Gamma}_z)} \lesssim h^{t-\sigma} \|\widetilde{v}_h\|_{H^t(\widetilde{\Gamma}_z)}$, $0 \le t \le \sigma \le 1$. Then, \eqref{macro_localization} and a mapping argument complete the proof of \eqref{eq: inverse_estimate} for $0 \le t <\sigma \le 1$.  The estimate for $t \le 1 \le \sigma$ is obtained by combining the previous two cases.  
\end{proof}

\begin{remark} {We recall that the constant hidden in $\lesssim$ may depend on $s$, and note in particular that the constant hidden in \eqref{eq: inverse_estimate} may degenerate when $s$ or $t$ lie near 0 or 1.  More precisely, the proof of the equivalence between the interpolation and Slobodeckij (double-integral) seminorms employed above eventually relies on \cite[Theorem 3.16]{mclean00}.  Tracing relevant constants in this theorem indicates such a degeneration.  }
\end{remark}




\subsection{Approximation of the Riesz representers - setting}\label{Subsec: ApproxRieszSetting}

The starting point is the partial differential equation \eqref{Eq: RieszOverview} on $\Gamma$  defining the Riesz representer $\phi_i$ associated with the measurement $\lambda_i$.
Recall that the scalar product in $\mathcal H^s(\Omega)$ is given by $\langle \phi_j,v\rangle_{\mathcal H^s(\Omega)} = \langle \phi_j,v\rangle_{H^s(\Gamma)}$ and thus, \eqref{Eq: RieszOverview} can be equivalently rewritten as $\phi_i = E\psi_i$ where $\psi_i \in H^s(\Gamma)$ satisfies
\begin{equation}\label{e:Riesz_bdy}
\langle \psi_i,v\rangle_{H^s(\Gamma)} \ = \ \lambda_i(Ev), \qquad \forall v \in H^s(\Gamma),
\end{equation}
where $E$ is the harmonic extension defined in \eqref{Eq: defE}.

We shall take advantage of the latter relation to design discrete approximations of the Riesz representers. In order to avoid an overload of notation, we use $\nu$ to denote one of the $\lambda_i$, $i=1,...,m$, and drop the index on the Riesz representer.  With these notations, our goal is to approximate $\phi = E \psi \in \mathcal H^s(\Omega)$, where $\psi \in H^s(\Gamma)$ satisfies
\begin{equation}\label{e:psi}
\langle \psi,v \rangle_{H^s(\Gamma)} \ = \ \nu(Ev), \qquad \forall v \in H^s(\Gamma).
\end{equation}

We proceed by introducing the discrete harmonic extension operator $E_h: \mathbb{T}_h \rightarrow \mathbb{V}_h$.  For any $g_h \in \mathbb{T}_h$:
\begin{equation}\label{Eq: Eh}
    E_hg_h \ := \ \argmin\{\|\grad v_h\|_{L_2(\Om; \R^n)} \ | \ v_h \in \mathbb{V}_h, \ T(v_h) = g_h\}.
\end{equation}
Alternatively, $E_h g_h$ is characterized by satisfying $T(E_hg_h) = g_h$ and the discrete harmonic property
\begin{equation}\label{Eq: discHarmonic}
    \int_{\Om}\grad(E_hg_h) \cdot \grad v_h \ = \ 0, \quad v_h \in \mathbb{V}_h^0;
\end{equation}
compare with \eqref{Eq: EgChar}.
We then define $\phi_h := E_h \psi_h \in \mathbb V_h$, where $\psi_h \in \mathbb T_h$ satisfies
\begin{equation}\label{e:psi_h}
\langle \psi_h,v_h \rangle_{H^s(\Gamma)} \ = \ \nu(E_hv_h), \qquad \forall v_h \in \mathbb T_h.
\end{equation}
Lax-Milgram theory guarantees the existence and uniqueness of $\psi_h \in \mathbb T_h$ and thus the existence and uniqueness of $\phi_h \in \mathbb V_h$. We postpone to Subsection \ref{Subsec: Saddle} a discussion on the reformulation of the discrete problem as a more practical saddle-point system.


\section{Regularity assumptions and results}\label{Subsec: regularityAssump}

The order of convergence for finite element approximations of the Riesz representers, and ultimately of the finite element approximation to the optimal recovery, depends on three main types of regularity:
\begin{enumerate}
\item Regularity of solutions to second-order elliptic problems posed on the polyhedral Euclidean domain $\Omega$.
\item Regularity of solutions to potentially fractional-order elliptic problems posed on the polyhedral surface $\Gamma$.
\item Boundedness of the harmonic extension operator $E$.  
\end{enumerate}
In all three cases we shall state or assume multiple types of regularity bounds, so our final convergence estimates involve a relatively complex interplay between a number of parameters.  We now define these parameters,   
the first of which appear when stating an elliptic regularity pick-up result on $\Omega$. 

\begin{assumption}[Regularity pickup on $\Om$]
\label{Assumption: R omega} 
There exists $\frac{1}{2} <\beta^* = \beta^*(\Om) \le \infty $ such that for all $0 < \beta \leq \beta^*$, if $\zeta \in H^{-1 + \be}(\Om)$, then the function $z_{\zeta} \in H^1_0(\Omega)$ satisfying
\begin{equation}\label{Eq: elliptic}
    \int_\Omega \grad z_{\zeta} \cdot \grad v \ = \ \lang \zeta, v\rang,  \quad v\in H^1_0(\Om)
\end{equation}
belongs to $H^{1 + \be}(\Om)$.
Here $\lang \zeta, v\rang$ denotes the $H^{-1}-H^{1}_0$ duality product. We also have the a priori estimate
    \begin{equation}\label{Eq: zVApriori}
        \|z_{\zeta}\|_{H^{1 + \be}(\Om)} \ \lesssim \ \|\zeta\|_{H^{-1+\beta}(\Omega)}.
    \end{equation}
\end{assumption}

For a non-convex polygonal/polyhedral bounded domain, there holds $\frac{1}{2}<\beta^*<1$.  When $\Om$ is smooth or convex we have $\beta^*\geq 1$.   More precisely:
\begin{itemize}
    \item $\beta^* = \infty$ if $\Om$ is smooth  \cite{evans2022partial, gilbarg1977elliptic}.
    \item $\beta^* \ge 1$ if $\Om$ is convex, with $\beta^*=1$ for a generic convex domain {\cite[Theorem 3.2.1.2]{grisvard2011elliptic}} and $\beta^* > 1$ for a convex polygonal or polyhedral domain with the precise value depending on the maximum vertex opening angle ($n=2$) or vertex and edge geometry ($n=3$) {\cite{Da92, dauge2006elliptic}}. 
    \item When $\Om$ is {Lipschitz and} polygonal but non-convex in $\R^2$, or {Lipschitz and} polyhedral but non-convex in $\R^3$, then $\frac{1}{2} < \beta^* < 1$ with the exact value depending on the maximum vertex opening angle ($n=2$; cf. \cite[Theorem 14.6]{dauge2006elliptic}) or vertex and edge geometry ($n=3$; cf. \cite[Theorem 1.1]{Da92}).  
\end{itemize}
Since we employ piecewise linear finite element spaces whose approximation properties cannot take advantage of higher regularity than $H^2(\Omega)$, we shall generally instead use the parameter
$$
\beta_\Omega \ := \ \min(\beta^*, 1) \qquad \textrm{so that} \qquad \frac 1 2 < \beta_\Omega \leq 1.
$$

 We next need to establish embeddings of $\mathcal{H}^\sigma(\Om)$ into $H^r(\Om)$. 

\begin{proposition}[Embedding $\mathcal{H}^{r^*+\frac{1}{2}}(\Om) \subset H^{r^*+1}(\Om)$]\label{Prop: EmbeddingRegularity}
Let $\beta^*$ as in Assumption~\ref{Assumption: R omega}. Let $0<r^*<1$ with $r^* \leq \beta^*$. Then $\mathcal{H}^{r^*+1/2}(\Om) \subset H^{r^*+1}(\Om)$.  In other words, there exists $C_e > 0$ such that
\begin{equation}\label{Eq: embedStable}
    \|Eg\|_{H^{r^*+1}(\Om)} \ \leq \ C_e\|g\|_{H^{r^*+1/2}(\Gamma)}
\end{equation}
holds for any $g \in H^{r^*+1/2}(\Gamma)$.
\end{proposition}

\begin{proof}
For $0<r^* \le \frac{1}{2}$ the desired result is stated for example in \cite{jerison1995inhomogeneous} for general Lipschitz domains.

For $\frac{1}{2}<r^*<1$ we instead employ properties of the trace operator on Lipschitz polygons and polyhedra.  Recall that $H_T^{r^*+1/2}(\Gamma)=T(H^{r^*+1}(\Omega))$ is the image of $H^{r^*+1}(\Omega)$ under the trace operator $T$ defined in \eqref{Eq: TraceDef}.  We wish to establish that $H_T^{r^*+1/2}(\Gamma)=H^{r^*+1/2}(\Gamma)$, where as throughout this work the latter space is defined as those functions with finite intrinsic norm given by \eqref{Eq: Hr}.  

The space $H_T^{r^*+1/2}(\Gamma)$ is characterized for general polygonal and polyhedral domains in \cite{BDM00,BDMPP07}, and a similar characterization is given for the more restricted case of triangles and tetrahedra in \cite{PS24, PS25}. There it is shown that $g \in H_T^{r^*+1/2}(\Gamma)$ if and only if $g \in H^{r^*+1/2}(\Gamma_j)$, $j=1,..,J$, and if in addition $g$ satisfies compatibility conditions at vertices of $\Gamma$ (if $n=2$ so $\Omega$ is a polygon) or at edges and vertices of $\Gamma$ (if $n=3$ so $\Omega$ is a polyhedron). 
Here $\{\Gamma_j\}_{j=1}^J$ are the boundary edges (when $n=2$) or boundary faces (when $n=3$) of $\Gamma$.
When $\frac{1}{2}<r^*<1$ these compatibility conditions reduce to requiring continuity of $g$ across vertices and edges of $\Gamma$ (for the case of polygons, see for example the case $n=0$ in \cite[Table 1]{BDM00}).  That is,
\begin{equation*}
H_T^{r^*+1/2}(\Gamma)=C(\Gamma) \cap \Pi_{j=1}^J H^{r^*+1/2}(\Gamma_j), ~~\frac{1}{2}<{r^*}<1.
\end{equation*}
If $g \in H^{r^*+1/2}(\Gamma)$ for $\frac{1}{2}<r^*<1$, then $g \in C(\Gamma)$ by Sobolev embedding, and in addition $g \in H^{r^*+1/2}(\Gamma_j)$, $j=1,...,J$.  Thus $H^{r^*+1/2}(\Gamma) \subset H_T^{r^*+1/2}(\Gamma)$. 

In order to prove the other inclusion, note first that continuous, facewise $H^1$ functions on $\Gamma$ are globally in $H^1$.  Because the $H^{r^*+1/2}(\Gamma)$ norm can be localized to faces (cf. \eqref{eq: global_localization}), we thus obtain $H_T^{r^*+1/2}(\Gamma) = H^{r^*+1/2}(\Gamma)$ with equivalence of norms also holding.  

Having established that $H_T^{r^*+1/2}(\Gamma) = H^{r^*+1/2}(\Gamma)$, we note that $T$ also has a bounded right inverse  $T^{-1}:H_T^{r^*+1/2}(\Gamma) \rightarrow H^{r^*+1}(\Omega)$ \cite[Corollary 5.8 and Corollary 6.10]{BDMPP07}.  Thus because $r^* \le \beta^*$ and $\Delta Eg=0$, we have for $g \in H^{r^*+1/2}(\Gamma)$ that
\begin{equation}
\begin{aligned}
    \|Eg\|_{H^{r^*+1}(\Omega)} &\le \|Eg-T^{-1} g\|_{H^{r^*+1}(\Omega)} + \|T^{-1}g\|_{H^{r^*+1}(\Omega)}
    \\ & \lesssim \|\Delta (Eg-T^{-1} g)\|_{H^{r^*-1}(\Omega)} + \|g\|_{H_T^{r^*+1/2}(\Gamma)}
    \\ & = \|\Delta T^{-1}g\|_{H^{r^*-1}(\Omega)} + \|g\|_{H_T^{r^*+1/2}(\Gamma)}
    \\ & \le \|T^{-1} g \|_{H^{r^*+1}(\Omega)} + \|g\|_{H_T^{r^*+1/2} (\Gamma)}
    \\ & \lesssim \|g\|_{H_T^{r^*+1/2} (\Gamma)} \lesssim \|g\|_{H^{r^*+1/2}(\Gamma)}.
    \end{aligned}
\end{equation}
This completes the proof.  
\end{proof}

We also need to measure the regularity of the functionals $\lambda_i:H^s(\Gamma) \rightarrow \mathbb{R}$ given by $\lambda_i(g) = Eg(x_i)$, $x_i \in \Omega$. As for the regularity pick-up, we state this as an assumption and then prove that the assumption holds for at least some parameters in particular cases.

\begin{assumption} \label{ass: lambda bounded}
There is $\mu \geq 0$ such that given $z \in \Omega$ with ${\rm dist}(z, \partial \Omega) \ge d$ and $g \in H^{1/2}(\Gamma)$, there holds
\begin{equation} \label{Eq: L2 boundary embedding}
|Eg(z)| \ \lesssim \ d^{-n/2-\tilde{\mu}} \|g\|_{H^{-\mu}(\Gamma)}
\end{equation}
where $\tilde{\mu} := \max(0, \mu-\frac{1}{2})$.  Consequently, if $\gamma_z(g)=(Eg)(z)$ denotes the bounded extension of $g \mapsto (Eg)(z)$ to $H^{\mu}(\Gamma)$, we have $\gamma_z \in H^\mu(\Gamma)$ and 
\begin{equation}\label{eq: gamma_reg}
\|\gamma_z\|_{H^\mu(\Gamma)} \ \lesssim \ d^{-n/2-\tilde{\mu}}.
\end{equation}
\end{assumption}

We first establish that this result holds for $\mu=0$ for any domain such that the shift estimate \eqref{Eq: zVApriori} holds for some $\beta>\frac{3}{2}$.  This class of domains includes any polygonal, polyhedral, convex, or smooth domain.

\begin{proposition}[$L_2$ embedding]\label{Prop: EmbeddingL^p}
    Assume that $\Omega$ is a Lipschitz domain such that Assumption~\ref{Assumption: R omega} holds for some $\beta^*>\frac{1}{2}$.  Then for all $g \in H^{1/2}(\Gamma)$ and $z\in \Omega$ with ${\rm dist}(z, \Gamma) \ge d$ we have
    \begin{equation}\label{Eq: embedStableL^p}
        |Eg(z)| \ \lesssim  \ d^{-n/2} \|g\|_{L_2(\Gamma)}.
    \end{equation}
    Here the hidden constant only depends on $\Om$.  
\end{proposition}
\begin{proof}
    Let $w \in H_0^1(\Omega)$ solve \eqref{Eq: elliptic} with right-hand side $Eg$. 
   Integrating by parts twice and $\Delta Eg =0$ yield
    \begin{equation}\label{Eq: funIBP}
        \|Eg\|_{L_2(\Omega)}^2 = \int_{\Om}Eg (-\Delta w) dx \ = \ \int_{\Om}\grad Eg \cdot \grad wdx - \int_{\Gamma}g\frac{\pa w}{\pa \nu}d\sig \ 
        = \ -\int_{\Gamma}g \frac{\pa w}{\pa \nu}d\sig \leq \ \|g\|_{L_2(\Gamma)}\Big|\Big|\frac{\pa w}{\pa \nu}\Big|\Big|_{L_2(\Gamma)},
    \end{equation}
    where $\frac{\pa w}{\pa \nu}$ stands for the outer normal derivative of $w$. Using that $\|\nu\|_{L_\infty(\Gamma)} = 1$, applying the continuity of the trace operator for $0<\eta  \le \max(\beta^*,1)-\frac{1}{2}$ 
    (cf. \cite[Theorem 1.5.1.2]{grisvard2011elliptic}) and then employing \eqref{Eq: zVApriori} gives
    \begin{equation}\label{Eq: embed}
        \Big\|\frac{\partial w}{\partial \nu}\Big\|_{L_2(\Gamma)} \le \| \nabla w \|_{L_2(\Gamma)} \le \|\nabla w\|_{H^\eta(\Gamma)}  \ \leq \ \|\grad w\|_{H^{\eta + 1/2}(\Om)} \ \leq \ \|w\|_{H^{\eta + 3/2}(\Om)} \ \lesssim \ \|w\|_{H^{1 + \max(\beta^*,1)}(\Om)},
    \end{equation}
    and so combining \eqref{Eq: zVApriori} and \eqref{Eq: embed} in \eqref{Eq: funIBP}, we obtain
    \begin{equation} \label{eq: EL2stable}
        \|Eg\|_{L_2(\Omega)} \ \lesssim \ \|g\|_{L_2(\Gamma)}.
    \end{equation}
    The desired result follows by combining the last inequality with the mean value property for harmonic functions:
    \begin{equation}
        |Eg(z)| = \frac{1}{|B_d(z)|} \int_{B_d(z)} Eg(x) dx \lesssim d^{-n/2} \|Eg\|_{L_2(\Omega)} \ \lesssim \ d^{-n/2} \|g\|_{L_2(\Gamma)},
    \end{equation}
    where $B_d(z)$ is the open ball centered at $z$ and of radius $d$.
    \end{proof}

  We next show that Assumption \ref{ass: lambda bounded} holds for arbitrarily large $\mu>0$ when $\Omega$ is smooth, at least when $\mu+\frac{1}{2}$ is an integer.

  \begin{proposition} \label{prop: gammasmooth} Assume that $\Gamma=\partial \Omega$ is of class $C^\infty$ and $g \in H^{1/2}(\Gamma)$.  Then for $\mu \ge \frac{1}{2}$ with $\mu+\frac{1}{2}$ an  integer and $z \in \Omega$ with ${\rm dist}(z, \Gamma) \ge d$, there holds
  \begin{equation*}
  |Eg(z)| \ \lesssim \ d^{-n/2 - \mu+ 1/2} \|g\|_{H^{-\mu}(\Gamma)}.
  \end{equation*}  
  \end{proposition}
\begin{proof}  Let $G$ be the Green's function for $\Omega$ and $\nu$ denotes the outward pointing normal to $\Omega$.  Then
\begin{equation}
Eg(z) \ = \ -\int_\Gamma g(y) \frac{\partial_y G(y,z)}{\partial \nu} dy \le \|g\|_{H^{-\mu}(\Gamma)} \left \|\frac{\partial G(\cdot, z)}{\partial \nu} \right \|_{H^\mu(\Gamma)}.
\end{equation}
Let now $\varphi$ be a cutoff function satisfying $\varphi|_{B_{d/2}(z)}\equiv 0$, $\varphi|_{\overline{\Omega} \setminus B_d(z)} \equiv 1$, and $|D^\alpha \varphi| \lesssim d^{-|\alpha|}$ for multi-indices $|\alpha|\ge 0$.  Using the smoothness of the normal $\nu$ to $\Gamma$ and standard trace theorems, we have
\begin{equation}\label{Eq: stTrace}
\left \|\frac{\partial G(\cdot, z)}{\partial \nu} \right \|_{H^\mu(\Gamma)} \lesssim \|\nabla G(\cdot, z)\|_{H^\mu(\Gamma)} \ = \ \|\varphi \nabla G(\cdot, z)\|_{H^\mu(\Gamma)} \lesssim \|\varphi \nabla G(\cdot, z)\|_{H^{\mu+1/2}(\Omega)}.  
\end{equation}
Standard estimates for the Green's function \cite{Kr69} give
\begin{equation} 
D_y^\alpha \nabla_y G(y,z) \ \lesssim \ |y-z|^{1-n-|\alpha|}, ~|\alpha| \ge 0.
\end{equation}
Translating to polar coordinates, we can compute 
\begin{equation}\label{Eq: translaPolarCoord}
|\nabla G(\cdot, z)|_{H^{\mu+1/2}(\Omega \setminus B_{d/2}(z))}^2 \ \lesssim \ \int_{d/2}^{{\rm diam}(\Omega)} (r^{1-n-\mu-1/2})^2 r^{n-1} dr \lesssim d^{-n-2\mu+1}.
\end{equation}
Employing Leibniz's Rule and $|D^\alpha \varphi| \lesssim d^{-|\alpha|}$ similarly yields
\begin{equation}\label{Eq: TraceLeibniz}
\|\varphi \nabla G(\cdot, z)\|_{H^{\mu+1/2}(\Omega)} \ \lesssim \ d^{-n/2-\mu+1/2}, 
\end{equation}
which yields the desired result when combined with the preceding equations. 
\end{proof}

\begin{remark}\label{Rmk: MuDiscussion}
We have established that Assumption 5.3 holds with $\mu=0$ for a broad class of domains (including Lipschitz polygonal and polyhedral domains) and that it holds for arbitrarily large $\mu$ when $\Omega$ is smooth.  A natural question is whether this assumption holds for some positive $\mu>0$ when $\Omega$ has intermediate smoothness.  One can quickly establish a positive answer if $\Omega$ is of class $C^{k,1}$ for some integer $k >0$, but when $\Omega$ is polygonal or polyhedral and thus merely of class $C^{0,1}$ the answer is more subtle.  Let for example $\Omega \subset \mathbb{R}^2$ be polygonal.  A critical step in the proof of Proposition \ref{Prop: EmbeddingL^p} is the bound \eqref{Eq: embed}.  In order to emulate this step for $\mu>0$ we would need to similarly bound $\left \| \frac{\partial z}{\partial \nu} \right \|_{H^{\mu}(\Omega)}$, but $\partial z/\partial \nu= \nabla z \cdot \nu$ is not generally in $H^\mu(\Gamma)$ when $\Omega$ is merely Lipschitz because $\nu$ is only in $[L_\infty(\Gamma)]^n$.  When $\Omega \subset \mathbb{R}^2$ is polygonal with edges $\{\Gamma_i\}_{i=1}^M$, Grisvard \cite[Theorem 1.2.1]{grisvard2011elliptic} has established that if $z \in W^{2, p}(\Omega)$ then $\frac{\partial z}{\partial \nu} \in W_p^{1-1/p}(\Gamma_j)$, $j=1,...,J$. The question of whether this additional piecewise regularity would ultimately allow us to systematically establish higher order convergence rates for our recovery algorithm on general Lipschitz polyhedral domains does not have a clear answer, especially when $s \neq 1$, and we do not pursue it further.  {If on the other hand $\Omega$ is convex and polygonal or polyhedral we may apply \cite[Theorem 2.1.65]{vexler2025numerical}, which states that a function lying in  $H^2(\Om) \cap H^1_0(\Om)$ has normal trace with global $H^{1/2}(\Gamma)$ regularity on any polygonal or polyhedral domain.  Since the dual solution $w$ in \eqref{Eq: funIBP}--\eqref{Eq: embed} possesses $H^2$ regularity on convex polygonal and polyhedral domains, employing this observation allows us to take $\mu = \frac{1}{2}$ in this case.}  
\end{remark}





We next consider regularity of solutions to (possibly fractional order) elliptic problems on the boundary $\Gamma$ of the form: For $\gamma \in H^{-s}(\Gamma)$, find $\kappa \in H^s(\Gamma)$ such that 
\begin{equation}\label{Eq: hsvargamma}
\lang \kappa, g\rang_{H^s(\Gamma)} \ = \ \gamma(g),\quad g\in H^s(\Gamma).
\end{equation}
The smoothness of the boundary $\Gamma$ and the functional $\gamma$ will generally both impact the smoothness of $\kappa$.  As noted in the predecessor work \cite{binev2024solving}, there seems to be relatively little literature quantifying expected regularity gains in the context of fractional-order problems.  
In the case of the classical Laplace-Beltrami problem ($s=1$) we may in contrast make some concrete statements.  

\begin{remark}[Regularity of solutions to the Laplace-Beltrami problem on polyhedral surfaces] \label{rem: lbregularity}
When $\Gamma$ is polygonal or polyhedral, we cannot expect global $H^\sigma$ regularity when $\sigma \ge \frac{3}{2}$ because these spaces are degenerate across vertices, but in many circumstances {\it facewise} $H^r$ regularity of $\kappa$ with $r \ge \frac{3}{2}$ does hold.  Such facewise regularity is sufficient to achieve higher-order approximation by finite element methods.  

When $n=2$ and thus $\Gamma$ is a piecewise affine closed Lipschitz curve, the weak solution $\kappa \in H^1(\Gamma)$ to \eqref{Eq: hsvargamma} is continuous by a Sobolev embedding.  Denote by $\{\Gamma_j\}_{j=1}^J$ the faces of $\Gamma$.  On each face $\kappa$ thus solves a classical two-point boundary value problem $-\kappa''+\kappa=f$ with right hand side data $f \in H^{-1+\delta}(\Gamma_j)$, $\delta \ge 0$,  {where $f$ is the restriction of the functional $\gamma$ in \eqref{Eq: hsvargamma} to $\Gamma_j$.} Thus $\kappa \in H^1(\Gamma)$ and $v|_{\Gamma_j} \in H^{1+\delta}(\Gamma_j)$, $j=1,...,J$.   

The situation is more complex when $n=3$.  From \cite[Theorem 8]{BCS02} there again holds that $\kappa \in H^1(\Gamma)$ with $\kappa$ possibly more regular on each face $\Gamma_i$.  However, now the facewise regularity depends on the geometry of $\Gamma$ around each vertex in addition to the smoothness of the right hand side.  Let $L$ be the maximum length of the perimeter of the spherical caps of radius 1 covering the vertices $v$ of $\Gamma$.   If $f \in H^{-1+\delta}(\Gamma)$, then for each face $\Gamma_j$ there holds $\kappa|_{\Gamma_j} \in H^\sigma(\Gamma_j)$ with $\sigma<{\rm min}(1+\delta, 1+\frac{2\pi}{L})$.  If $\Omega$ is convex then it is easy to see that $L < 2 \pi$ and thus $ \kappa|_{\Gamma_j} \in H^{1+\delta-\epsilon}(\Gamma_j)$ for any $\epsilon>0$ and $\delta \le 1$. When $\Omega$ is not convex, then regularity may be more restricted.  For example, the spherical cap for the central vertex of a classical two-brick domain has perimeter $L=3\pi$ (cf. \cite[pg. 19]{DG12} for an illustration).  Thus even if the right hand side data is infinitely smooth the facewise regularity of $\kappa$ is restricted to $H^{5/3-\epsilon}$.  A pathological example in \cite{BCS02} indicates that polyhedral domains $\Omega$ can be constructed with $L$ arbitrarily large, and thus the facewise regularity of $\kappa$ may be only $H^{1+\epsilon}$ with $\epsilon$ arbitrarily small.  
\end{remark}

With these considerations we make the following assumption where $\|\kappa\|_{\tilde{H}^\xi(\Gamma)}$ denotes the standard $H^\xi(\Gamma)$ norm if $\xi<\frac{3}{2}$, and $(\sum_{j=1}^J \|\kappa\|_{H^\xi(\Gamma_j)}^2)^{1/2}$ if $\xi \ge \frac{3}{2}$ with $\{ \Gamma_i\}_{i=1}^J$ denoting the faces (or edges when $n=2$) of $\Gamma$.

\begin{assumption}[Regularity pickup on $\Gamma$]\label{Assumption: R} For $\frac{1}{2} <s<\frac{3}{2}$ and $\delta>0$, there exists $0<R(s,\delta) \le 2-s$ such that
if $\gamma\in H^{-s+\delta}(\Gamma)$, 
then the solution $\kappa$ to \eqref{Eq: hsvargamma} belongs to ${\tilde{H}^{s+R(s,\delta)}(\Gamma)}$ with
\begin{equation}\label{Eq: stabetas}
\|\kappa\|_{\tilde{H}^{s+R(s,\delta)}(\Gamma)} \ \preceq \ \|\gamma\|_{H^{-s+\delta}(\Gamma)}.
\end{equation}
Moreover, there holds
\begin{equation}\label{Eq: approxetas}
\min_{{\chi} \in {\mathbb T}_h} \|\kappa-{\chi}\|_{H^s(\Gamma)} \ \preceq \ h^{R(s,\delta)}  \|\kappa\|_{\tilde{H}^{s+R(s,\delta)}(\Gamma)} \ \preceq \ h^{R(s,\delta)} \|\gamma\|_{H^{-s+\delta}(\Gamma)}.
\end{equation}
The constants hidden in the $\preceq$ symbol depends on $s$, $\delta$,  and the shape parameter $\sigma$.
\end{assumption}

Lax-Milgram theory guarantees that the solution $\kappa$ in \eqref{Eq: hsvargamma} is bounded in $H^s(\Gamma)$, with $\|\kappa\|_{H^s(\Gamma)}=\|\gamma\|_{H^{-s}(\Gamma)}$. Assumption~\ref{Assumption: R} provides regularity beyond $H^s(\Gamma)$ provided the data is smoother. {For smooth domains, it is natural to expect that $R(s, \de) = \min(\de, 2-s)$, although this statement does appear directly in the literature. Nevertheless, by reproducing the argument of Theorem~4.2 in~\cite{BP17}, one could show that the domain of $(-\Delta_\Gamma)^r$ embeds into the interpolation space $H^{2r}(\Gamma)$ and deduce the desired regularity estimates.} For polygonal / polyhedral Lipschitz domains, we expect $R(s,\de)>0$.  The restriction $R(s,\delta) \le 2-s$ reflects the inability of piecewise linear finite element spaces to take advantage of higher regularity.


An important step in our argument below is bounding norms of harmonic extensions of solutions $\kappa$ to elliptic problems on $\Gamma$.  As quantified in the following proposition, the global $H^\sigma$ regularity of $E \kappa$ on $\Omega$ is constrained both by the regularity of $\kappa$ and by elliptic regularity on $\Omega$.  

\begin{proposition}[Bounds for harmonic extensions of solutions to elliptic problems on $\Gamma$]  
Let $\beta_\Omega=\min(\beta^*,1)$ with $\beta^*$ as in Assumption~\ref{Assumption: R omega} and $R(s,\delta)$ as in Assumption~\ref{Assumption: R} for $\delta>0$. Consider $E\kappa$, where $\kappa$ solves \eqref{Eq: hsvargamma} with $\gamma \in H^{-s+\delta}(\Gamma)$ and let 
\begin{equation} \label{eq: rdef}
    r(\delta)\ :=\ \min \left (\beta_\Omega, s+R(s,\delta)-\frac{1}{2}, 1-\epsilon \right )
\end{equation}
for some fixed (small) $\epsilon>0$.  Then
\begin{equation} \label{eq: kappa_est}
\|E\kappa\|_{H^{r(\delta)+1}(\Omega)} \ \lesssim \|\gamma\|_{H^{-s+\delta}(\Gamma)}.
\end{equation}
\end{proposition}
\begin{proof}
    This result directly follows from Proposition~\ref{Prop: EmbeddingRegularity} with $r^*=r(\delta)$ and \eqref{Eq: stabetas}.
\end{proof}    

Finally, we state assumptions concerning the regularity of the solution $u_0$ to $-\Delta u_0=f$ with $u_0=0$ on $\Gamma$.  We have already noted that in order to make sense of interior point measurements we require $u_0 \in C(\Omega)$ (and thus obtain $u \in C(\Omega)$ according to \eqref{eq:decomp}).  We thus require that $f \in H^{{-1+\beta_\Omega}}(\Omega)$ for $\beta_\Omega > \max(0,n/2-1)$.


Beyond ensuring that $u_0 \in C(\Omega)$, we also need assumptions that {guarantee} the convergence rates of finite element approximations of $u_0$ in order to ensure that  \eqref{Eq: NROAlgEq0B} in the recovery algorithm holds, namely
$$
|u_{h0} - u_0|_{X} +\max_{i=1,...,m} |u_{h0}(x_i)-u_0(x_i)| \ \leq \ \ep_1.
$$
The types of regularity needed and thus the assumptions we make here depend on the choice of  $(X,|.|_X)$ used to measure the recovery error as well as the need to guarantee that $\max_{i=1,...,m} |u_{0h}(x_i)-u_0(x_i)|$ is small.  We assume that Assumption~\ref{Assumption: R omega} is satisfied for some $\beta^*$ and recall that $\beta_\Omega = \min(1,\beta^*)$.

For the approximation of $u_0$ at the measurement locations, recall that $x_i \in \Omega_d$ and local (interior) pointwise error estimates \cite{schatz1995interior} yield  for any $z \in \overline{\Omega}_d:= \{ x \in \Omega \ | \ {\rm dist}(x, \partial \Omega) \geq d\}$ that
\begin{equation} \label{Eq: u0 interior}
|(u_0-u_{h0})(z)| \ \lesssim \ \ln \left(\frac{d}{2h}\right) \inf_{\chi \in \mathbb{V}_h^0} \|u_0-\chi\|_{L_\infty (B_{d/2}(z))} + d^{-n/2} \|u_0-u_{h0}\|_{L_2(\Omega)}.
\end{equation}
Here we have assumed that $d \ge c_0 h$ with $c_0$ sufficiently large.  A standard duality argument yields for $0<\beta_f\leq \beta_\Omega$ and $f \in H^{-1+\beta_f}(\Omega)$
\begin{equation} \label{Eq: u0 L2}
\|u_0-u_{h0}\|_{L_2(\Omega)} \ \lesssim \ h^{\beta_\Omega+\be_f} \|f\|_{H^{-1+\be_f}(\Omega)}.
\end{equation}
When $u_0$ is sufficiently smooth, standard finite element approximation theory yields
\begin{equation*} 
\inf_{\chi \in \mathbb{V}_h^0} \|u_0-\chi\|_{L_\infty (B_{d/2}(z))} \ \lesssim \ h^{\tau_d} |u|_{C^{\ell, {\tau_d - \ell} }(\Omega_d/4)},
\end{equation*} 
where $\Omega_{d/4} = \{x \in \Omega: {\rm dist}(x, \partial \Omega)\}> d/4$.  Here $0< \tau_d \le 2$, {$\ell = 0$ when $0 < \tau_d \le 1$, and $\ell=1$ when $1< \tau_d \le 2$}.  The range of $\tau_d$ values allowed depends on interior H\"older regularity of $u_0$, which in turn depends only on the regularity of $f$.  Such regularity may for example be obtained for $0<\tau_d<1$ by assuming $f \in L_q(\Omega)$ for appropriate $q>\frac{n}{2}$ \cite[Theorem 8.24]{gilbarg1977elliptic}, for $1 < \tau_d <2$ by assuming $f \in L_{\infty}(\Omega)$ \cite[Section 8.11]{gilbarg1977elliptic}, and for $\tau_d=2$ by assuming $f \in C^{0,\alpha}(\Omega)$ for any $\alpha>0$ \cite[Chapter 4]{gilbarg1977elliptic}.  Such results could alternatively be obtained by employing interior $W_p^\sigma$ regularity results and then using Sobolev embeddings to obtain appropriate H\"older continuity of $u_0$.  Under these assumptions, we finally obtain
\begin{equation} \label{Eq: u0 point approx}
|(u_0-u_{h0})(z)| \ \le \ 
\left(h^{\tau_d} \ln \left(\frac{d}{2h}\right) + d^{-n/2} h^{\beta_\Omega+\be_f}\right) C_f,
\end{equation}
where $C_f$ depends on $f$.

We now address the approximation of $u_0$ in $(X,|.|_X)$ for the three cases considered.

{\bf Case 1:} $X=H^1(\Omega)$ and $|.|_X=\|.\|_{H^1(\Omega)}$. Standard finite element theory and \eqref{Eq: zVApriori} guarantee that for $\max(-1,n/2-2) < \beta_f \leq \beta_\Omega$ and $f \in H^{\beta_f}(\Omega)$, we have
\begin{equation}\label{Eq:u0H1}
\|u_0-u_{h0}\|_{H^1(\Omega)} \ \lesssim \ h^{\be_f} \|f\|_{H^{-1+\be_f}(\Omega)}.
\end{equation}

{\bf Case 2:} $X=L_\infty(\Omega)$ and $|.|_X=\|.\|_{L_\infty(\Omega)}$.   For any polygonal domain when $n=2$ \cite{Sch80} and for convex polyhedral domains when $n=3$ \cite{LV16}, it is known that
\begin{equation*}
    \|u_0-u_{h0}\|_{L_\infty(\Omega)} \ \lesssim \ \ln \left(\frac{1}{h}\right) \inf_{\chi \in \mathbb{V}_h^0} \|u-\chi\|_{L_\infty(\Omega)}. 
\end{equation*}
We are unaware of similar results in the literature for nonconvex polyhedral domains in three space dimensions.  Accordingly, we assume that there exists $0<\tau_\Omega \le 2$ such that
\begin{equation} \label{Eq: linf global}
\|u_0-u_{h0}\|_{L_\infty(\Omega)} \ \le \ C_{f, \Omega}\ln \left(\frac{1}{h}\right) h^{\tau_\Omega}. 
\end{equation}
The allowed range of values of $\tau_\Omega$ will depend on the regularity of $f$ as in Case 2, plus the geometry of $\Omega$, which when considering polygonal domains in two space dimensions corresponds to the maximum interior vertex opening angle of $\Omega$ (cf. \cite{SW78}).  

{\bf Case 3:} $X=\{ v\in L_2(\Omega) \ | \ v|_{\Omega_d}\in L_\infty(\Omega_d)\}$ and $|.|_X=\|.\|_{L_\infty(\Omega_d)}$. The approximation of $u_0$ by $u_{h0}$ already follows from \eqref{Eq: u0 point approx}.

\section{Error estimates on Riesz representers}\label{Sec: Mainresult}








Before stating and proving our main results, we make some comments. For sake of notational simplicity we  suppress the dependence of the Riesz representers on $i$, letting $\psi \in H^s(\Gamma)$ satisfy 
\begin{equation} \label{Eq: psi}
    \lang \psi, g\rang_{H^s(\Gamma)} \ = \ \nu(Eg) = Eg(\overline x), \quad g\in H^s(\Gamma)
    \end{equation}
for some $\overline x\in \{x_1,...,x_m\}$.  Also, $\phi = E \psi$.  Analogously, we let $\phi_h \in \mathbb{V}_h$ be the finite element approximation of $\phi$; that is, $\phi_h = E_h\psi_h$, where $\psi_h \in \mathbb{T}_h$ is the solution to
\begin{equation}\label{Eq: psiH}
    \lang \psi_h, g_h\rang_{H^s(\Gamma)} \ = \  E_hg_h(\overline x), \quad g_h \in \mathbb{T}_h.
\end{equation}



\subsection{Pointwise error estimates}\label{Subsec: pointwiseError}

Now we may state the first major theoretical result of our paper, which is a variant of \cite[Theorem 4.7]{binev2024solving}. That paper established both $H^1$ and (global) $L_{\infty}$-error estimates for the Riesz representers in the case of pointwise evaluations. Here we refine the latter estimate to make explicit the dependence of the error on the distances between the measurement locations and $\Gamma$.





\begin{theorem}\label{Thm: theorateinf}
Assume that $\frac 1 2 < s < \frac 32$. Let $\beta^*$ be as in Assumption~\ref{Assumption: R omega} and set $\beta_\Omega=\min(\beta^*,1)$. 
Let also $\mu, \tilde{\mu} \ge 0$ be as in Assumption \ref{ass: lambda bounded}, $R(.,.)$ as in Assumption \ref{Assumption: R}, and $r(.)$ be given by \eqref{eq: rdef}. 
 Let $d \geq c_0 h$ with $c_0>0$ sufficiently large. For $z \in \Om$ with ${\rm dist}(z, \Gamma) \ge d$, if $t_p := \min\{2R(s, s+\mu), r(s+\mu) + \beta_\Omega \}$, there holds
\begin{equation}
\label{Eq: convphiboundinf}
 |(\phi - \phi_h)(z)| \ \lesssim \ h^{t_p}.
\end{equation}
The implicit constant depends on $s$, the geometry of $\Omega$ (in particular through the constant $c_0$), $d$, and the elliptic regularity parameter $\beta_\Omega$. More explicitly, 
\begin{equation}\label{Eq: Claim14DiscReprise}
    |(\phi-\phi_h)(z)| \ \lesssim \ d^{-n-2\tilde{\mu}} h^{2R(s, s+\mu)} +  d^{-n-\tilde{\mu}} h^{r(s+\mu) + \beta_\Omega}.
\end{equation}
\end{theorem}
\begin{remark}[Expected orders of convergence] Before proving Theorem \ref{Thm: theorateinf} we explain how the parameters $R(s,s+\mu)$, $\beta_\Omega$, and $r(s+\mu)$ may interact in the concrete case $s=1$ corresponding to the classical Laplace-Beltrami operator on $\Gamma$.  When $\Omega$ is convex, we have $\beta_\Omega=1$. Taking $\mu=0$, per Remark \ref{rem: lbregularity} we can take $R(s,s+\mu)=R(1,1)=1-\frac{\epsilon}{2}$ for any $\epsilon>0$. Finally, we have $r(s+\mu) =\min(\beta_\Omega, s+R(s,s+\mu)-\frac{1}{2}, 1-\epsilon)=\min(1,\frac{3}{2}-\frac{\epsilon}{2}, 1-\epsilon)=1-{\epsilon}$ for any $\epsilon>0$.  Thus our overall convergence rate is $d^{-n}h^{2R(s,s+\mu)}+h^{r(s+\mu)+\beta_\Omega}=d^{-n}h^{2-\epsilon}$ for any $\epsilon>0$.  When $\Omega$ is non-convex $\beta_\Omega<1$.  Per the discussion above there still always holds $R(1,1)=1$ when $n=2$. When $n=3$, the value of $R(1,1)$ varies with the geometry of $\Gamma$ and is typically less than but not arbitrarily close to 1.  The expected convergence rate is then at most $d^{-n}h^{2\beta_\Omega}$.  

Much less information relevant to predicting expected rates of convergence is available in the literature when $s \neq 1$.  It is however the case that $R(s,s+\mu) \le 2-s$, so the predicted convergence rate when $s>1$ is at most $2(2-s)<2$.  When $s<1$, the convergence rate is still limited by 2, but the effect of the term $R(s,s+\mu)$ is less clear as we are not aware of references establishing regularity of fractional order problems on polyhedral surfaces.  


\end{remark}

\begin{proof}

Let $0<r<1$ with $r \le \beta_\Omega$, and let $g_h \in \mathbb{T}_h \subset H^{r+\frac{1}{2}}(\Gamma)$ be arbitrary.

\textbf{Step 1:} We first prove that
\begin{equation}\label{Eq: Claim3ADisc}
    \|(E - E_h)g_h\|_{L_2(\Om)} \ \lesssim \ h^{r + \beta_\Omega}\|g_h\|_{H^{r+\frac{1}{2}}(\Gamma)}.
\end{equation}
Using C\'ea's Lemma, $(P_h Eg_h)|_\Gamma=g_h$ according to \eqref{Eq: szproj}, standard approximation properties, and \eqref{Eq: embedStable} yields
\begin{equation}\label{Eq: Claim2Disc1}
\begin{aligned}
    \|\nabla(E - E_h)g_h\|_{L_2(\Om)} & \lesssim \ \min_{v_h \in \mathbb{V}_h, v_h|_\Gamma=g_h}\|Eg_h - v_h\|_{H^1(\Om)} \le \|Eg_h-P_h Eg_h\|_{H^1(\Omega)} \\ & \lesssim h^r\|E g_h \|_{H^{r+1}(\Om)} \lesssim h^r \|g_h\|_{H^{r+\frac{1}{2}}(\Gamma)}.
    \end{aligned}
\end{equation}
 We may then use \eqref{Eq: Claim2Disc1} and a standard duality argument involving \eqref{Eq: elliptic} to prove  \eqref{Eq: Claim3ADisc}. 

\textbf{Step 2:} Let $x \in \Omega$ be fixed and let $\rho > 0$ be such that $B_{\rho}(x)\subset \Omega$. We prove that for any harmonic function $v: \Om \rightarrow \R$, and multiindex $\alpha$ with $|\alpha|=2$, and $r$ as in Step 1 that
\begin{equation}\label{Eq: Claim5Disc}
     |D^{\alpha}v(x)| \ \lesssim \ \rho^{-2 - n/2 + r + 1}\|v\|_{H^{r + 1}(B_{\rho}(x))}.
\end{equation}
To this end, we observe by standard estimates for harmonic functions that
\begin{equation}\label{Eq: Claim4Disc}
    |D^{\alpha}v(x)| \ \lesssim \ \rho^{-n - 2}\|v\|_{L^1(B_{\rho}(x))},
\end{equation}
for each multi-index $|\al| = 2$; see Theorem 7 in Chapter 2 of \cite{evans2022partial}. We then apply a standard Bramble-Hilbert argument on the space of continuous piecewise linear functions with \eqref{Eq: Claim4Disc} to obtain \eqref{Eq: Claim5Disc}.

\textbf{Step 3:} Now we prove the following pointwise estimate:  Given $r$ with $0<r<1$ and $r \le \beta_\Omega$, there holds
\begin{equation}\label{Eq: Claim10Disc}
    |(E-E_h)g_h(z)| \ \lesssim \  d^{-n/2}h^{r +\be_{\Om}} \|g_h\|_{H^{r+\frac{1}{2}}(\Gamma)}.
\end{equation}
To start, we apply \cite[Theorem 1.1]{schatz1995interior} with $B := B_{d/2}(z)$, $q := 2$, $\Om_0 := B_{d/2}(z)$, $\Om_d, \mathcal{D} := \Om$, $\overline{r} := 1$, $s := 0$, and $N := n$.
Note that equation (1.2) in that paper is satisfied due to Galerkin orthogonality, so we obtain 
\begin{equation}\label{Eq: Claim6Disc}
    |(E-E_h)g_h(z)| \ \lesssim \ \ln\left(\frac{d}{2h}\right) \min_{\chi \in \mathbb V_h, \chi|_\Gamma= g_h} \| Eg_h - \chi \|_{L_{\infty}(B_{d/2}(z))} + d^{-n/2}\| (E-E_h)g_h \|_{L_2(B_{d/2}(z))}.
\end{equation}
Recall that by assumption $d \geq c_0 h$ for $c_0$ sufficiently large. In fact we first require that $c_0 \geq 1$ is such that $B_{d/2+h}(z) \subset \Omega$.
Next, by standard error estimates
\begin{equation}\label{Eq: Claim7Disc}
   \min_{\chi \in \mathbb V_h, \chi|_\Gamma = g_h} \| Eg_h - \chi \|_{L_{\infty}(B_{d/2}(z))} \ \lesssim \ h^2 \| Eg_h\|_{W^2_\infty(B_{d/2+h}(z))}.
\end{equation}
Now, let $x \in B_{d/2+h}(z)$. Applying  \eqref{Eq: Claim5Disc} to $Eg_h$ with $\rho := d/c_0$, for any multi-index $|\al| = 2$,  gives us
\begin{equation}\label{Eq: Claim8Disc1}
    |D^{\al}(Eg_h)(x)| \ \lesssim \ d^{-2 - n/2 + r + 1}\|Eg_h\|_{H^{r + 1}(B_{d/c_0}(x))}.
\end{equation}
 Taking the supremum over all possible multi-indices $\al$ and points $x \in B_{d/2 + h}(z)$ and restricting $c_0$ further so that $B_{d/2+d/c_0+h}(z) \subset \Omega$ gives
\begin{equation}\label{Eq: Claim8Disc2}
\begin{aligned}
    h^2\|Eg_h\|_{W^2_{\infty}(B_{d/2+h}(z))} &  \lesssim \ h^2d^{-2 - n/2 + r + 1}\|Eg_h\|_{H^{r + 1}(B_{d/2 + d/c_0+h}(z))} 
    \\ & \leq \ h^2d^{-2 - n/2 + r + 1}\|Eg_h\|_{H^{r + 1}(\Om)}.
    \end{aligned}
\end{equation}

Combining \eqref{Eq: embedStable}, \eqref{Eq: Claim7Disc} and \eqref{Eq: Claim8Disc2} yields
\begin{equation}\label{Eq: Claim9Disc}
   \min_{\chi_h \in \mathbb V_h} \| Eg_h - \chi_h \|_{L_{\infty}(B)} \ \lesssim \ d^{-2-n/2+r + 1} h^2 \| Eg_h\|_{H^{r + 1}(\Omega)} \ \lesssim \ d^{-1-n/2+r}h^2\|g_h\|_{H^{r+\frac{1}{2}}(\Gamma)}.
\end{equation}

Since $1 - r > 0$ and $1-\beta_\Omega \ge 0$, we note that 
$$
d^{-1-n/2+r} h^2 \ln\left(\frac{d}{2h}\right)= (d^{-n/2} h^{r +\beta_\Omega})h^{1-\beta_\Omega} \left ( \frac{h}{d} \right )^{1-r} \ln\left(\frac{d}{2h}\right) \lesssim d^{-n/2} h^{r +\beta_\Omega}.
$$
We may then conclude \eqref{Eq: Claim10Disc} by inserting \eqref{Eq: Claim3ADisc}  and \eqref{Eq: Claim9Disc} into \eqref{Eq: Claim6Disc}.

\textbf{Step 4:} We show that 
\begin{equation}\label{Eq: Claim13Disc}
|E(\psi-\psi_h)(z) |\ \lesssim \ d^{-n-2\tilde{\mu}} h^{2R(s, s+\mu)} + d^{-n-\tilde{\mu}} h^{r(s+\mu)+\beta_\Omega}.
\end{equation}
For the remainder of this proof, we set $g := \psi$ and $g_h := \psi_h$. Since $E(\psi - \psi_h)$ is harmonic, we have by Assumption \ref{ass: lambda bounded} that 
\begin{equation}\label{Eq: Claim11DiscA}
    |E(\psi - \psi_h)(z)| \ \lesssim \ d^{-n/2-\tilde{\mu}} \|\psi - \psi_h\|_{H^{-\mu}(\Gamma)}.
\end{equation}
We thus obtain \eqref{Eq: Claim13Disc} upon proving that
\begin{equation}\label{Eq: Claim13Disc2}
    \|\psi - \psi_h\|_{H^{-\mu}(\Gamma)} \ \lesssim \ d^{-n/2-\tilde{\mu}} h^{2R(s, s+\mu)} + d^{-n/2} h^{r(s+\mu)+\beta_\Omega}.
\end{equation}
To this end we note that
$$\|\psi-\psi_h\|_{H^{-\mu}(\Gamma)} \ = \ \sup_{v \in H^\mu(\Gamma), \ \|v\|_{H^\mu(\Gamma)} = 1} \langle \psi-\psi_h, v \rangle,
$$
where $\langle ., . \rangle$ stands for the $H^{-\mu}-H^\mu$ duality product.
Given a $v \in H^\mu(\Gamma)$ with $\|v\|_{H^\mu(\Gamma)}=1$ and $\|\psi-\psi_h\|_{H^{-\mu}(\Gamma)} \lesssim \langle \psi-\psi_h, v \rangle$, we then let $\kappa$ be the solution of \eqref{Eq: hsvargamma} with $\gamma_\kappa(g)=\langle v,g\rangle$ and taking $g=\psi-\psi_h$. Assumption~\ref{Assumption: R} guarantees that $\kappa \in \tilde{H}^{s + R(s, s+\mu)}(\Gamma)$. Noting that $\psi_h$ is {\it not} the Galerkin approximation to $\psi$, we let $\kappa_h, \bar{\psi}_h \in \mathbb{T}_h$ be the $H^s(\Gamma)$ Galerkin approximations to $\kappa$ and $\psi$, respectively.  Then using Galerkin orthogonality of $\kappa-\kappa_h$ we have
\begin{equation} \label{eq: error_rep}
 \|\psi-\psi_h\|_{H^{-\mu}(\Gamma)} \ \lesssim \ \langle \psi-\psi_h, \kappa \rangle_{H^s(\Gamma)} \ = \ \langle \psi-\bar{\psi}_h, \kappa-\kappa_h \rangle_{H^s(\Gamma)} + \langle \psi-\psi_h, \kappa_h \rangle_{H^s(\Gamma)}.
\end{equation}
Let $\gamma_z \in H^\mu(\Gamma)$ be the extension of $g \mapsto Eg(z)$ guaranteed by Assumption~\ref{ass: lambda bounded}. Estimate~\eqref{Eq: L2 boundary embedding} together with Assumption~\ref{Assumption: R} and C\'ea's Lemma then yield 
\begin{equation} \label{eq: errstep1}
\|\psi-\bar{\psi}_h\|_{H^s(\Gamma)}  \ \lesssim \ h^{R(s, s+\mu)} \|\gamma_z\|_{H^{\mu}(\Gamma)}
\ \lesssim \ d^{-n/2-\tilde{\mu}} h^{R(s,s+\mu)}.
\end{equation}
Assumption \ref{Assumption: R} also  yields 
\begin{equation} \label{eq: errstep2}
\|\kappa-\kappa_h \|_{H^s(\Gamma)} \ \lesssim \ h^{R(s,s+\mu)} \|v\|_{H^\mu(\Gamma)} \ = \ h^{R(s,s+\mu)}.
\end{equation}

Using \eqref{Eq: psi}, \eqref{Eq: psiH}, \eqref{Eq: Claim10Disc} with $r=r(s+\mu)$, and the $H^s$-boundedness of the Galerkin projection, we next find that
\begin{equation}
\begin{aligned} \label{eq: kappahrep}
\langle \psi-\psi_h, \kappa_h \rangle_{H^s(\Gamma)}\ & = \ [(E-E_h)\kappa_h](x_i) \ \lesssim \ d^{-n/2} h^{r(s+\mu)+\beta_\Omega} \|\kappa_h\|_{H^{r(s+\mu)+\frac{1}{2}}(\Gamma)}. 
\end{aligned}
\end{equation}
If $r(s+\mu)+\frac{1}{2} \le s$, then 
\begin{equation*}
\|\kappa_h\|_{H^{r(s+\mu)+\frac{1}{2}}(\Gamma)} \ \le \ \|\kappa_h\|_{H^s(\Gamma)} \ \le \ \|v\|_{H^{-s}(\Gamma)} \ \lesssim \ 1.
\end{equation*}
If $s<r(s+\mu)+\frac{1}{2}$, then using the triangle inequality, inverse estimates (Proposition~\ref{p:inverse}), stability of $\pi_h$ (since $r(s+\mu)+\frac{1}{2}< \frac{3}{2}$),   approximation properties, and $r(s+\mu)+\frac{1}{2} \le s+R(s,s+\mu)$ yields
\begin{equation}\label{eq: kappahone}
\begin{aligned}
\|\kappa_h \|_{H^{r(s+\mu)+\frac{1}{2}}(\Gamma)} & \ \lesssim \ h^{s-r(s+\mu)-\frac{1}{2}} \|\kappa_h-\pi_h \kappa\|_{H^s(\Gamma)} + \|\pi_h \kappa\|_{H^{r(s+\mu)+\frac{1}{2}} (\Gamma)}
\\ & \ \lesssim \ h^{s-r(s+\mu)-\frac{1}{2}} \left ( \|\kappa_h -\kappa\|_{H^s(\Gamma)} + \|\kappa-\pi_h \kappa\|_{H^s(\Gamma)} \right ) + \|\kappa\|_{H^{r(s+\mu)+\frac{1}{2}}(\Gamma)}
\\ & \ \lesssim \ \|\kappa\|_{H^{r(s+\mu)+\frac{1}{2}}(\Gamma)}  \lesssim \|v\|_{H^{\mu}(\Gamma)} \lesssim 1.
\end{aligned}
\end{equation}
Thus $\langle \psi-\psi_h, \kappa_h\rangle_{H^s(\Gamma)} \lesssim d^{-n/2} h^{r(s+\mu)+\beta_\Omega}$, and recalling \eqref{eq: error_rep} we conclude that
\begin{equation}\label{Eq: Claim13Disc6}
\begin{aligned} 
      \|\psi - \psi_h\|_{H^{-\mu}(\Gamma)} & \lesssim \|\kappa-\kappa_h\|_{H^s(\Gamma)}\|\psi-\bar{\psi}_h\|_{H^s(\Gamma)} + \langle \psi-\psi_h, \kappa_h \rangle_{H^s(\Gamma)}
      \\ & \lesssim h^{R(s, s+\mu)}d^{-n/2-\tilde{\mu}} h^{R(s, s+\mu)}  + d^{-n/2} h^{r(s+\mu)+\beta_\Omega} 
      \\ &  \lesssim d^{-n/2-\tilde{\mu}} h^{2R(s,s+\mu)}  + d^{-n/2} h^{r(s+\mu)+\beta_\Omega}.  
      \end{aligned}
\end{equation}
This yields \eqref{Eq: Claim13Disc} when combined with \eqref{Eq: Claim11DiscA}.

\textbf{Final step:} We now complete the proof of \eqref{Eq: convphiboundinf}.  
We first bound $\|\psi_h\|_{H^{r(s+\mu)+\frac{1}{2}}(\Gamma)}$.  Recall that $\bar{\psi}_h \in \mathbb{T}_h$ is the $H^s(\Gamma)$ Galerkin approximation to $\psi$.  Employing \eqref{Eq: Claim10Disc} with $r=r(s+\mu)$ and inverse estimates yields
\begin{equation}\label{Eq: psih-psihbar}
    \begin{aligned}
        \|\bar{\psi}_h-\psi_h\|_{H^s(\Gamma)}^2 & = \langle \bar{\psi}_h-\psi_h, \bar{\psi}_h-\psi_h\rangle_{H^s(\Gamma)}  = [(E-E_h)(\bar{\psi}_h-\psi_h)](x_i) 
        \\ & \lesssim d^{-n/2} h^{r(s+\mu)+\beta_\Omega} \|\bar{\psi}_h -\psi_h\|_{H^{r(s+\mu)+\frac{1}{2}} (\Gamma)} 
        \\ & \lesssim d^{-n/2} h^{r(s+\mu)+\beta_\Omega+s-r(s+\mu)-\frac{1}{2}} \|\bar{\psi}_h-\psi_h\|_{H^s(\Gamma)}.
        \\ & \lesssim d^{-n/2} h^{\beta_\Omega+s-\frac{1}{2}} \|\bar{\psi}_h-\psi_h\|_{H^s(\Gamma)}.
    \end{aligned}
\end{equation}
Since $s>\frac{1}{2}$, we have
\begin{equation}
\|\bar{\psi}_h-\psi_h\|_{H^s(\Gamma)} \lesssim d^{-n/2} h^{\beta_\Omega}.  
\end{equation} 
If $r(s+\mu)+\frac{1}{2}>s$, we argue as in \eqref{eq: kappahone} to obtain
\begin{equation} \label{eq: psihone}
\begin{aligned}
\|\psi_h\|_{H^{r(s+\mu)+\frac{1}{2}}(\Gamma)} &\  \lesssim h^{s-r(s+\mu)-\frac{1}{2}}\|\bar{\psi}_h-\psi_h\|_{H^s(\Gamma)} + \|\psi\|_{H^{r(s+\mu)+\frac{1}{2}}(\Gamma)} 
\\ & \lesssim d^{-n/2} h^{s-r(s+\mu)-\frac{1}{2}+\beta_\Omega} + \|\gamma_z\|_{H^{\mu}(\Gamma)} 
\lesssim d^{-n/2-\tilde{\mu}},
\end{aligned}
\end{equation}
where in the last line we have used $s>\frac{1}{2}$ and $r(s+\mu) \le \beta_\Omega$. If $r(s+\mu)+\frac{1}{2} \le s$, then
\begin{equation} \label{eq: psih2}
\|\psi_h\|_{H^{r(s+\mu)+\frac{1}{2}}(\Gamma)} \ \lesssim \ \|\psi_h-\bar{\psi}_h\|_{H^s(\Gamma)} + \|\bar{\psi}_h\|_{H^s(\Gamma)} \lesssim d^{-n/2} h^{\beta_\Omega} + \|\gamma_z\|_{H^{-s}(\Gamma)} \lesssim d^{-n/2} \lesssim d^{-n/2-\tilde{\mu}}.
\end{equation}

To prove \eqref{Eq: convphiboundinf}, we decompose
\begin{equation}\label{Eq: Claim14Disc1}
    (\phi-\phi_h)(z) \ = \ (E\psi-E_h\psi_h)(z) \ = \ E(\psi-\psi_h)(z) + (E-E_h)\psi_h(z).
\end{equation}
Next, due to \eqref{Eq: Claim10Disc},  \eqref{eq: psihone}, and \eqref{eq: psih2},
\begin{equation}\label{Eq: Claim14Disc3}
    |(E-E_h)\psi_h(z)| 
   \ \lesssim \ d^{-n/2} h^{r(s+\mu)+\beta_\Omega} \|\psi_h\|_{H^{r(s+\mu)+\frac{1}{2}}(\Gamma)} \ \lesssim \ d^{-n-\tilde{\mu}} h^{r(s+\mu)+\beta_\Omega}.  
\end{equation}
Finally, \eqref{Eq: convphiboundinf} follows upon combining \eqref{Eq: Claim13Disc}, \eqref{Eq: Claim14Disc1}, and \eqref{Eq: Claim14Disc3}. 


\end{proof}

Note, importantly, that the previous result does not rely on a continuous embedding of $H^s(\Gamma)$ into $L_{\infty}(\Gamma)$. Now we show the error estimate for the optimal recovery in Case 3, i.e., $X=\{ v \in L_2(\Omega) \ : \ v|_{\Omega_d} \in L_\infty(\Omega_d)\}$ and $|.|_X = \|.\|_{L_\infty(\Omega_d)}$.
We let $u_h^*$ be the output of the optimal recovery algorithm (Algorithm~\ref{algo:rec}), where $\hat u_0=u_{0h}$ satisfies \eqref{Eq: u0 point approx} and $\hat \phi_j = E_h \psi_{jh}$ is the finite element approximation of $\phi_j$ with $\psi_{jh}$ the solution to \eqref{Eq: psiH} for $E_hg_h(\overline x) = E_h g_h(x_j)$.

\begin{corollary}\label{Cor: quantEstPointwiseRec}
    Suppose the assumptions of Theorem~\ref{Thm: theorateinf} hold and that $f$ is such that  \eqref{Eq: u0 point approx} and \eqref{Eq:u0H1} hold. Then we have
    \begin{equation}\label{Eq: quantEstPointwiseRec2D}
        \sup_{v\in \mathcal K_\om} \|v - u_h^*\|_{L_\infty(\Omega_d)} \ \leq \ R(\mathcal{K}_{\om})_{L_\infty(\Omega_d)} + C \max \Bigg\{h^{\tau_d} \ln\left(\frac{d}{2h}\right) , d^{-n/2} h^{\beta_\Omega+\be_f}, d^{-n-2\tilde{\mu}} h^{t_p} \Bigg\}.
    \end{equation}
\end{corollary}
\begin{proof}
    Combine Theorem~\ref{Thm: NORB} and Theorem~\ref{Thm: theorateinf}.
\end{proof}





\subsection{$H^1$ error estimates}\label{Subsec: H^1FEM}

Now, observe that \cite[Theorem 4.7]{binev2024solving} uses global $L_{\infty}$ estimation techniques to prove a convergence result for $\|\phi - \phi_h\|_{H^1(\Om)}$. However, those techniques rely on $s$ being sufficiently large so that $H^s(\Gamma)$ embeds continuously into $C(\Om)$. By using Theorem \ref{Thm: theorateinf}, we can obtain a new estimate for $\|\phi - \phi_h\|_{H^1(\Om)}$ that avoids the use of such embeddings when $x_i \in \Omega_d$. 




\begin{theorem}[$H^1$ estimate]\label{Thm: TheorateH^1}
    Assume that the setting of Theorem \ref{Thm: theorateinf} holds.  Then if $t_H := \min\{r(s+\mu), R(s, s + \mu)+R(s,s-\frac{1}{2})\}$, we have the estimate
    \begin{equation}\label{Eq: TheorateH^1Succinct}
 \|\phi - \phi_h\|_{H^1(\Om)}\ \lesssim \ h^{t_H}.  
    \end{equation}
    More explicitly,
    \begin{equation}\label{Eq: TheorateH^1Expanded}
         \|\phi - \phi_h\|_{H^1(\Om)} \ \lesssim \ d^{-n/2-\tilde{\mu}}\left [ h^{r(s+\mu)}+h^{R(s,s+\mu) + R(s,s-\frac{1}{2})} \right ]. 
    \end{equation}
\end{theorem}

\begin{proof}
We employ the decomposition
\begin{equation}\label{Eq: TheorateH^1Eq0A}
     \|\phi - \phi_h\|_{H^1(\Om)} \ = \ \|E\psi - E_h\psi_h\|_{H^1(\Om)} \ \leq \ \|(E - E_h)\psi_h\|_{H^1(\Om)} + \|E(\psi - \psi_h)\|_{H^1(\Om)}.
\end{equation}
For the first term in \eqref{Eq: TheorateH^1Eq0A}, combining \eqref{Eq: Claim2Disc1} with \eqref{eq: psihone} and \eqref{eq: psih2} while replacing $r$ with $r(s + \mu)$ yields
\begin{equation}\label{Eq: TheorateH^1Eq0B}
    \|(E - E_h)\psi_h\|_{H^1(\Om)} \ \lesssim \ h^{r(s+\mu)}\|\psi_h\|_{H^{r(s+\mu) + 1/2}(\Gamma)} \ \lesssim \ h^{r(s+\mu)} d^{-n/2-\tilde{\mu}} .
\end{equation}
    Now we consider the second term in \eqref{Eq: TheorateH^1Eq0A}. We use \eqref{Eq: EgStable} and C\'ea's lemma to observe
    \begin{equation}\label{Eq: TheorateH^1Eq7}
        \|E(\psi - \psi_h)\|_{H^1(\Om)} \ \lesssim \ \|\psi - \psi_h\|_{H^{1/2}(\Gamma)} \ = \ \sup_{\gamma \in H^{-1/2}(\Gamma), \ \|\gamma\|_{H^{-1/2}(\Gamma)}=1} \gamma(\psi-\psi_h).
    \end{equation}
    Let $\gamma \in H^{-1/2}(\Gamma)$ with $\|\psi-\psi_h\|_{H^{1/2}(\Gamma)} \lesssim \gamma(\psi-\psi_h)$ and let $\kappa \in H^s(\Gamma)$ solve \eqref{Eq: hsvargamma} with $\gamma$ thus defined.  Letting $\bar{\psi}_h, \kappa_h \in \mathbb{T}_h$ be the $H^s(\Gamma)$ Galerkin approximations to $\psi, \kappa$, we have as in \eqref{eq: error_rep} that 
    \begin{equation} \label{eq: h1one}
        \begin{aligned}
            \|\psi-\psi_h\|_{H^{1/2}(\Gamma)} \ \lesssim \ \langle \psi-\bar{\psi}_h, \kappa-\kappa_h \rangle_{H^s(\Gamma)} + \langle \psi-\psi_h, \kappa_h \rangle_{H^s(\Gamma)}. 
        \end{aligned}
    \end{equation}
        Arguing as in \eqref{eq: errstep1}--\eqref{eq: errstep2} with $\gamma \in H^{-1/2}(\Gamma)$ replacing $\gamma_{\kappa} \in H^\mu(\Gamma)$ in \eqref{eq: errstep2} yields 
        \begin{equation}
        \langle \psi-\bar{\psi}_h, \kappa-\kappa_h \rangle_{H^s(\Gamma)} \ \lesssim \ h^{R(s,s+\mu) + R(s,s-\frac{1}{2})} d^{-n/2-\tilde{\mu}}.  
        \end{equation}
          We employ \eqref{eq: kappahrep}-\eqref{eq: kappahone} with $r(s-\frac{1}{2})$ in place of $r(s+\mu)$ and \eqref{Eq: Claim10Disc} with $r$ replaced by $r(s-\frac{1}{2})$  to obtain 
        \begin{equation} \label{eq: h1three}
        \langle \psi-\psi_h, \kappa_h \rangle_{H^s(\Gamma)} \ \lesssim \ d^{-n/2} h^{r(s-\frac{1}{2})+\beta_\Omega}.
            \end{equation}
Combining the preceding inequalities while noting that $d^{-n/2}\lesssim d^{-n/2-\tilde{\mu}}$ and $\beta_\Omega+r(s-\frac{1}{2})\ge r(s+\mu)$ yields \eqref{Eq: TheorateH^1Succinct} and \eqref{Eq: TheorateH^1Expanded}.  
\end{proof}

Now we discuss quantitative convergence rates for $\sup_{v \in \mathcal K_{\om}}\|v - u_h^*\|_{H^1(\Om)}$, where $u_h^*$ is the output of the optimal recovery algorithm, as an application of Theorem~\ref{Thm: NORB} (originally \cite[Theorem 3.6]{binev2024solving}), to the finite element setting.

\begin{corollary}\label{Cor: QuantConvUs}
    Suppose the assumptions of Theorem \ref{Thm: theorateinf} hold and that $f$ is such that \eqref{Eq: u0 point approx} and \eqref{Eq:u0H1} hold. Then, for some constants $C_1$ and $C_2$ with $C_2$ depending on $d$, we have
    \begin{equation}\label{Eq: QuantConvUs2D}
        \begin{split}
       \sup_{v\in \mathcal K_\om} &  \|v - u_h^*\|_{H^1(\Om)} 
\\ &        \leq \ R(\mathcal{K}_{\om})_{H^1(\Om)} + C_1\Big [h^{\beta_f} (1+d^{-n/2} h^{\beta_\Omega})+ h^{\tau_d} \ln\left(\frac{d}{2h}\right)  + d^{-n/2-\tilde{\mu}} h^{r(s+\mu)}(1+d^{-n/2} h^{\beta_\Omega})\\ & \qquad   + d^{-n/2-\tilde{\mu} } h^{R(s,s+\mu)+R(s,s-\frac{1}{2})}(1+d^{-n/2-\tilde{\mu}} h^{R(s,s+\mu)-R(s,s-\frac{1}{2})}) \Big ]\\
       &\leq \ R(\mathcal{K}_{\om})_{H^1(\Om)} + C_2\left (h^{\beta_f}+ h^{\tau_d} \ln\left(\frac{d}{2h}\right) + h^{r(s+\mu)} + h^{R(s,s+\mu)+R(s,s-\frac{1}{2})} \right).
        \end{split}
    \end{equation}    
\end{corollary}

\begin{proof} We apply Theorem~\ref{Thm: NORB} invoking \eqref{Eq:u0H1}, \eqref{Eq: u0 point approx} for the approximation of $u_0$ and Theorem \ref{Thm: theorateinf}, Theorem \ref{Thm: TheorateH^1}, and $r(s+\mu) \le \beta_\Omega$ for the approximation of the Riesz representers.
\end{proof}

\subsection{Recovery error in $L_{\infty}(\Om)$}\label{Subsec: LinfErrorOpt}

Now we seek quantitative error estimates for $\sup_{v\in \mathcal K_{\om}}\|v - u_h^*\|_{L_{\infty}(\Om)}$, where $u_h^*$ is the output of the optimal recovery algorithm in the finite element setting. To achieve these, we need to assume $L_{\infty}$-stability of the Galerkin projection $R_h: H^1_0(\Om) \rightarrow \mathbb{V}^0_h$, with $R_hv$ defined via
\begin{equation}\label{Eq: Rh}
    \int_{\Om}\grad R_h v \cdot \grad v_h \ = \ \int_{\Om}\grad v \cdot \grad v_h, \ \quad \ \fa v_h \in \mathbb{V}^0_h.
\end{equation}
Under various conditions on the mesh and domain, this will guarantee the existence of $C_{\text{gal}} > 0$ and $a \geq 0$ so that for all $h > 0$ and $v \in H^1_0(\Om) \cap L_{\infty}(\Om)$, we have
\begin{equation}\label{Eq: RitzStability}
    \|R_hv\|_{L_{\infty}(\Om)} \ \leq \ C_{\text{gal}}(1 + |\ln(h)|)\|v\|_{L_{\infty}(\Om)}.
\end{equation}
In the specific settings  of interest to us, such estimates hold:
\begin{itemize}
    \item Convex polygonal domain in $\R^2$ with a quasi-uniform mesh (see \cite[Theorem 1.2]{Sch80})
    \item Convex polyhedral domain in $\R^3$ (see \cite[Theorem 12]{LV16})
\end{itemize} 
In either case, we assume $s > \frac{n - 1}{2}$ and maintain Lemma 4.6 of \cite{binev2024solving}, which in our setting reads:
\begin{equation}\label{Eq: LinfExtError}
    \|(E - E_h)g_h\|_{L_{\infty}(\Om)} \ \lesssim \ (1 + |\ln(h)|)h^{r(s+\mu) + 1 - \frac{n}{2}}\|g_h\|_{H^s(\Gamma)}, \ \fa g_h \in \mathbb{T}_h.
\end{equation}

Here is the $L_{\infty}$ estimate for the Riesz representers.  The proof roughly follows those of \cite[Lemma 4.6 and Theorem 4.7]{binev2024solving}, but taking into account the assumption $x_i \in \Omega_d$.

\begin{theorem}[$L_{\infty}$ estimate]\label{Thm: RieszGlobalLInfinity}
Let $\beta^*$ be as in Assumption~\ref{Assumption: R omega} and set $\beta_\Omega=\min(\beta^*,1)$. 
Let also $\mu, \tilde{\mu} \ge 0$ be as in Assumption \ref{ass: lambda bounded}, $R(.,.)$ as in Assumption \ref{Assumption: R}, and $r(.)$ be given by \eqref{eq: rdef}.
Assume that $s+R(s,s+\mu)> \frac{n-1}{2}$ and let $\theta \in \left(\frac{n - 1}{2}, s +R(s,s+\mu) \right)$.
Let also 
\begin{equation}
t_{\infty} \ := \ \left \{ \begin{array}{l} \min\Big ( r(s+\mu) + 1 - \frac{n}{2},R(s, s + \mu)+ R(s, s - \theta), \beta_\Omega + r(s-\theta) \Big ) , ~ \frac{n-1}{2}<\theta<s,
\\ \min\Big ( r(s+\mu) + 1 - \frac{n}{2}, s - \theta + R(s, s + \mu), \beta_\Omega + 2s-\theta-\frac{1}{2} \Big ) , ~s \le \frac{n-1}{2}<\theta.
\end{array} \right .  
\end{equation}
Then we have the estimate
    \begin{equation}\label{Eq: RieszGlobalLInfinity}
        \begin{split}    
        \|\phi - \phi_h\|_{L_{\infty}(\Om)} & \lesssim (1+|\ln h|) d^{-n/2-\tilde{\mu}} h^{r(s+\mu)+1-n/2}
\\ & \qquad + \left \{ \begin{array}{l} d^{-n/2-\tilde{\mu}} h^{R(s,s-\theta)+R(s,s+\mu)} + d^{-n/2} h^{\beta_\Omega +r(s-\theta)}, ~~\frac{n-1}{2}<\theta<s, 
\\ d^{-n/2-\tilde{\mu}} h^{s-\theta+R(s,s+\mu)}+ d^{-n/2} h^{\beta_\Omega+2s-\theta-1/2}, ~~s \le \frac{n-1}{2} < \theta
\end{array} \right . \\
& \lesssim 
        \ (1 + |\ln(h)|)h^{t_{\infty}}.
        \end{split}
    \end{equation}
\end{theorem}

\begin{proof}
First imitating the proof of \cite[Lemma 4.6]{binev2024solving} while applying \eqref{eq: psihone}-\eqref{eq: psih2}, we find that
\begin{equation} \label{eq: linfstepone}
\begin{aligned}
\|&(E-E_h) \psi_h\|_{L_\infty(\Omega)} 
\\ & \lesssim \ln \left(\frac{1}{h}\right) h^{r(s+\mu)+1-\frac{n}{2}} |E \psi_h|_{C^{r(s+\mu)+1-\frac{n}{2}}(\Omega)} 
 \lesssim \ln \left(\frac{1}{h}\right) h^{r(s+\mu)+1-\frac{n}{2}} \|E \psi_h \|_{H^{r(s+\mu)+1}(\Omega)} \\ & \lesssim \ln \left(\frac{1}{h}\right) h^{r(s+\mu)+1-\frac{n}{2}} \|\psi_h\|_{H^{r(s+\mu)+\frac{1}{2}}(\Gamma)}
\lesssim \ln \left(\frac{1}{h}\right) d^{-n/2-\tilde{\mu}} h^{r(s+\mu)+1-\frac{n}{2}}.
\end{aligned}
\end{equation}
In order to bound $\|E(\psi-\psi_h)\|_{L_\infty(\Omega)}$, first assume that $s> \frac{n-1}{2}$ and let $s>\theta>\frac{n-1}{2}$.  Then employing the maximum principle, a Sobolev embedding, and proceeding as in \eqref{eq: h1one}-\eqref{eq: h1three} with $H^{\theta}(\Gamma)$ replacing $H^{1/2}(\Gamma)$ in the arguments we find
\begin{equation} \label{eq: linfsteptwo}
    \|E(\psi-\psi_h)\|_{L_\infty(\Omega)} \le \|\psi-\psi_h\|_{L_\infty(\Gamma)} \lesssim \|\psi-\psi_h\|_{H^\theta(\Gamma)} \lesssim d^{-n/2-\tilde{\mu}} h^{R(s,s+\mu)+R(s,s-\theta)} + d^{-n/2} h^{r(s-\theta)+\beta_\Omega}.
\end{equation}
If instead $s\le \frac{n-1}{2}< \theta $, we use a Sobolev embedding, inverse estimate, approximation properties, and \eqref{eq: errstep1} to obtain
\begin{equation}
\begin{aligned}
\|E(\psi-\bar{\psi}_h)\|_{L_\infty(\Omega)} & \lesssim \|\psi-\pi_h \psi\|_{H^\theta(\Gamma)}+h^{s-\theta} (\|\psi-\bar{\psi}_h\|_{H^s(\Gamma)}+ \|\psi-\pi_h \psi\|_{H^s(\Gamma)}) 
\\ & \lesssim h^{s-\theta+R(s,s+\mu)} d^{-n/2-\tilde{\mu}}.
\end{aligned}
\end{equation}
A Sobolev embedding, inverse estimate, and \eqref{Eq: psih-psihbar} also yield
\begin{equation}\label{Eq: inverseEmbed}
    \|E(\bar{\psi}_h-\psi_h)\|_{L_\infty(\Omega)} \lesssim h^{s-\theta}\|\bar{\psi}_h-\psi_h\|_{H^s(\Gamma)} \lesssim d^{-n/2} h^{\beta_\Omega+2s-\theta-\frac{1}{2} }.
\end{equation}
Combining the Triangle Inequality with \eqref{eq: linfstepone}-\eqref{Eq: inverseEmbed} completes the proof.
\end{proof}

\begin{remark} We compare the powers of $h$ in \eqref{Eq: RieszGlobalLInfinity} when $s > \frac{n-1}{2}$ versus when $s \le \frac{n-1}{2}$.  Recall that $R(s,s-\theta)\le s-\theta$ with equality occurring when sufficient elliptic regularity holds.  In the latter case $R(s,s+\mu)+R(s,s-\theta)=R(s,s+\mu)+s-\theta$, i.e., the powers of $h$ in the second term take the same form in both cases.  Also, recall that $r(s-\theta)=\min \left\{\beta_\Omega, 1-\epsilon, s+R(s,s-\theta)-\frac{1}{2} \right\}.$ When the minimum occurs in the third term and sufficient elliptic regularity holds, we thus have $r(s-\theta) =2s-\theta-\frac{1}{2}$, in which case the powers of $h$ in the third term of the previous estimates take the same form in both cases.  
\end{remark}

We may use this to derive an error estimate for an optimal recovery in $L_{\infty}(\Om)$, similarly to Corollaries \ref{Cor: QuantConvUs} and \ref{Cor: quantEstPointwiseRec}. For the sake of readability, we do not collate powers of $d$ in this estimate.

\begin{corollary}\label{Cor: QuantConvLInf}
    Suppose the assumptions of Theorem \ref{Thm: RieszGlobalLInfinity} and \eqref{Eq: linf global} hold and in addition that $s>\frac{n-1}{2}$. Then we have for some $C > 0$ depending on $d$ that
    \begin{equation}\label{Eq: QuantConvLInf}
        \sup_{v\in \mathcal K_\om}\|v - u_h^*\|_{L_{\infty}(\Om)} \ \leq \ R(\mathcal{K}_{\om})_{L_{\infty}(\Om)} + C\min\Big\{1 + |\ln(h)|, \ \ln\left(\frac{1}{h}\right)\Big\}h^{\min\{\tau_{\Om}, t_{\infty}\}}.
    \end{equation}
\end{corollary}

\begin{proof}
    Since $s > \frac{n - 1}{2}$, the Banach space $(L_{\infty}(\Om), \|\cdot\|_{L_{\infty}(\Om)})$ satisfies Assumption \ref{ass:X}, so we may apply Theorem \ref{Thm: NORB} with the tolerances $\ep_1 = \mathcal{O}\left(\ln\left(\frac{1}{h}\right)h^{\tau_{\Om}}\right)$ (due to \eqref{Eq: linf global}) and $\ep_2 = \mathcal{O}((1 + |\ln(h)|)h^{t_{\infty}})$ (due to Theorem \ref{Thm: RieszGlobalLInfinity}).
\end{proof}

\subsection{Comparison with global error estimation techniques}
In this subsection we summarize expected orders of convergence we have obtained for $\phi-\phi_h$ in various norms using local error estimation techniques under the assumption that the measurement points $x_i$ lie in $\Omega_d = \{x \in \Omega: {\rm dist}(x, \Gamma) \ge d \}$, and then compare them with corresponding results obtained in \cite{binev2024solving}, where the less restrictive assumption $x_i \in \overline{\Omega}$ and global error estimation techniques were used.  In both cases we assume that full elliptic regularity holds on $\Omega$ and $\Gamma$, that is,  $\beta_\Omega=1$, $R(s,\delta)=\min(\delta,2-s)$ in all cases, and $\mu= \frac{1}{2}$. Assuming that $x_i \in \Omega_d$ and letting $\epsilon>0$ be arbitrary, Theorems \ref{Thm: theorateinf}, 
\ref{Thm: TheorateH^1}, and \ref{Thm: RieszGlobalLInfinity} respectively yield
\begin{equation} \label{eq: locsummary}
    \begin{aligned}
        \|\phi-\phi_h\|_{L_\infty(\Omega_d)} \ & \ \lesssim  d^{-n} \left  (  h^{2-\epsilon} + h^{4-2s} \right ),
        \\ \|\phi-\phi_h\|_{H^1(\Omega)} & \lesssim d^{-n/2} h^{1-\epsilon},
        \\  \|\phi-\phi_h\|_{L_\infty(\Omega)} &  \lesssim d^{-n/2} h^{2-\frac{n}{2}-\epsilon}. 
    \end{aligned}
\end{equation}
From \cite[Theorem 4.7]{binev2024optimal}, we have when $x_i \in \overline{\Omega}$ and $s>\frac{n-1}{2}$ that
\begin{equation} \label{eq: globsummary}
    \begin{aligned}
        \|\phi-\phi_h\|_{H^1(\Omega)} & \lesssim h^{s+\frac{1}{2}-\frac{n}{2}-\epsilon},
      \\  \|\phi-\phi_h\|_{L_\infty(\Omega)} & \lesssim h^{s+\frac{1}{2}-\frac{n}{2}-\epsilon}.
    \end{aligned}
\end{equation}

In Table \ref{tab: convsum} we summarize expected orders of convergence for $\|\phi-\phi_h\|_{L_\infty}$ using global and local error analysis (here we have included the endpoints $s=\frac{1}{2}, \frac{3}{2}$ for comparison as $s$ approaches them even though they are excluded from the analysis).  Note first that the order of convergence with respect to $h$ is better in all cases when assuming $x_i \in \Omega_d$; overlap only occurs as $s \rightarrow \frac{3}{2}$. Similarly, when estimating $\|\phi-\phi_h\|_{H^1(\Omega)}$ a comparable rate of convergence is obtained only when $n=2$ and as $s \rightarrow \frac{3}{2}$.  On the other hand, our analysis for $x_i \in \Omega_d$ compensates higher orders of convergence in $h$ with negative powers of $d$ that may be significant when $d$ is not sufficiently resolved by $h$. As indicated in our numerical experiments in Section \ref{Sec: NumericalResults}, while errors do in fact increase as $d \rightarrow 0$ our estimates are not necessarily optimal with respect to the stated powers of $d$.  Proving estimates with sharper powers of $d$ is potentially a very technical undertaking and is beyond the scope of this paper.  

Note as well that when $n=3$ and $\frac{1}{2}<s \le 1$, the estimates of \cite{binev2024solving} do not apply when $x_i \in \overline{\Omega}$.  In contrast, our estimates in Theorem \ref{Thm: RieszGlobalLInfinity} apply in this case so long as the elliptic problem defining $\psi$ possesses sufficient regularity to ensure $s+R(s,s+\mu)>1$.   However, even though the Riesz representers may be defined in this case the optimal recovery framework continues to apply only when $s>\frac{n-1}{2}$ due to Assumption \ref{ass:X}.

 \begin{center}
 \begin{table} 
 \caption{Summary of expected orders of convergence}
 \begin{tabular}{|c|c|c|c|c|} 
 \hline
   $s$ & $n$ &  $x_i \in \overline{\Omega}$, $\|\phi-\phi_{h}\|_{L_\infty(\Omega)}$ & $x_i \in \Omega_d$, $\|\phi-\phi_{h}\|_{L_\infty(\Omega)}$ & $x_i \in \Omega_d$, $\|\phi-\phi_{h}\|_{L_\infty(\Omega_d)}$ \\ \hline
 $\frac{1}{2}$ & 2 &  $h^0$ & $d^{-1} h^{1-\epsilon}$ & $d^{-2} h^{2-\epsilon}$ \\ \hline
 $1$ & 2 & $h^{\frac{1}{2}-\epsilon}$ & $d^{-1} h^{1-\epsilon}$ &$d^{-2} h^{2-\epsilon}$ \\ \hline
 $\frac{3}{2}$ & 2 & $h^{1-\epsilon}$ & $d^{-1} h^{1-\epsilon}$ & $d^{-2} h^{1} $ \\ \hline
 $\frac{1}{2}$ & 3 & n/a  & $d^{-3/2} h^{\frac{1}{2} -\epsilon}$ & $d^{-3} h^{2-\epsilon}$ \\ \hline
 1 & 3 & $h^0$ & $d^{-3/2} h^{\frac{1}{2} -\epsilon}$ & $d^{-3} h^{2-\epsilon}$ \\ \hline
 $\frac{3}{2}$ & 3 & $h^{\frac{1}{2}-\epsilon}$ & $d^{-3/2} h^{\frac{1}{2} -\epsilon}$ & $d^{-3} h^{1}$ \\ \hline
 \end{tabular} \label{tab: convsum}
 \end{table}
 \end{center}



\subsection{Smooth domains and higher order elements}\label{Subsec: HigherOrder}

In this section we address optimality of our estimates in the context of smooth domains and possible extension to higher-order elements.  

The convergence results in \eqref{eq: locsummary} and Table \ref{tab: convsum} are near-optimal for $\|\phi-\phi_h\|_{H^1(\Omega)}$ and $\|\phi-\phi_h\|_{L_\infty(\Omega_d)}$ when using linear elements.  However, the stated convergence rates of $h^{2-\frac{n}{2}-\epsilon}$ for $\|\phi-\phi_h\|_{L_\infty(\Omega)}$ are highly suboptimal when full regularity is present.  The suboptimality is due to the use of Sobolev embeddings to bound $L_\infty$ norms  by $L_2$-based Sobolev norms.  In order to obtain optimal estimates it is necessary to make sharper arguments both in bounding $\|(E-E_h)\psi_h \|_{L_\infty(\Omega)}$ as in \eqref{eq: linfstepone} and $\|\psi-\psi_h\|_{L_\infty(\Gamma)}$, which is an intermediate term in \eqref{eq: linfsteptwo}.  Below we assume that a degree-$\zeta$ ($\zeta \ge 1$) finite element space of Lagrange type is used and sketch a proof of optimal order $h^{k+1}$ error estimates in the case $s=1$.  Some steps also apply to the general case $\frac{1}{2}<s<\frac{3}{2}$, but we do not provide a full proof because critical $L_\infty$ error estimates on $\Gamma$ are only available when $s=1$.      

Let now $\Gamma$ be smooth.  Assume that $\Omega$ and $\Gamma$ are meshed exactly with compatible meshes, and let $V_h$ be a (parametric) finite element space of degree $\zeta$ defined on this mesh. Assume that $P_h: H^1(\Omega) \rightarrow V_h$ is a Scott-Zhang type interpolant which coincides with $\pi_h$ on $\Gamma$, as above.  Further assume that $P_h$ and $\pi_h$ possess standard properties even though $\Gamma$ is curved.  We shall also employ standard finite element error estimates without specific reference.  In summary, we work within an idealized setting which preserves standard error estimates and approximation properties and does not introduce any variational crimes due to approximation of $\Omega$ or $\Gamma$. 

We first compute using standard maximum norm stability results, the maximum principle, and $(P_h E\psi)|_\Gamma = \pi_h \psi$ that 
\begin{equation} \label{eq: main_steps}
    \begin{aligned}
\|& \phi-\phi_h\|_{L_\infty(\Omega)} \le \|(E-E_h)\psi_h\|_{L_\infty(\Omega)} + \|E(\psi-\psi_h)\|_{L_\infty(\Omega)} 
\\ & \le \|(E-E_h )(\psi_h - \pi_h \psi)\|_{L_\infty(\Omega)} + \|(E-E_h)\pi_h \psi\|_{L_\infty(\Omega)} + \|\psi-\psi_h \|_{L_\infty(\Gamma)}
\\ \ &\lesssim \ \ell_h\left [  \inf_{\chi_1 \in V_h, \chi_1|_\Gamma =\psi_h-\pi_h\psi} \|E(\psi_h-\pi_h \psi)-\chi_1\|_{L_\infty(\Omega)} +
\inf_{\chi_2 \in V_h, \chi_2|\Gamma = \pi_h \psi} \|E \pi_h \psi-\chi_2\|_{L_\infty(\Gamma)} \right ]
\\ & \quad +\|\psi-\psi_h\|_{L_\infty(\Gamma)}
\\ \ & \lesssim \ \ell_h \left [  \|\psi_h -\pi_h \psi\|_{L_\infty(\Gamma)} + \|E(\pi_h \psi-\psi)\|_{L_\infty(\Omega)} + \|E\psi-P_h E \psi\|_{L_\infty(\Omega)} \right ]+ \|\psi-\psi_h\|_{L_\infty(\Gamma)}
\\ \ & \lesssim \ \ell_h \left [  \|\psi-\psi_h \|_{L_\infty(\Gamma)} + \|\psi-\pi_h \psi \|_{L_\infty(\Gamma)} + h^{\zeta+1} \|E\psi\|_{W_\infty^{\zeta+1}(\Omega)}    
\right ].
\end{aligned}
\end{equation}
Here we let $\ell_h = \ln \frac{1}{h}$, and have chosen $\chi_1$ so that $\|\chi_1\|_{L_\infty(\Omega)} =\|\psi_h - \pi_h  \psi\|_{L_\infty(\Gamma)}$.

Sobolev embeddings and regularity of $E$ yield for any $\epsilon>0$
\begin{equation}
    \|E\psi\|_{W_\infty^{\zeta+1}(\Omega)} \ \lesssim \ \|E\psi\|_{H^{\zeta+1+\frac{n}{2}+\epsilon}(\Omega)} \ \lesssim \ \|\psi\|_{H^{\zeta+\frac{n+1}{2}+\epsilon}(\Gamma)}.
\end{equation}
Let $\tilde{\mu}(s,\zeta)$ be the smallest integer strictly larger than $\zeta+\frac{n}{2}-2s$.  Note that $\zeta \ge 1$, $n \ge 2$, and $s<\frac{3}{2}$ yield $\tilde{\mu}(s,\zeta) \ge 0$.  Recall that $\psi$ solves an elliptic problem over $H^s(\Gamma)$ with right hand side $\gamma_z$ with $\gamma_z(g)=Eg(z)$. Proposition \ref{prop: gammasmooth} and  \eqref{eq: gamma_reg} yield $\gamma_z \in H^{\zeta+\frac{n+1}{2}-2s}(\Gamma)$.  Because $\zeta+\frac{n+1}{2}-2s+\epsilon \le \tilde{\mu}(s,\zeta)+\frac{1}{2}$ for $\epsilon$ small enough, setting $\mu=\tilde{\mu}(s,\zeta)+\frac{1}{2}$ in \eqref{eq: gamma_reg} yields $\|\gamma_z\|_{H^{\zeta+\frac{n+1}{2}-2s+\epsilon}(\Gamma)} \lesssim d^{-n/2-\tilde{\mu}(s,\zeta)}$. Thus by a shift theorem, 
\begin{equation}
    \|E\psi \|_{W_\infty^{\zeta+1}(\Omega)} \ \lesssim \ \|\gamma_z\|_{H^{\zeta+\frac{n+1}{2}-2s+\epsilon}(\Gamma)} \lesssim d^{-n/2-\tilde{\mu}(s,\zeta)}.
\end{equation}
Similarly, 
\begin{equation}
\|\psi-\pi_h \psi\|_{L_\infty(\Gamma)} \ \lesssim \ h^{\zeta+1} \|\psi\|_{H^{\zeta+1+\frac{n}{2}+\epsilon}(\Gamma)} \lesssim h^{\zeta+1} d^{-n/2-\tilde{\mu}(s,\zeta)}.
\end{equation}

We are left to bound $\|\psi-\psi_h\|_{L_\infty(\Gamma)}$.  Such pointwise estimates are not known for fractional $s$, but when $s=1$ \cite[Theorem 3.2]{dem09} yields for finite element approximations for the Laplace-Beltrami operator $-\Delta_\Gamma$ that $\|\psi-\bar{\psi}_h\|_{L_\infty(\Gamma)} \lesssim \ell_h h^{\zeta+1} d^{-n/2-\tilde{\mu}(s,\zeta)}$; the same estimate can be shown to hold for the operator $-\Delta_\Gamma +I$ with modest modifications.  

We now bound $\|\bar{\psi}_h -\psi_h\|_{L_\infty(\Gamma)}$.  Let $x \in \Gamma$ satisfy $\|\bar{\psi}_h-\psi_h \|_{L_\infty(\Gamma)} = |(\bar{\psi}_h-\psi_h)(x)|$.
Let $G^x \in H^1(\Gamma)$ and $G_h^x \in \mathbb{T}_h$ be the discrete Green's function corresponding to $x \in\Gamma$ and its $H^1(\Gamma)$ Galerkin approximation, respectively.  In particular, $-\Delta_\Gamma G^x + G^x=\delta_{h,x}$, where $\delta_{h,x}$ is smooth, locally supported in the element containing $x$, and satisfies $\|\delta_{h,x}\|_{L_p(\Gamma)} \lesssim h^{-(n-1)+\frac{n-1}{p}}$.  Then
\begin{equation}\label{Eq: GalerkinDif}
(\bar{\psi}_h-\psi_h)(x) \ = \ \langle \bar{\psi}_h-\psi_h, G_h^x \rangle_{H^1(\Gamma)}\  = \ (E-E_h)G_h^x(z).
\end{equation}
Let $p=\frac{n}{n-1}$.  Local error estimates (cf. \cite{schatz1995interior}), a duality argument, and inverse estimates for harmonic functions yield that 
\begin{equation}\label{Eq: EEhGh}
\begin{aligned}
|(E-E_h)G_h^x(z)| & \  \lesssim \ \ell_h h^{\zeta+1} |E G_h^x|_{W_\infty^{\zeta+1}(B_d(z))} + d^{-n/p-\zeta+1} \|(E-E_h)G_h^x\|_{W_p^{-\zeta+1}(\Omega)} 
\\ & \ \lesssim \ \ell_h h^{\zeta+1} d^{-n/p-\zeta} \|E G_h^x \|_{W_p^1(\Omega)} + d^{-n/2-\zeta+1} h^{\zeta} \|(E-E_h)G_h^x\|_{W_p^1(\Omega)}.
\end{aligned} 
\end{equation}
Employing the existence of a bounded inverse to the trace map $T:W_p^1(\Omega) \rightarrow W_p^{1-\frac{1}{p}}(\Gamma)$ along with $W_p^1(\Omega)$ regularity of Poisson's problem, properties of discrete Green's functions (cf. \cite[Lemma 3.3]{dem09}), and a Sobolev embedding, we find that
\begin{equation}
\begin{aligned}
    \|EG_h^x\|_{W_p^1(\Omega)} \ \lesssim \ \|G_h^x\|_{W_p^{1-1/p}(\Gamma)} \ \lesssim \  \|G_h^x\|_{W_1^1(\Gamma)} \ \lesssim \ \|G_h^x-G^x\|_{W_1^1(\Gamma)}+ \|G^x\|_{W_1^2(\Gamma)} \ \lesssim \ \ell_h.
\end{aligned}
\end{equation}
Next we employ $W_p^1$ stability of the Ritz projection \cite[Theorem 8.5.3]{BS08}, $P_hv|_\Gamma = \pi_h (v_\Gamma)$, approximation properties of $\Pi_h$, and the existence of a bounded inverse to the trace map $T:W_p^1(\Omega) \rightarrow W_p^{1-\frac{1}{p}}(\Gamma)$ (and similarly for $T:W_p^2(\Omega) \rightarrow W_p^{2-\frac{1}{p}}(\Gamma) \times W_p^{1-\frac{1}{p}}(\Gamma)$) along with $W_p^1(\Omega)$ and $W_p^2(\Omega)$ regularity of Poisson's problem,   
Sobolev embeddings, and finally \cite[Lemma 3.3]{dem09} to compute
\begin{equation}
\begin{aligned}
\|(E-E_h)G_h^x\|_{W_p^1(\Omega)} & \le \|(E-E_h)(G_h^x-\pi_h G^x)\|_{W_p^1(\Omega)}+ \|(E-E_h)\pi_h G^x\|_{W_p^1(\Omega)}
\\ & \lesssim \|E(G_h^x-\pi_h G^x)\|_{W_p^1(\Gamma)} + \inf_{\chi \in \mathbb{V}_h, \chi|_\Gamma = \pi_h G^x} \|E \pi_h G_h^x-\chi\|_{W_p^1(\Omega)}
\\ & \lesssim \|E(G_h^x-\pi_h G^x)\|_{W_p^1(\Omega)}+\|E(\pi_h G^x-G^x)\|_{W_p^1(\Omega)} + \|EG^x-P_h EG^x\|_{W_p^1(\Omega)}
 \\ & \lesssim \|G^x-G_h^x\|_{W_p^{1-\frac{1}{p}}(\Gamma)} + \|G^x-\pi_h G^x\|_{W_p^{1-\frac{1}{p}}(\Gamma)} + h \|G^x\|_{W_p^{2-\frac{1}{p}}(\Gamma)}
 \\ & \lesssim \|G^x-G_h^x\|_{W_1^1(\Gamma)} + \|G^x-\pi_h G^x\|_{W_1^1(\Gamma)} + h \|G^x\|_{W_1^2(\Gamma)}
 \\ & \lesssim h \ell_h.
\end{aligned}
\end{equation}
Recalling that $p=\frac{n}{n-1}$ and using \eqref{Eq: GalerkinDif}, \eqref{Eq: EEhGh}, we have
$$\|\bar{\psi}_h-\psi_h\|_{L_\infty(\Gamma)} \lesssim h^{\zeta+1} \ell_h (d^{-n/p-\zeta}+d^{-\frac{n}{p}-\zeta+1})\lesssim  h^{\zeta+1} \ell_h d^{-(n-1)-\zeta}.$$  

Now recalling that $\tilde{\mu}(s,1)$ is the smallest integer larger than $\zeta+\frac{n}{2}-2s=\zeta-\frac{n}{2}-2$, we have $d^{-n/2-\tilde{\mu}(s,\zeta)} \le d^{-n-\zeta+1}$. Combining the previous inequalities thus yields the following proposition.
\begin{proposition}\label{Prop: SmootherCase}
Assume that $\Gamma$ is smooth, that $\mathbb{V}_h$ consists of degree-$\zeta$ Lagrange elements possessing standard approximation properties on a mesh that exactly fits $\Omega$ and $\Gamma$, that $d \ge c_0 h$ with $c_0$ sufficiently large, and that $s=1$.  Then
\begin{equation}
    \|\phi-\phi_h\|_{L_\infty(\Omega)} \ \lesssim \ h^{\zeta+1}(\ell_h)^2 d^{-n-\zeta+1}.
    \end{equation}
    \end{proposition}

We do not give a proof outline for similar optimal-order estimates for $|(\phi-\phi_h)(z)|$ or $\|\phi-\phi_h\|_{H^1(\Omega)}$ in the context of smooth domains and higher order elements.  The critical modifications are however similar.  In particular, note the steps taken in \eqref{eq: main_steps} to ensure that the smooth function $E \psi$ is approximated on $\Omega$ instead of the rougher function $E\psi_h$ as in the previous proofs.  These steps are unnecessary when $\zeta=1$ and can complicate the final estimates on polyhedral domains where regularity gains are limited, but in the context of higher order elements they are needed to ensure interpolants are only applied to functions having maximum smoothness.

\section{Numerical results}\label{Sec: NumericalResults}

\subsection{Saddle point formulation}\label{Subsec: Saddle}


While \eqref{e:psi_h} is the most streamlined way of posing our finite element approximation of the Riesz representers, our implementation instead proceeds by solving an equivalent saddle point problem. To pose this problem, we define
\begin{equation}\label{Eq: Xs}
    X^s(\Om)\ := \ \{v \in H^1(\Om) \ | \ v_{\Gamma} \in H^s(\Gamma)\}
\end{equation}
which is a Hilbert space when equipped with the norm
\begin{equation}\label{Eq: XsNorm}
    \|v\|_{X^s(\Om)} \ := \ (\|v_{\Gamma}\|^2_{H^s(\Gamma)} + \|\grad v\|^2_{L_2(\Om; \R^n)})^{\frac{1}{2}}.
\end{equation}
Then solving \eqref{Eq: hsvargamma} is equivalent to finding the solution $(\phi, \pi) \in X^s(\Om) \times H^1_0(\Om)$ of the saddle point problem
\begin{eqnarray}\label{Eq: saddlePointGeneric}
    \begin{aligned}
    \lang \phi_{\Gamma}, v_{\Gamma}\rang_{H^s(\Gamma)} + \lang \grad v, \grad \pi\rang_{L_2(\Om; \R^n)} \ &= \ \nu(v), \ v \in X^s(\Om) \\
    \lang \grad \phi, \grad z\rang_{L_2(\Om; \R^n)} \ &= \ 0, \ \ \ \ \ z \in H^1_0(\Om).
    \end{aligned}
\end{eqnarray}
The well-posedness of \eqref{Eq: saddlePointGeneric} (and the corresponding discretized form) are described in Subsection 4.2 of \cite{binev2024solving}. As is also discussed there, solving \eqref{e:psi_h} requires computing a discrete harmonic extension of every basis function for the load vector and then solving a separate linear system for each extension. Meanwhile, numerically approximating \eqref{Eq: saddlePointGeneric} allows us to enforce discrete harmonicity and Riesz representer adherence sequentially, substantially reducing computation time.

For the numerical approximation of \eqref{Eq: saddlePointGeneric} when $s$ is nonintegral we refer to the algorithm proposed and analyzed in \cite{bonito2024approximating} in the context of optimal recovery. It is based on the Balakrishnan representation of the solution to fractional diffusion problems as originally proposed for Euclidean domains in \cite{bonito2019sinc, bonito2015numerical,BP17} and extended to fractional diffusion problems on surfaces in \cite{bonito2021approximation}. However, the numerical experiments proposed below are for $s=1$, which is sufficient to illustrate the results obtained in this paper.

\subsection{Experimental data}\label{Subsec: Numerics}

We will provide three different sets of numerical results on open bounded domains $\Om \subset \R^2$ with $s = 1$ and $f = 0$ (which means that $u_0 = 0$). In each setting discussed below, we consider sequences of uniformly refined sub-divisions $\{\mathcal T_k\}_{k \in \mathbb N_*}$. Therefore, rather than indexing the finite element quantities by the associated meshsize $h \approx 2^{-(k+1)}$, we use the iteration number $k$. Furthermore, the numerical simulations presented below use the deal.ii library \cite{bangerth2007deal}, which is mainly designed for sub-divisions made of quadrilaterals. Consequently, for a sub-division $\mathcal T_k$ of $\Omega$ made of (possibly curved) quadrilaterals, the finite element space is 
\begin{equation}\label{Eq: specificVH}
    \mathbb{V}_k \ := \ \{v \in C^0(\overline{\Om}) \ | \ v|_T \circ F_T \in \mathbb{Q}_\zeta, \ \fa T \in \mathcal{T}_k\},
\end{equation}
for some $\zeta \in \N^+$, where $\mathbb{Q}_\zeta$ denotes the space of continuous piecewise degree-$\zeta$ polynomials in each coordinate directions and $F_T: [0,1]^2 \rightarrow  T$ is the mapping from the reference element to the physical element $T\in \mathcal T_k$.
Although the theory presented above does not strictly apply to quadrilaterals, we expect it to extend in this setting as well. Throughout this section $\Om_d = \{x \in \Om \ | \ \dist(x, \Gamma) > d\}$, for a distance $d$ that is prescribed by the configuration of measurements, the minimal distance from a measurement point to $\Gamma$.

Several general comments are in order. First, we focus on reporting errors for the Riesz representers since our central theoretical results are the quantitative error estimates of Theorems \ref{Thm: theorateinf}, 
\ref{Thm: TheorateH^1}, and \ref{Thm: RieszGlobalLInfinity}, though we also include numerical validation of Corollaries \ref{Cor: quantEstPointwiseRec} and \ref{Cor: QuantConvLInf} for a square domain. Second, since analytical formulas for Riesz representers $\phi$ cannot be found, we report $\|\phi_K - \phi_k\|$, where $K > k$ and treat $\phi_K$ as a numerical surrogate for $\phi$ when $K$ is large. 


\subsection{Numerical results for Riesz representers}\label{Subsec: riesz}

\subsubsection{Interior measurement points on domains of varying smoothness}

We first let $\Om = B(0, 1)$ be the open unit disc and let $\mathcal{T}_k$, $k=1,...,9$, be a sequence of transfinite meshes exactly subdividing $\Om$. The initial sub-division $\mathcal T_1$ consists of $4$ uniform squares whose boundary is curved to match $\Gamma$. Subsequent sub-divisions are obtained by uniform refinements, always mapping the boundary edges to $\Gamma$ and using a transfinite interpolation to determine the location of the interior vertices. The final sub-division $\mathcal T_9$ is made of $1310720$ cells. The circle, having a smooth boundary, allows us to demonstrate optimal convergence rates (up to logarithmic factors) of $O(h^{\zeta+1})$ in $L_\infty$ and $O(h^\zeta)$ in $H^1$ when using polynomial of degree 1 or 2 in each coordinate direction, thus showcasing the flexibility of Proposition \ref{Prop: SmootherCase}. Overrefinement data for the numerical approximation of the Riesz representer is shown for both cases in Figure \ref{Fig: CircleStandard}. 

\begin{figure}
\centering
\includegraphics[scale = 0.28]{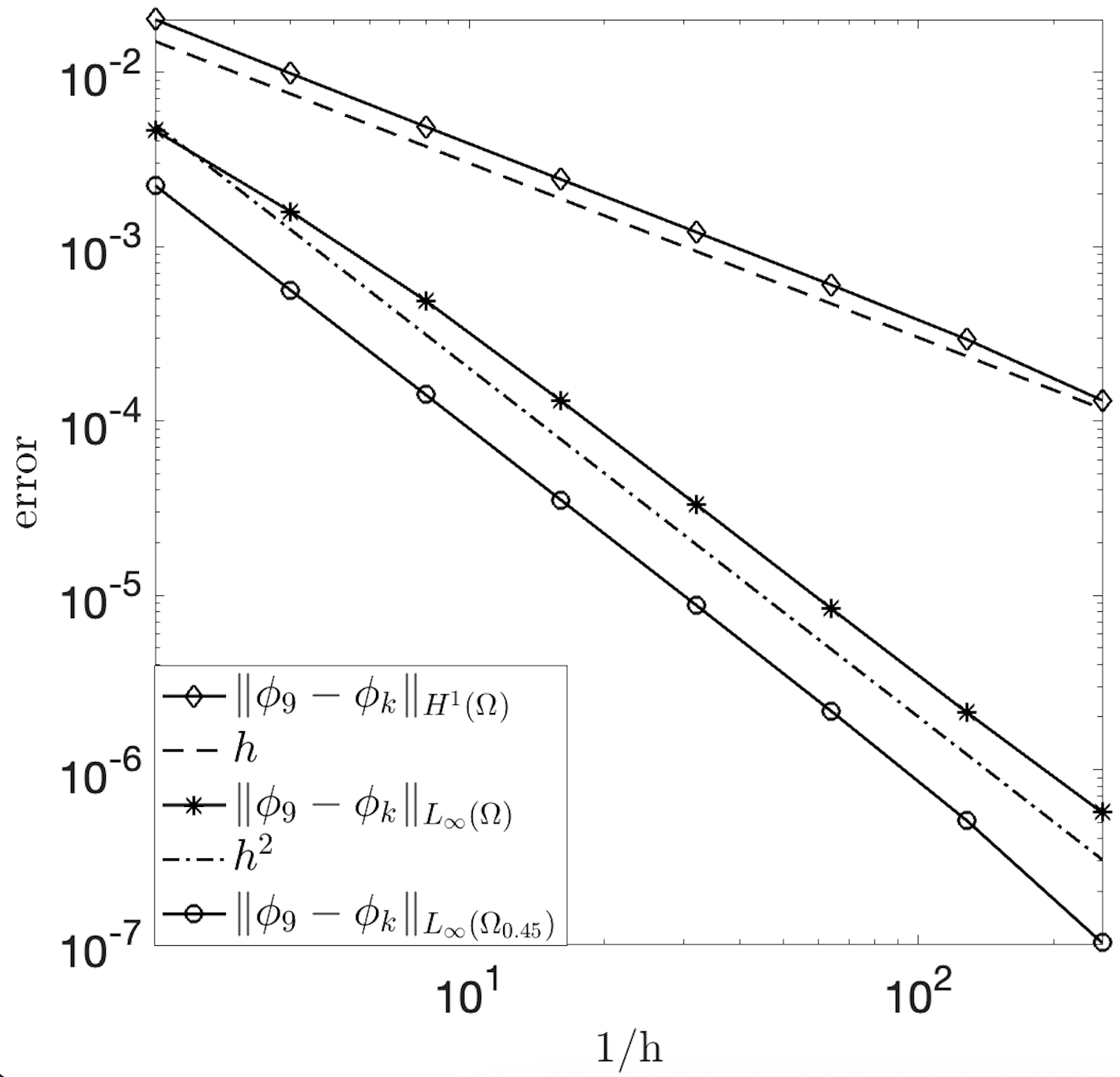}
\includegraphics[scale = 0.28]{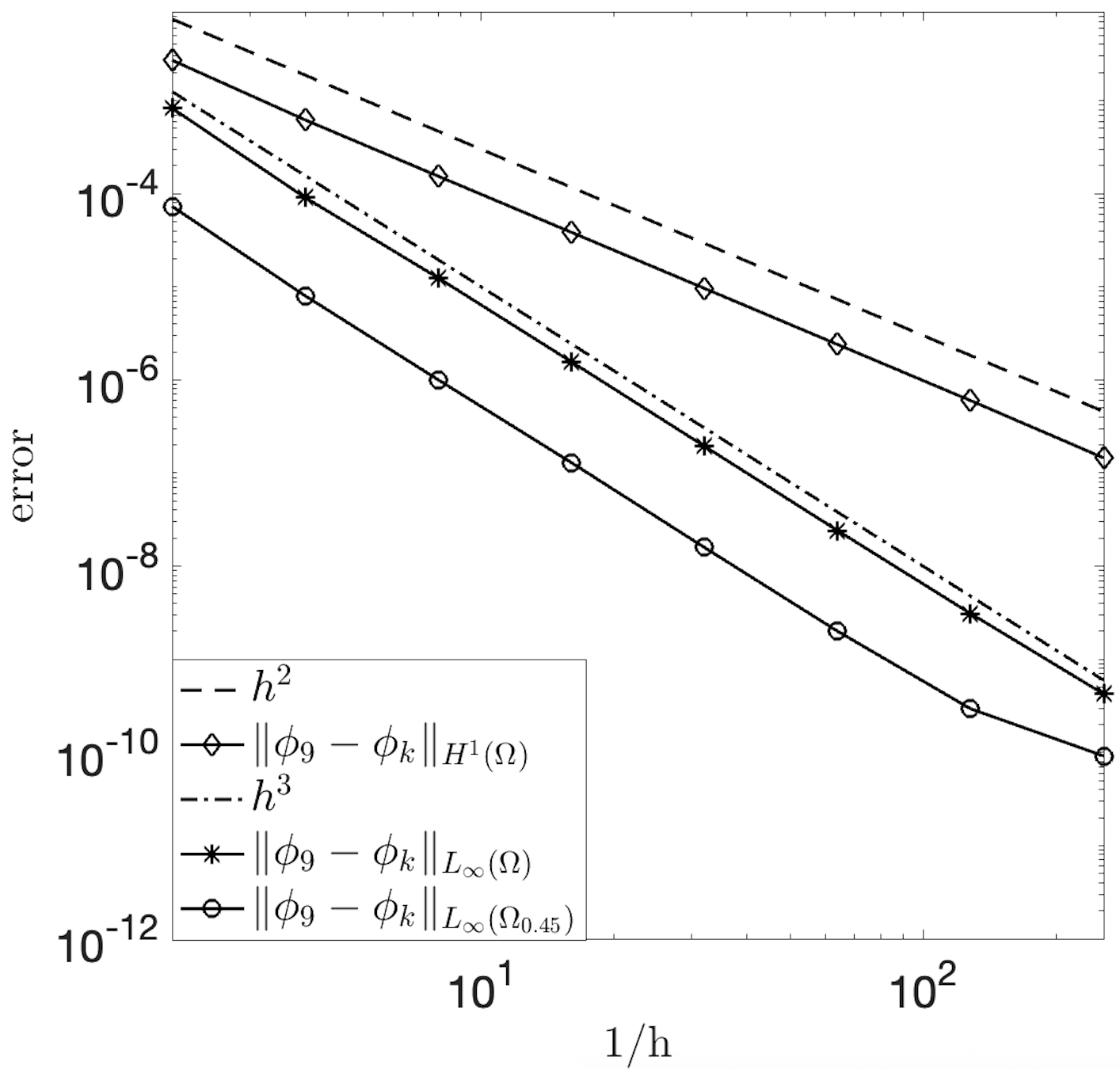}
\caption{Overrefinement errors $\| \phi_9 - \phi_k\|_{X}$ using polynomial degree $\zeta=1$ (left) and $\zeta=2$ (right) for the Riesz Representer $\phi \approx \phi_9$ associated with the measurement location $(0.5, 0.5)$ on the disc $\Om = B(0, 1)$.}
\label{Fig: CircleStandard}
\end{figure}

Next we perform a similar experiments but on a less smooth domain $\Om = (0, 1)^2$ with the same sub-divisions as before except that no curved elements are necessary. 
The finer sub-division $\mathcal T_{10}$ has  $16777216$ square cells.
 As depicted in Figure~\ref{Fig: SquareP1P2}, optimal order is achieved in all norms when  $\zeta=1$, whereas the order of convergence in $L_\infty(\Omega)$ is only 2 when $\zeta=2$, indicating limited regularity of the Riesz representer due to the limited smoothness of $\Gamma$. 

 To explore further the influence of the domain smoothness on the rates of convergence, 
  we also consider a $L$-shaped domain $\Om = (0, 1)^2\setminus((0, 1) \times (-1, 0))$. The finest sub-division $\mathcal{T}_9$ has $3145728$ square cells. On such a domain, we expect a singularity of the form $\rho^{2/3}$, where $\rho$ is the distance to the reentrant corner. This limit the convergence rates to $h^{2/3}$ in $L_\infty(\Omega)$ and $H^1(\Omega)$, and to $h^{4/3}$ in $L_\infty(\Omega_d)$. Although the rates reported in Figure~\ref{Fig: Lshape} are slightly better, they follow the expected trend. For this experiment, the initial sub-division consists of $3$ uniform square cells and the subsequent sub-divisions are obtained by uniform refinement. 
  

\begin{figure}
\centering
\includegraphics[scale = 0.28]{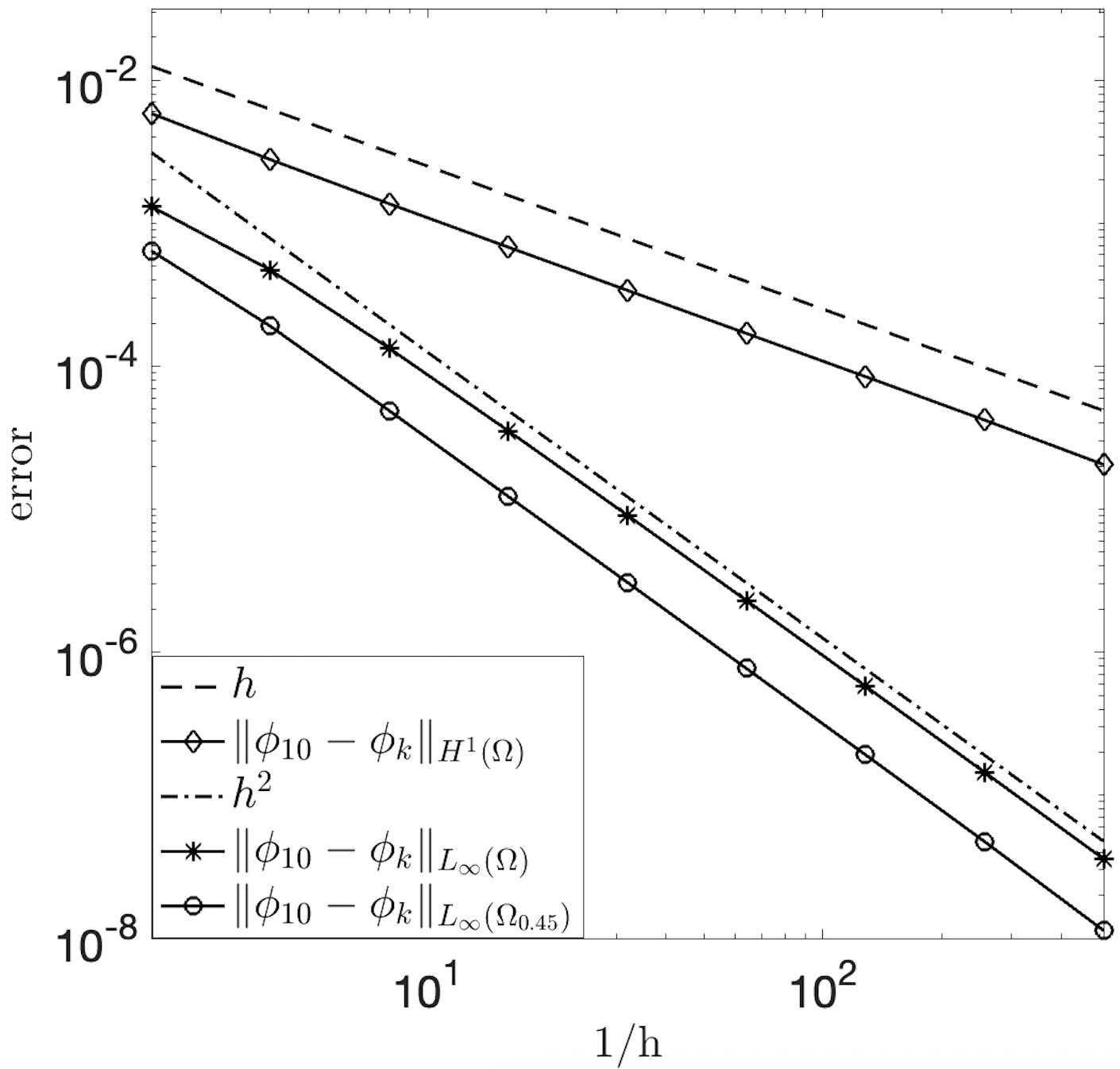}
\includegraphics[scale = 0.28]{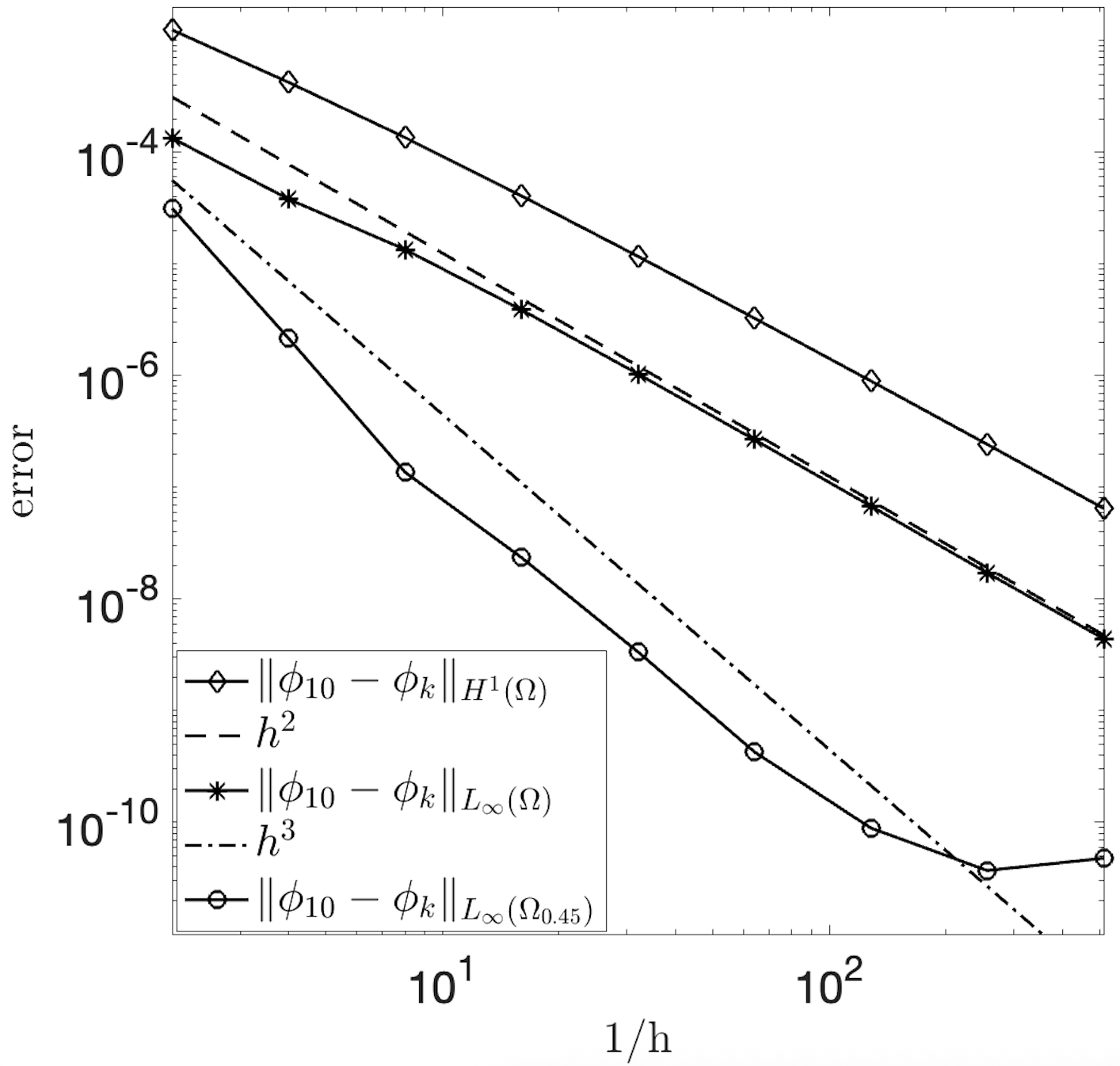}
\caption{Overrefinement errors $\| \phi_{10} - \phi_k\|_{X}$ using polynomial degree $\zeta=1$ (left) and $\zeta=2$ (right) for the Riesz Representer $\phi \approx \phi_{10}$ associated with the measurement location $(0.5, 0.5)$ on the disc $\Om = B(0, 1)$.}
\label{Fig: SquareP1P2}
\end{figure}

\begin{figure}
\centering
\includegraphics[scale = 0.28]{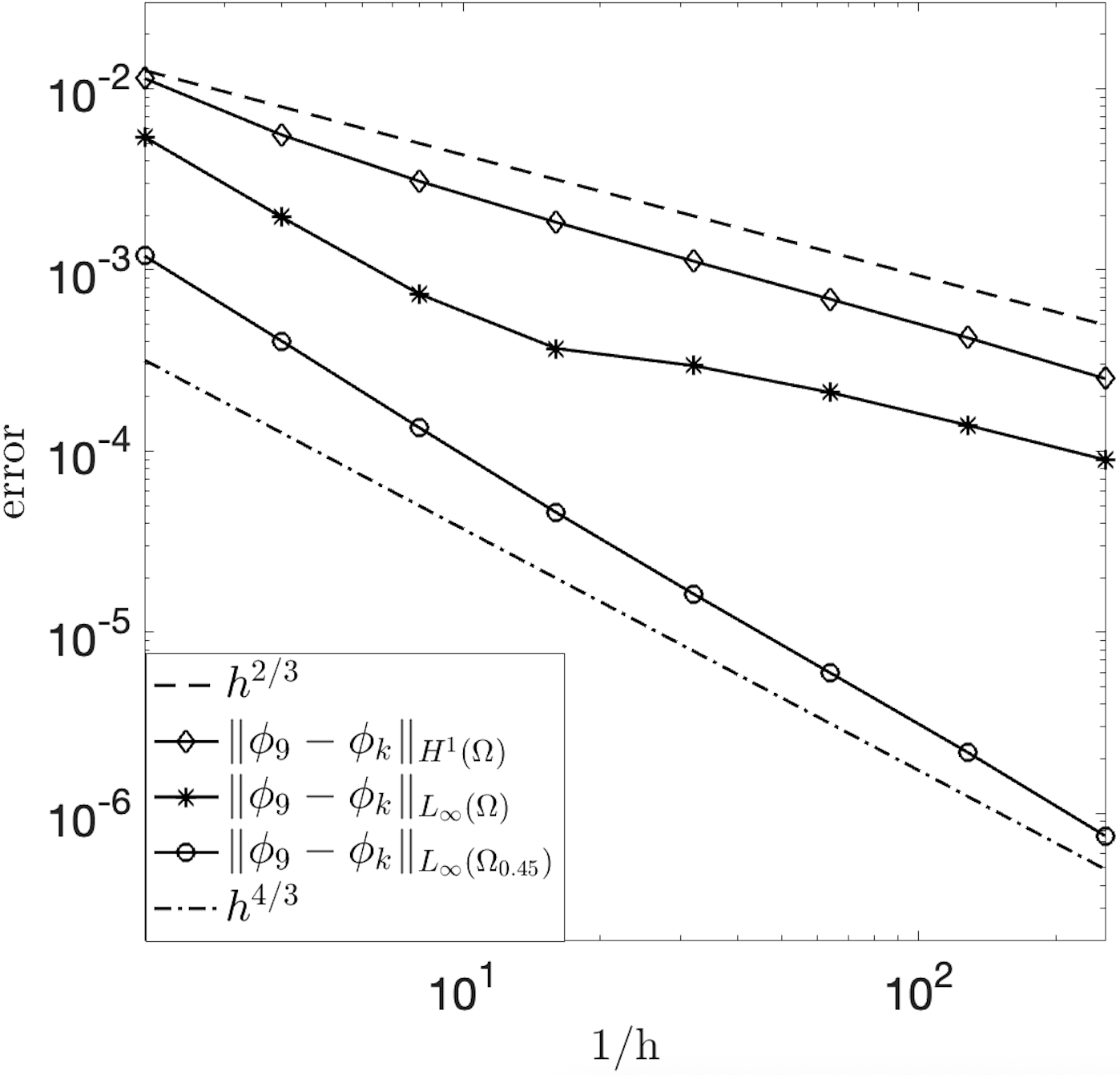}
\caption{Overrefinement errors $\| \phi_{9}-\phi_k\|_X$ using polynomials of degree $\zeta=1$ for the Riesz representer  $\phi \approx \phi_9$ associated with the measurement location $(-0.47, 0.47)$ on the 2D L-shaped domain $\Om = (0, 1)^2\setminus ((0, 1) \times (-1, 0))$.}
\label{Fig: Lshape}
\end{figure}

\subsubsection{Measurement points near or on the boundary}

We first show how the Riesz representer error deteriorates as $d \rightarrow 0^+$ (with fixed mesh size) using polynomials of degree $1$ and $2$ on the unit disc. The sub-divisions of the unit disc are obtained as in the previous section.  In Figure \ref{Fig: d->0} we depict the error $\| \phi_9 - \phi_8\|_X$ when the measurement locations are successively chosen to be $(0.9, 0)$, $(0.95, 0)$, $(0.975, 0)$, and $(0.9875, 0)$.  While clear deterioration rates appear in both cases, the $H^1(\Om)$ and $L^{\infty}(\Om)$ rates for $\mathbb{Q}_1$ elements  from \eqref{eq: locsummary} and \eqref{Prop: SmootherCase} ($d^{-1}$ and $d^{-2}$, respectively), along with the $L_{\infty}(\Om)$ rate for $\mathbb{V}_2$ from Proposition \ref{Prop: SmootherCase} ($d^{-3}$), appear to be suboptimal.  Proof of sharp rates of deterioration with respect to $d$ would involve highly technical and domain-dependent arguments and is not pursued further.   On the other hand, as predicted by Proposition \ref{Prop: SmootherCase}, we do observe a higher rate of deterioration in the maximum norm when $\zeta=2$ versus when $\zeta=1$ (roughly $O(d^{-2})$ versus $O(d^{-4/3})$, respectively).  

\begin{figure}
\centering
\includegraphics[scale = 0.28]{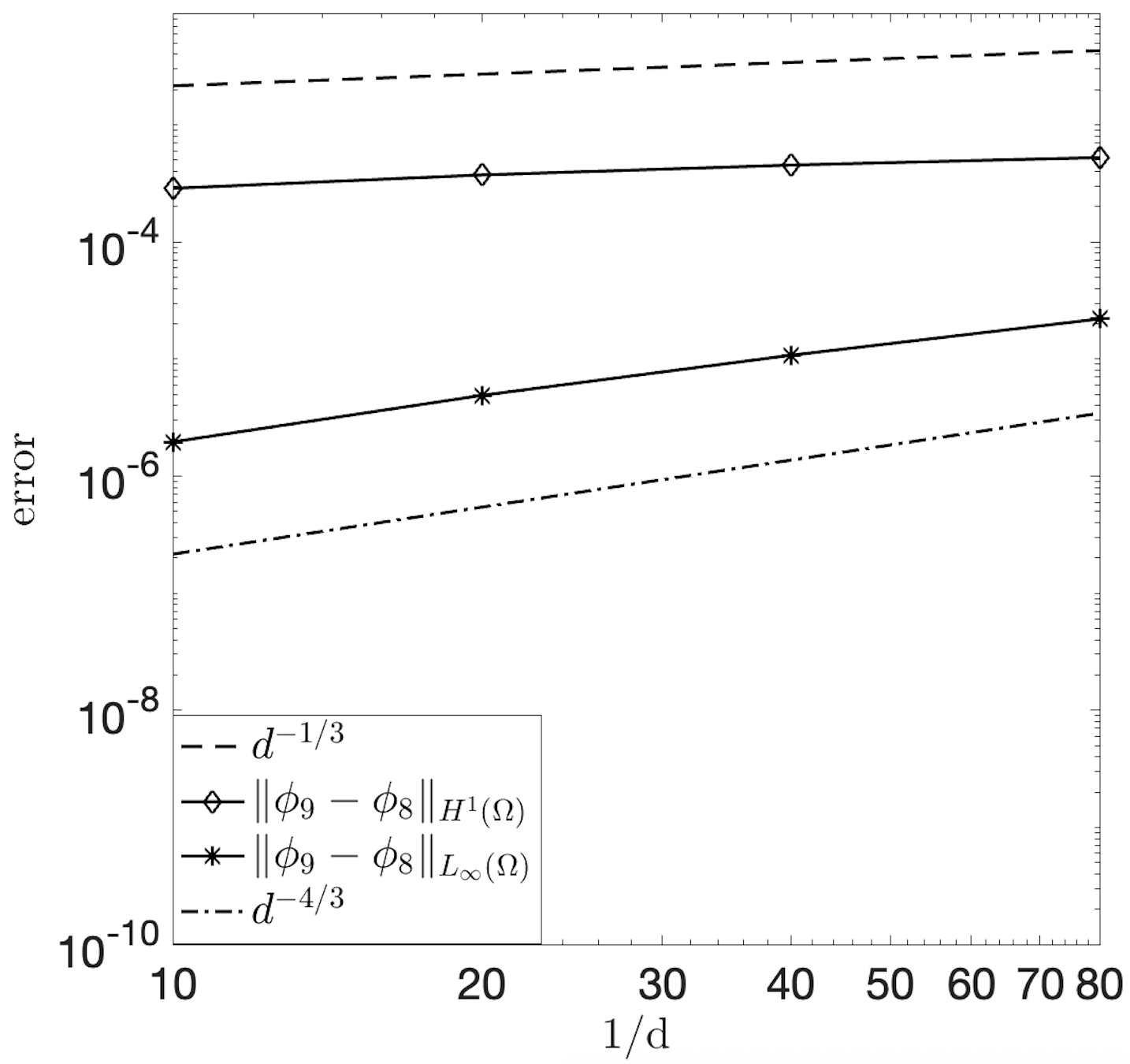}
\includegraphics[scale = 0.28]{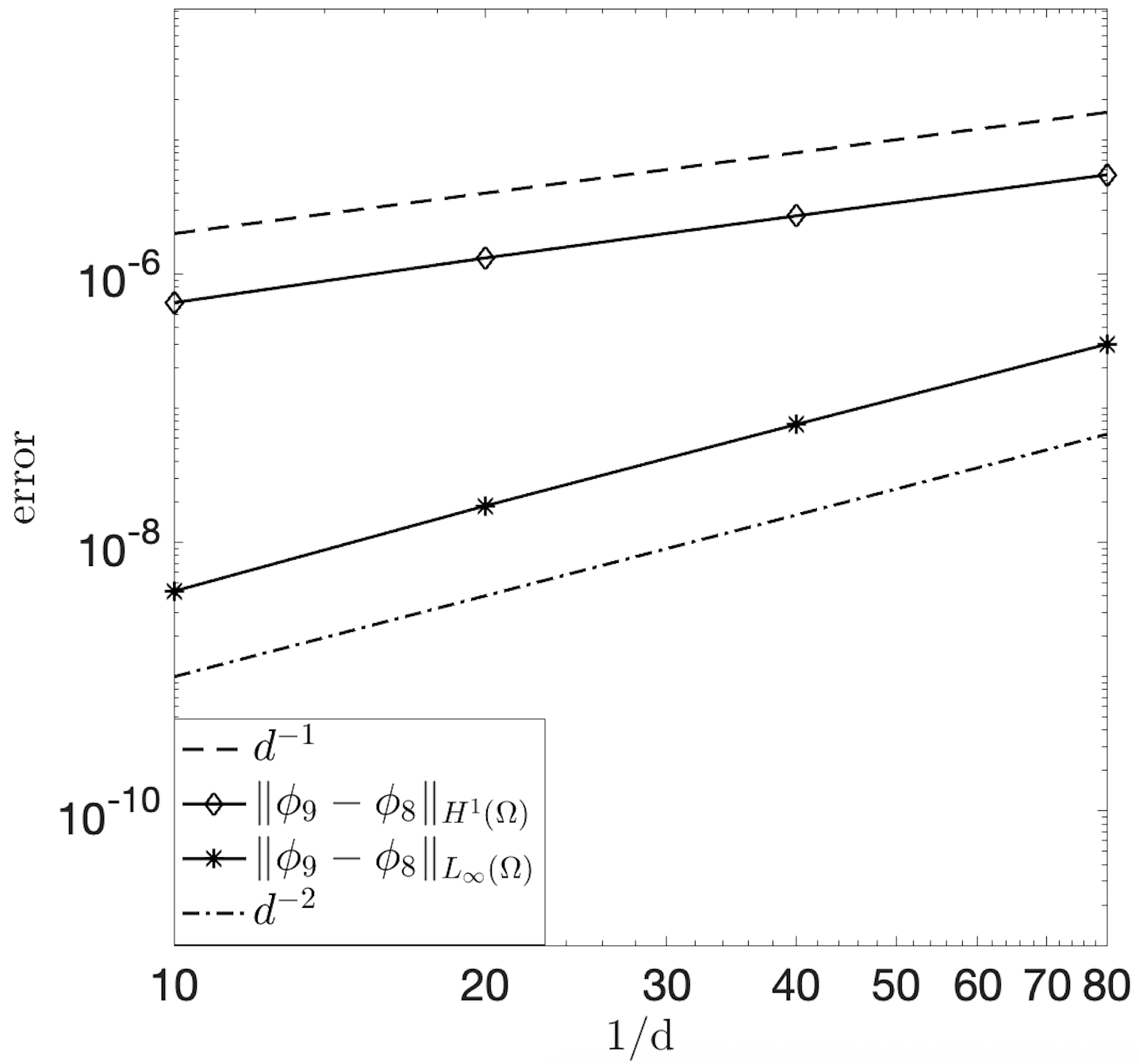}
\caption{Overrefinement errors $\|\phi_9 - \phi_8\|_{X}$ using polynomials of degree $\zeta=1$ (left) and $\zeta=2$ (right) with  $X=H^1(\Om)$ and $X=L_{\infty}(\Om)$ and $\Om = B(0, 1)$ as the measurement location approaches the boundary ($d \rightarrow 0^+$).}
\label{Fig: d->0}
\end{figure}

  We also consider a case where the measurement location is on the boundary of the square. We observe (see Figure \ref{Fig: SqOnBoundary}) that the errors are of order $O(h)$ when measured  in $L_\infty(\Omega)$ and in $H^1(\Omega)$, even when using $\mathbb{Q}_2$ elements. Instead, we observe an order of $O(h^{\zeta+1})$ ($\zeta=1,2$) when measured in $L_\infty(\Omega_d)$. Note that when $x_i \in \Gamma$, the boundary data $\psi_i$ for the Riesz representer is a Green's function on $\Gamma$.  The Green's function in one space dimension is merely Lipschitz (in fact piecewise linear for simple two-point boundary value problems). The reduced convergence rate in $L_\infty(\Omega)$, and in $H^1(\Omega)$ when $\zeta=2$, reflects this reduced regularity. 
 
 \begin{figure}
 \centering
 \includegraphics[scale = 0.28]{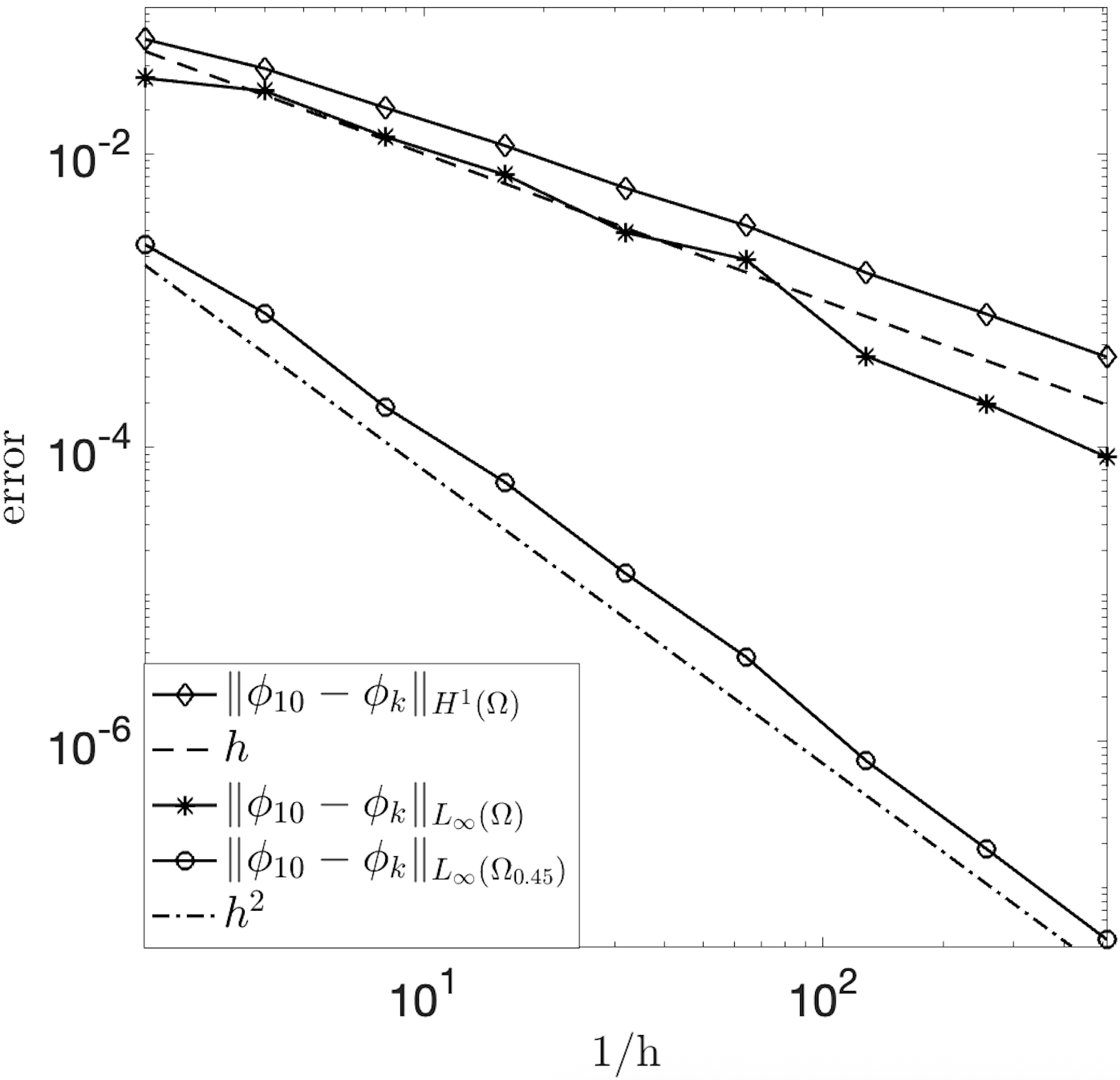}
 \includegraphics[scale = 0.28]{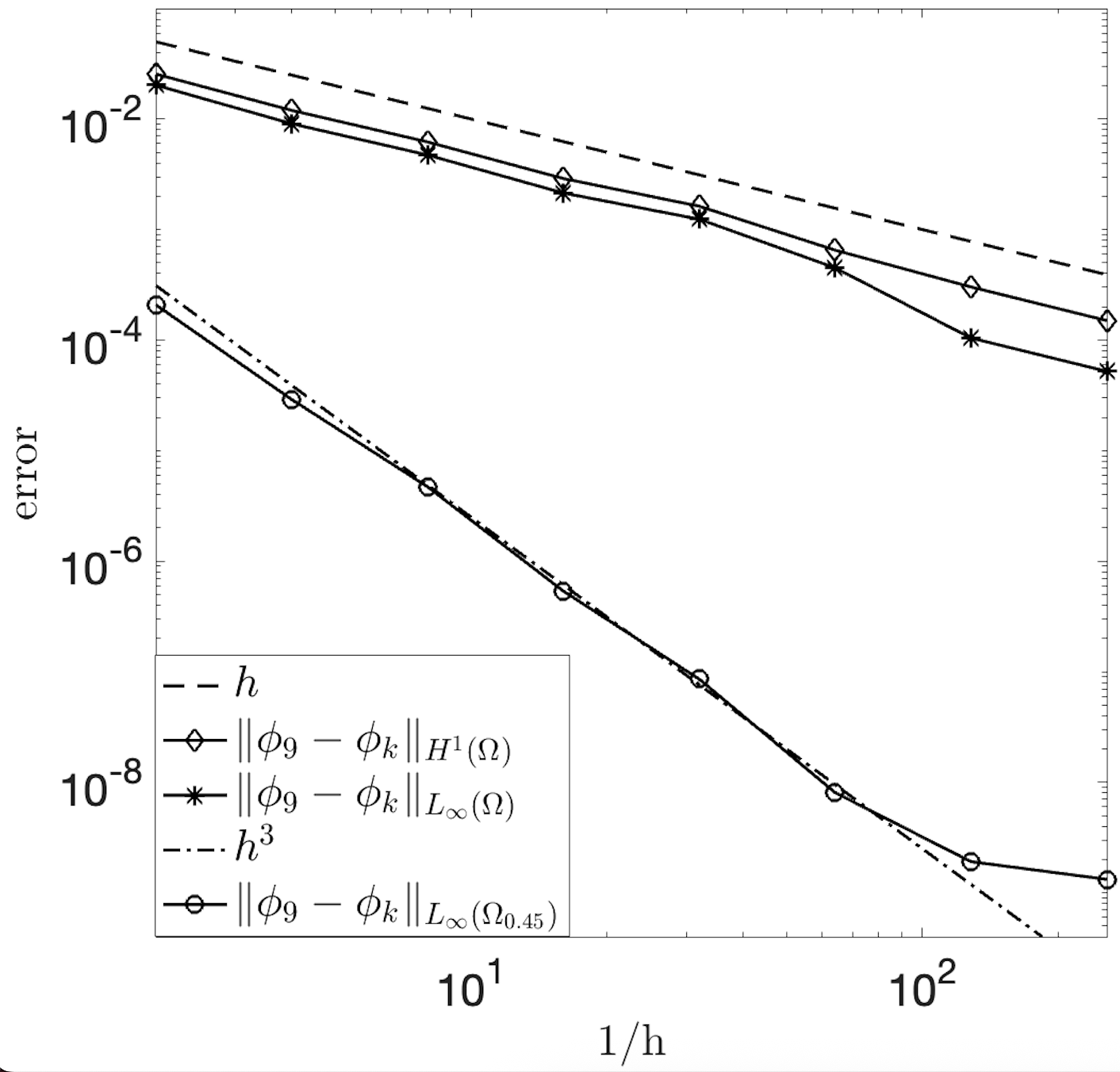}
 \caption{Overrefinement errors $\|\phi_{10}-\phi_k\|$ using polynomials of degree $\zeta=1$ (left) and $\zeta=2$ (right) for the Riesz representer $\phi \approx \phi_{10}$ associated with the measurement location  $(0, \sqrt{2}/2)\in \Gamma$.} \label{Fig: SqOnBoundary}
 \end{figure}



\subsection{Recovery approximation}\label{Subsec: recApprox}
To illustrate the recovery error, we choose the harmonic function
\begin{equation}\label{Ex: Harmonic}
    u_{\mathcal{H}}(x, y) \ := \ e^x\cos(y), \ (x, y) \in \Om=(0,1)^2
\end{equation}
as the function from which the measurements are drawn.  Data for the difference between $u_{\mathcal{H}}$ and the numerical approximation of the minimal norm interpolant are provided for numbers of measurement points $m = 4,16,36,64$.  The measurement points were placed in an \textit{interior box formation} consisting of measurement locations $(0.9, (i + 1)/17)$, $(0.1, (i + 1)/17)$, $((i + 1)/17, 0.1)$, and $((i + 1)/17, 0.9)$ ($i = 0$ when $m = 4$; $0 \le i \le 4$ when $m=16$ measurements; $0 \leq i \leq 8$ when $m = 36$ measurements; and $0 \leq i \leq 15$ when $m = 64$ measurements).  Computations were also performed using a \textit{grid formation} with measurement locations uniformly distributed in the interior of the unit square as in \cite{binev2024solving, bonito2024approximating}. The results  were very similar using this formation, so we do not report them here.  Also, we measured the recovery errors in the $H^1(\Omega)$, $L_\infty(\Omega)$, and $L_\infty(\Omega_{0.45})$ norms for a sequence of uniform refinements of $\mathcal T_0$ made of 4 uniform squares.  The errors in the $H^1$ norm was very similar to those reported below for the $L_\infty(\Omega)$ norm so we only report results for $L_\infty(\Omega)$ and $L_\infty(\Omega_{0.45})$ in Figure \ref{Fig: LinfPtwsBox} below.

\begin{figure}
\centering

\includegraphics[scale = 0.28]{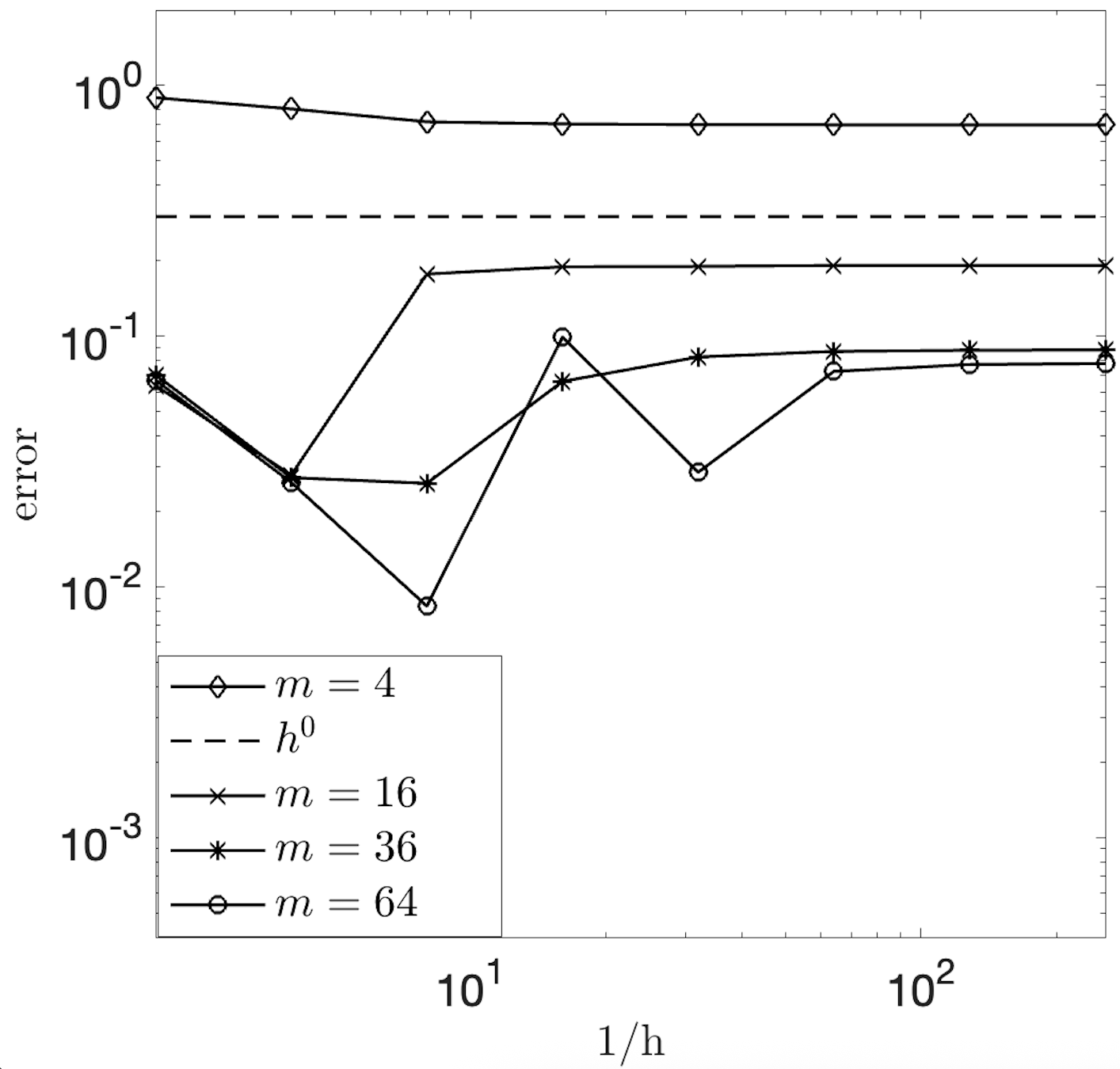}
\includegraphics[scale = 0.28]{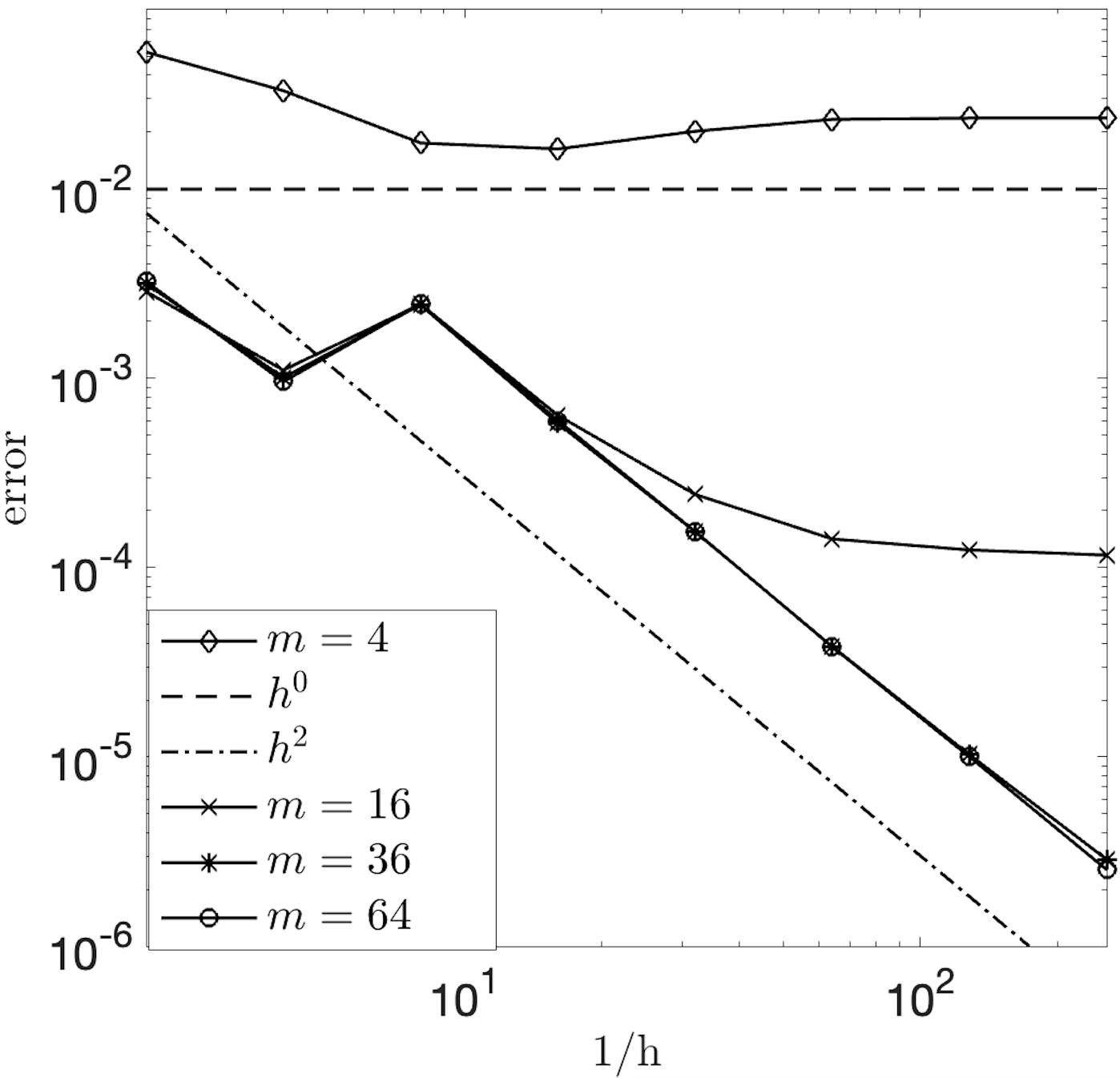}
\caption{Errors $\| u_{\mathcal H} - u_k^*\|_X$  on the square $\Om = (0, 1)^2$ using $m=4,16,36,64$ measurements in the interior box formation  for $X=L_{\infty}(\Om)$ (left) or $X=L_{\infty}(\Om_{0.45})$  (right).}
\label{Fig: LinfPtwsBox}
\end{figure}


 Note that the finite element recovery $u_h^*$ is proven to approximate the minimal norm interpolant associated with $\mathcal K_{\om}$. The latter may not be $u_{\mathcal H}$ but is at a distance at most $R(\mathcal K_{\om})_X$ to $u_{\mathcal H}$.  Nevertheless, when measuring the error in global norms ($H^1(\Omega)$ and $L_\infty(\Omega)$), the finite element recovery appears to approximate $u_{\mathcal H}$ well even on relatively coarse meshes, as reflected in the fact that the observed error changes little as the mesh is refined.  However, the observed recovery error does decrease as the number of measurement points increases, as made clear in Figure \ref{Fig: LinfPtwsBox}, which shows the recovery error in the $L_{\infty}(\Om)$ and $L_{\infty}(\Om_{0.45})$ norms. 
 

When measuring the error in $L_\infty(\Omega_d)$, it appears that the corresponding Chebychev radius is around $2 \times 10^{-2}$ when employing $4$ measurement points, about $10^{-4}$ when employing $16$ measurement points, and unclear when using 36 and 64 measurement points as the error continues to decrease with optimal order $h^2$ for the range of $h$ values tested.  This indicates that the Chebyshev radius may be much smaller when measuring interior errors as compared with global errors.  

We finally make a note on algorithmic implementation.  The near-optimal recovery theorem (Theorem~\ref{Thm: NORB}) assumes that the matrix $G$ is invertible. Depending on the location of the measurements, $G$ might be singular or nearly singular. Moreover, it is expected that the condition number of $G$ increases as $m$ increases. This leads to numerical difficulties as reported in \cite{binev2024solving,bonito2024approximating}. In the following, we follow the strategy proposed in \cite{binev2024solving}, which consists of applying a thresholding that eliminates the small eigenvalues of $G$. We refer to \cite{adcock2019frames} for a review and analysis of this regularization strategy. Empirically, we observe that the amount of thresholding necessary to obtain robust recovery can be decreased as the finite element resolution increases but in order not to obfuscate the purpose of the numerical simulation provided below, we fix the thresholding to $10^{-14}$. This means that the singular values of $G$ that are smaller than $10^{-14}\lambda_{\textrm{max}}$ are discarded when computing the Moore-Penrose pseudo inverse of $G$. Here $\lambda_{\textrm{max}}$ is the largest singular value of $G$.

\bibliographystyle{amsplain} 
\bibliography{learning} 

\end{document}